\documentclass[a4paper,11pt]{article}

\usepackage[english]{babel}
\usepackage[english]{translator}
\usepackage[T1]{fontenc}
\usepackage[utf8]{inputenc}

\usepackage{amsmath}
\usepackage{amssymb}
\usepackage{amsfonts}
\usepackage{amstext}
\usepackage{amsthm}
\usepackage{mathtools}
\usepackage{bbm}

\usepackage{graphicx}
\usepackage{xcolor}
\usepackage{psfrag}
\usepackage{pgfplots}
\usepackage{adjustbox}
\usepackage{subfig}
\usepackage{colortbl}

\usepackage{titlesec}
\usepackage{listings}
\usepackage{enumerate}
\usepackage{enumitem}
\usepackage{todonotes}
\definecolor{asparagus}{rgb}{0.53, 0.66, 0.42}
\definecolor{ballblue}{rgb}{0.13, 0.67, 0.8}
\definecolor{cadmiumgreen}{rgb}{0.0, 0.42, 0.24}
\definecolor{cobalt}{rgb}{0.0, 0.28, 0.67}
\definecolor{darklavender}{rgb}{0.45, 0.31, 0.59}
\definecolor{green(pigment)}{rgb}{0.0, 0.65, 0.31}

\definecolor{myBlue3}{RGB}{60,124,155}

\definecolor{myForestGreen}{RGB}{34,139,34}

\usepackage{hyperref}

\usepackage{geometry}
\usepackage{tabularx}
\makeatletter
\def\hlinewd#1{%
\noalign{\ifnum0=`}\fi\hrule \@height #1 %
\futurelet\reserved@a\@xhline}
\makeatother

\hypersetup{
  pdffitwindow=false,
  pdfhighlight=/O,
  pdfnewwindow,
  colorlinks=true,
  citecolor=red,             
  linkcolor=blue,            
  menucolor=blue,            
  urlcolor=blue,             
  pdfpagemode=UseOutlines,
  bookmarksnumbered=true,
  linktocpage,
  pdfkeywords={},
  pdfcreator={pdflatex},
  pdfproducer={LaTeX with hyperref}
}

\newcommand{\C}{\mathbb{C}}               
\newcommand{\R}{\mathbb{R}}               

\renewcommand{\Re}{\mathrm{Re}\,}          
\renewcommand{\Im}{\mathrm{Im}\,}          

\renewcommand{\i}{\mathrm{i}} 
\newcommand{\ci}{\mathrm{i}} 
\renewcommand{\d}{\,\mathrm{d}}

\newcommand{\abeta}{a_{\beta}}
\newcommand{\Ainvbeta}{\mathcal{A}^{-1}_\beta}
\newcommand{\Abeta}{\mathcal{A}_\beta}

\newcommand\UN{\textup{N}}
\DeclareMathOperator{\diam}{diam}

\newcommand{\orthiu}{(\ci u)^{\perp}}
\newcommand{\LOD}{{\text{\normalfont\tiny LOD}}}

\newcommand{\Vlod}{V_{h}^{\LOD}}
\newcommand{\Vlodell}{V_{h,\ell}^{\LOD}}

\renewcommand{\div}{\mathrm{div} \,}		

\newcommand{\Wbot}{W^{\bot_{a_\beta}}}

\newcommand{\bfA}{\boldsymbol{A}}

\newcommand{\Ccoe}{\rho(\kappa)} 

\newcommand{\RitzLOD}{ R_{h}^{\LOD}}
\newcommand{\RitzLODperp}{ R_{h}^{\perp,\LOD}}

\newcommand{\RitzLODell}{ R_{h,\ell}^{\LOD}}

\DeclareMathOperator{\identity}{id}
\DeclareMathOperator{\Int}{int}
\DeclareMathOperator{\ol}{ol}

\newcommand{\fine}{\mbox{\tiny\normalfont fine}}

\newcommand{\id}{\identity}

\newcommand{\quotes}[1]{``#1''}

\newcommand{\FEM}{{\text{\normalfont\tiny FEM}}}

\theoremstyle{definition}
\newtheorem{definition}{Definition}[section]

\newtheorem{remark}[definition]{Remark}

\theoremstyle{plain}
\newtheorem{theorem}[definition]{Theorem}
\newtheorem{lemma}[definition]{Lemma}

\newtheorem{proposition}[definition]{Proposition}

\allowdisplaybreaks

\begin{document}

\begin{center}
{\Large 
Vortex-capturing multiscale spaces for the Ginzburg--Landau equation\renewcommand{\thefootnote}{\fnsymbol{footnote}}\setcounter{footnote}{0}
 \hspace{-3pt}\footnote{
 The authors acknowledge the support by the German Research Foundation (DFG). M. Blum and P. Henning received funding through the project grant HE 2464/7-1 and C. D\"oding through Germany's Excellence Strategy, project number EXC-2047/1 -- 390685813.}\\[2em]}
\end{center}

\begin{center}
{\large Maria Blum\footnote[1]{Department of Mathematics, Ruhr University Bochum, DE-44801 Bochum, Germany, \\ e-mail: \textcolor{blue}{maria.blum@rub.de}.}},
{\large Christian D\"oding\footnote[2]{Institute for Numerical Simulation, University Bonn, DE-53115 Bonn, Germany, \\ e-mail: \textcolor{blue}{doeding@ins.uni-bonn.de}.}}, 
{\large Patrick Henning\footnote[3]{Department of Mathematics, Ruhr University Bochum, DE-44801 Bochum, Germany, \\ e-mail: \textcolor{blue}{patrick.henning@rub.de}.}}\\[2em]
\end{center}

\noindent
\begin{center}
\begin{minipage}{0.8\textwidth}
  {\small
    \textbf{Abstract.} This paper considers minimizers of the Ginzburg--Landau energy functional in particular multiscale spaces that are based on finite elements. The spaces are constructed by localized orthogonal decomposition techniques and their usage for solving the Ginzburg--Landau equation was first suggested in [D\"orich, Henning, SINUM 2024]. In this work we further explore their approximation properties and give an analytical explanation for why vortex structures of energy minimizers can be captured more accurately in these spaces. We quantify the necessary mesh resolution in terms of the Ginzburg--Landau parameter $\kappa$ and a stabilization parameter $\beta \ge 0$ that is used in the construction of the multiscale spaces. Furthermore, we analyze how $\kappa$ affects the necessary locality of the multiscale basis functions and we prove that the choice $\beta=0$ yields typically the highest accuracy. Our findings are supported by numerical experiments.}
\end{minipage}
\end{center}


\section{Introduction}
Materials that conduct electricity with no electrical resistance, so called superconductors, can be modeled by the Ginzburg--Landau equation (GLE) \cite{DuGuPe,Landau:486430}. In its simplest form, the equation seek the order parameter $u : \Omega \rightarrow \C$ such that
\begin{align} 
\label{stationarGLE}	 (\tfrac{\i}{\kappa} \nabla  + \bfA )^{2} \hspace{1pt}u + (|u|^2-1) u &= 0  \qquad \mbox{in } \Omega,\\
\nonumber	 (\tfrac{\i}{\kappa} \nabla u + \bfA \, u) \cdot \mathbf{n} &=0 \qquad \mbox{on } \partial \Omega.
\end{align}
Here, $\Omega \subset \R^d$ (for $d=2,3$) describes a domain that is occupied by the superconducting material, $\bfA : \Omega \rightarrow \R^d$ is a magnetic vector potential and $\kappa$ is a real material parameter which describes the properties of the considered (type-II) superconductor, cf. \cite{PhysRev.140.A1169}. Formally, $\kappa$ is defined as the ratio between the magnetic-field penetration depth and the coherence length. For superconductors with a high critical temperature, $\kappa$ typically has a large value. The Ginzburg--Landau equation \eqref{stationarGLE} is obtained from the condition $E^{\prime}(u)=0$ that is fulfilled by minimizers of the total energy $E$ of the system (see \eqref{eq_GLenergy} below). For the considered application, the physical quantity of interest is the density of superconducting charge carriers, so-called Cooper pairs, which is obtained from the order parameter as $|u|$. It can be shown that it naturally holds $0 \le |u| \le 1$ (cf. \cite{DuGuPe}) where $|u(x)|=0$ represents the {\it normal state} without any superconducting electron pairs (locally in $x \in \Omega$) and $|u(x)|=1$ represents a perfectly {\it superconducting state}. Of particular physical interest are superconductors with mixed states where normal and superconducting phases coexist in a lattice of quantized vortices, the so-called Abrikosov lattice \cite{Abr04}. The center of the vortices are in the normal state (i.e. it holds $|u(x)|=0$) and magnetic fields can no longer be expelled here. Mathematically, the appearance of vortices in $|u|$ is crucially triggered by the material parameter $\kappa$, cf. \cite{Aftalion99,Ser99,SaS07,SeS10,SaS12}. If the magnetic potential $\bfA$ is fixed, the number of vortices grows with $\kappa$. At the same time, the diameter of the individual vortices shrinks and they become more \quotes{point-like}. Consequently, the computational complexity for solving the Ginzburg--Landau equation increases significantly in the so-called high-$\kappa$ regime. For that reason it is important to understand and quantify the accuracy of numerical approximations depending on the size of $\kappa$ and to identify approaches that are well-suited for capturing vortex pattern.

For the time-dependent Ginzburg--Landau equation, which models the dynamics of superconductors, there has been intense work on its numerical treatment and we exemplary refer to \cite{Chen97,CheD01,Du94b,Du94,Du97,DuGray96,DuanZhang22,GJX19,GaoS18,Li17,LiZ15,LiZ17} and the references therein.  
On the contrary, error estimates and convergence results for the stationary Ginzburg--Landau equation \eqref{stationarGLE} are rare in the literature. First approximation results for finite element methods were obtained by Du, Gunzburger and Peterson \cite{DuGuPe,DuGP93}, for covolume methods by Du, Nicolaides and Wu \cite{DuNicolaidesWu98} and for finite volume methods by Du and Ju \cite{QuJu05}. The precise influence of $\kappa$ on the approximation results was however not traced in these works. The first $\kappa$-explicit error estimates for finite element (FE) approximations were only obtained recently in \cite{DoeHe23}. One of the major findings was that for linear Lagrange elements, the $H^1$-error behaves asymptotically (in a suitably scaled norm) like $\mathcal{O}(h \kappa)$, where $h$ denotes the mesh size. This suggests the resolution condition $h \lesssim \kappa^{-1}$ for reliable approximations. This condition does not change with higher order FE spaces. However, the error estimates in \cite{DoeHe23} also indicated the presence of a preasymptotic convergence regime, which required $h$ to be significantly smaller than $\kappa^{-1}$ in order to observe the predicted asymptotic rates. In general, using $\mathbb{P}^k$ Lagrange finite elements and assuming sufficient regularity, we can expect the additional constraint $h \lesssim \rho(\kappa)^{-1/k} \kappa^{-1}$, where $\rho(\kappa)$ depends on the smallest eigenvalue of $E^{\prime\prime}(u)$. The precise dependency of $\rho(\kappa)$  on $\kappa$ is unknown, but it seems to grow at least polynomially with $\kappa$. Numerical experiments confirmed the existence of this regime. To reduce the preasymptotic effect, it was suggested in \cite{DoeHe23} to discretize the problem in certain generalized finite element spaces based on Localized Orthogonal Decomposition (LOD). The LOD concept was first introduced in \cite{MaP14} to solve elliptic problems with multiscale coefficients and the corresponding approximation spaces are constructed as localization of $\mathcal{A}^{-1} V_h$, where $\mathcal{A}$ is a suitable linear differentiable operator and $V_h$ a conventional finite element space. Like that, problem-specific information can be encoded in the spaces through $\mathcal{A}$. For further reading on the topic, we refer to the surveys \cite{AHP21Acta,MaPeBook21}. 
The particular construction given in \cite{DoeHe23} for the GLE was based on selecting $\mathcal{A}$ as the linear part of the GLE \eqref{stationarGLE}, that is $( \tfrac{\i}{\kappa} \nabla u + \bfA u , \tfrac{\i}{\kappa} \nabla v + \bfA v)_{L^2(\Omega)}$, but stabilized by an additional contribution of the form $\beta ( u, v )_{L^2(\Omega)}$, where $\beta$ is selected large enough to ensure that the operator becomes elliptic. Even though a preliminary error analysis was given for the arising method, several aspects were left open, in particular, the precise role of $\beta$ and a justification of the method in numerical applications. In fact, even though the analysis in \cite{DoeHe23} required $\beta$ to be large, numerical experiments show that the method performs best for $\beta=0$. Since $\mathcal{A}$ is no longer elliptic in this case, the error analysis has to be crucially revisited. With this, we extend the results of \cite{DoeHe23} with regard to the following aspects: We prove that the LOD space is well-defined for any $\beta \ge 0$; we derive $\kappa$- and $\beta$-explicit error estimates that confirm that $\beta=0$ is the preferable choice; we present a localized approximation of the LOD space and quantify the influence of $\kappa$ on the localization error; and we validate our theoretical findings in corresponding numerical experiments. In particular, we demonstrate that the necessary mesh resolution (relative to the size of $\kappa$) is significantly reduced in the LOD spaces and that vortices are also captured on coarse meshes. A preasymptotic effect, as observed for $\mathbb{P}^1$-FEM, is barely (if at all) visible in LOD spaces. This will be also reflected in the error estimates.\\[0.3em]
{\it Outline.} The paper is structured as follows. The precise analytical setting is presented in Section \ref{section:analytical_setting}. In Section \ref{section:standardFEM} we describe the discretization of the GLE with standard $\mathbb{P}^1$ finite elements and recall corresponding $\kappa$-explicit error estimates. The ideal LOD discretization is established in Section \ref{section:ideal_LOD} together with a comprehensive error analysis. The fully localized approximations are studied in Section \ref{section:localization-LOD-basis} and in Section \ref{section:numerical_experiments} we conclude with numerical experiments to illustrate our theoretical findings.
\section{Analytical setting}
\label{section:analytical_setting}
In this section, we start with describing the precise analytical setting of this paper. 

\subsection{General assumptions and problem formulation}
In the following, we assume for the computational domain that
\begin{enumerate}[label={(A\arabic*)}]
\item \label{A1} $\Omega \subset \mathbb{R}^d$ is a bounded, convex polytope, with $d=2,3$.
\end{enumerate}

We denote by $L^2(\Omega):= L^2(\Omega;\C)$ the space of complex-valued, square integrable functions and consider it as a {\it real} Hilbert space equipped with the {\it real} inner product
$$ 
(v,w)_{L^2(\Omega)}:= \Re \Big( \int_{\Omega} v \, w^* \d x \Big) 
\qquad \mbox{for } v,w \in L^2(\Omega),
$$
where $w^*$ is the complex conjugate of $w$. In a similar way, we define the Sobolev space $H^{1}(\Omega):= H^{1}(\Omega;\mathbb{C})$ to be a real Hilbert space. With this, we consider the Ginzburg--Landau energy functional $E: H^1(\Omega) \rightarrow \R$ defined by
 \begin{equation} 
 E(v) = \frac{1}{2} \int_{\Omega} \left|  \tfrac{\i}{\kappa} \nabla v + \bfA\, v \right|^2 + \frac{1}{2} (|v|^2-1)^2 \, \d x,
 \label{eq_GLenergy}
 \end{equation}
where $\bfA : \Omega \rightarrow \R^d$ denotes a given magnetic potential and $\kappa$ a real material parameter. Compared to the setting in \cite{DoeHe23}, the above energy functional $E$ is scaled with the factor $\kappa^{-2}$ which is more common in the literature. However, both settings are fully equivalent and the results of \cite{DoeHe23} remain valid. For $\bfA$ and $\kappa$ we assume 
\begin{enumerate}[resume,label={(A\arabic*)}]
\item\label{A2} $\bfA \in L^\infty (\Omega,\mathbb{R}^d), \hspace{20pt} \div \bfA = 0 \;\mbox{in}\; \Omega, \hspace{20pt} \bfA \cdot \mathbf{n} = 0 \;\mbox{on}\; \partial \Omega$;
\item\label{A3} $\kappa > 0$. 
\end{enumerate}
It is worth to note that, in practice, the potential $\bfA$ is typically not explicitly given but has to be determined through an additional equation derived from Maxwell's equations, cf. \cite{DuGuPe}.\\[0.5em]
Our goal is now to find the lowest energy states of the system, i.e., minimizers $u \in H^1(\Omega)$ with
\begin{align}
\label{minimizer-energy-def} 
E(u) \,\,= \min_{v \in H^1(\Omega)} E(v).
\end{align}
In view of \eqref{minimizer-energy-def}, it becomes clear why we naturally equipped $H^1(\Omega)$ with a real inner product: Since $E$ maps into $\R$, the differentiability of $E$ is  only ensured in {\it real} Hilbert spaces containing complex-valued functions (cf. \cite{AHP21Jmethod,Begout22} for a general motivation). \\[0.2em]
Of particular interest are minimizers when the material parameter $\kappa$ is large. As described in the introduction, $\kappa$ triggers the appearance of a vortex lattice, where the number of vortices increases with $\kappa$. For that reason, our error analysis will focus on the so-called \quotes{high-$\kappa$ regime} where we trace all $\kappa$-dependencies in our error estimates. For simplicity, we will write $a \lesssim b$ to abbreviate the relation $a \le C\, b$ for a constant $C>0$ that is independent of $\kappa$, the mesh size $h$ and the stabilization parameter $\beta$ (the latter two are to be defined later). To account for different $\kappa$-scalings of the derivatives of a minimizer, we equip the Sobolev spaces $H^1$ and $H^2$ with $\kappa$-weighted norms given by %
\begin{eqnarray}
 \label{scaled_norms} 
\| v \|_{H^1_{\kappa}(\Omega)}&:=& \left(  \| v \|_{L^2(\Omega)}^2  + \tfrac{1}{\kappa^2} \| \nabla v \|_{L^2(\Omega)}^2  \right)^{1/2}
\qquad
\mbox{and}\\
\nonumber \| v \|_{H^2_{\kappa}(\Omega)}&:=& \left(  \| v \|_{L^2(\Omega)}^2  + \tfrac{1}{\kappa^2} \| \nabla v \|_{L^2(\Omega)}^2 + \tfrac{1}{\kappa^4} \| D^2 v \|_{L^2(\Omega)}^2 \right)^{1/2},
\end{eqnarray}
where $D^2 v$ denotes the Hessian of a function $v \in H^2(\Omega)$. 

As first step towards the desired estimates, the following theorem summarizes analytical results regarding the existence of minimizers as well as a quantification of its \quotes{size} with respect to $\kappa$. 
\begin{theorem} \label{estimates_u}
	Assume \ref{A1}-\ref{A3}. Then for each $\kappa$ there exists a global minimizer $u \in H^1(\Omega)$ of $E$, i.e., it holds \eqref{minimizer-energy-def}. %
	 Furthermore, any such minimizer fulfills 
	 \begin{align*}
		\vert u(x) \vert \leq 1 \quad \mbox{ for all }  x \in \Omega
		\qquad\hspace{10pt}
		\mbox{and}\hspace{10pt}
		\qquad \Vert u \Vert_{H^1_\kappa(\Omega)} \lesssim \| u\|_{L^2(\Omega)} \lesssim 1.
	\end{align*}
        Furthermore, we have $u \in H^2(\Omega)$ and the bounds
	 \begin{align*}
	 \| \nabla u \|_{L^4(\Omega)} \lesssim \kappa \qquad \mbox{and} \qquad 
      \Vert u \Vert_{H^2_\kappa(\Omega)} \lesssim 1.
	\end{align*}
\end{theorem}
Global minimizers $u$ are not unique, but the estimates in the theorem hold uniformly in $\kappa$ for all such $u$. For a proof of the theorem we refer to \cite{DuGuPe} for the existence and the pointwise bounds and to \cite{DoeHe23} for the $\kappa$-explicit estimates of $u$ in $H^1$, $H^2$ and of $\nabla u$ in $L^4$. The estimate $\Vert u \Vert_{H^1_\kappa(\Omega)} \lesssim \| u\|_{L^2(\Omega)}$ is proved in \cite{DDH24}. The theorem shows that $\nabla u$ scales essentially with $\kappa^{-1}$, whereas the Hessian $D^2u$ scales with $\kappa^{-2}$. The pointwise bounds imply that the density $|u|$ of superconducting electrons takes values between $0$ and $1$ which is consistent with its physical interpretation.

\subsection{First and second order minimality conditions}
\label{subsection:first-second-order-minimality-conditions}
In a straightforward manner it is seen that the energy is Fr\'echet differentiable on the real Hilbert space $H^1(\Omega)$ and the corresponding first and second Fr\'echet derivatives are given by
\begin{eqnarray}
\label{derivaitves-1}
 \langle E^\prime (v), w \rangle &=& ( \tfrac{\i}{\kappa} \nabla v + \bfA v , \tfrac{\i}{\kappa} \nabla w + \bfA w )_{L^2(\Omega)} + (\, (\vert v \vert^2 -1) v , w )_{L^2(\Omega)} 
\end{eqnarray}
and
\begin{eqnarray}
\label{derivaitves-2}
 \langle E^{\prime\prime} (v)z, w \rangle &=& ( \tfrac{\i}{\kappa} \nabla z+ \bfA z , \tfrac{\i}{\kappa} \nabla w + \bfA w  )_{L^2(\Omega)} + ( \, (\vert v \vert^2 \hspace{-2pt} -\hspace{-2pt}1) z  + v^2 z^\ast +\vert v \vert^2 z ,  w \,)_{L^2(\Omega)} 
\end{eqnarray}
for $v,w,z \in H^1(\Omega)$. Standard theory for minimization/optimization problems (cf. \cite{Casas-trroeltsch-2015}) guarantees that if $u$ is a minimizer of $E$, then it must hold $E^{\prime}(u)=0$ and $E^{\prime\prime}(u) \ge 0$ (where the last statement is understood as the spectrum of $E^{\prime\prime}(u)$ is non-negative). The first order condition $E^{\prime}(u)=0$ for any minimizer is equivalent to the Ginzburg--Landau equation (cf. \eqref{stationarGLE}) and reads in variational form
\begin{align}
\label{eq_GL}
 ( \tfrac{\i}{\kappa} \nabla u + \bfA u , \tfrac{\i}{\kappa} \nabla w + \bfA w)_{L^2(\Omega)} + (\, (\vert u \vert^2 -1) u , w )_{L^2(\Omega)}  = 0 \qquad \mbox{for all } w \in H^1(\Omega).
\end{align}
Note here that the Ginzburg--Landau equation is usually introduced without the real parts in front of the integrals (as contained in our definition of $(\cdot,\cdot)_{L^2(\Omega)}$). However, note that these real parts can be dropped if we simultaneously test in \eqref{eq_GL} with $w$ (giving the real part of the integrals) and with $\ci w$ (giving the imaginary parts of the integrals) and sum up the two contributions.

The second order condition $E^{\prime\prime}(u) \ge 0$ is a necessary, but not a sufficient condition for local minima. This raises the question if we can even expect the sufficient condition $E^{\prime\prime}(u) > 0$ (i.e. only positive eigenvalues) to hold. This is unfortunately not the case, as seen by a simple calculation. Let $u \in H^1(\Omega)$ be any minimizer, then $\ci u$ is another minimizer because $E(u)=E(\ci u)$. Hence, we have with the first order minimality condition
\begin{align}
\label{first-order-opt-cond}
E^\prime ( u ) = E^\prime (\ci u ) = 0.
\end{align}
With this, we use $v=u$ and $z = \ci u$ in \eqref{derivaitves-2} to obtain for all $w \in H^1(\Omega)$
\begin{eqnarray*}
\lefteqn{ \langle E^{\prime\prime} (u) \ci u, w \rangle } \\
&=& ( \tfrac{\i}{\kappa} \nabla (\ci u)+ \bfA (\ci u) , \tfrac{\i}{\kappa} \nabla w + \bfA w  )_{L^2(\Omega)} + ( \, (\vert u \vert^2 -1) (\ci u)  + u^2 (\ci u)^\ast +\vert u \vert^2 (\ci u) ,  w \,)_{L^2(\Omega)}  \\
&\overset{\eqref{derivaitves-1}}{=}&
 \langle E^\prime ( \ci u ), w \rangle +  ( \,  - \ci |u|^2 u + \ci \vert u \vert^2 u ,  w \,)_{L^2(\Omega)}  \,\,\, = \,\,\,\langle E^\prime ( \ci u ), w \rangle  \,\,\, \overset{\eqref{first-order-opt-cond}}{=} \,\,\,  0.
\end{eqnarray*}
Hence, $\ci u$ is an eigenfunction of $E^{\prime\prime} (u)$ with eigenvalue $0$ and $E^{\prime\prime}(u) > 0$ cannot hold without further restrictions. This makes also sense when making the well-known observation that
$$
E( \,\exp(\ci \omega) u  \,) = E( u ) \qquad \mbox{for all } \omega \in [-\pi,\pi ).
$$ 
The property, known as gauge invariance of the energy, implies that ground states are at most unique up to a constant complex phase shift $\exp(\ci \omega)$. Hence, the energy is constant on the circular curve given by $\gamma: \omega \mapsto \exp(\ci \omega) u$ and we consequently obtain $\frac{\mbox{\scriptsize d}^{k}}{\mbox{\scriptsize d}^{k}\omega}E(\,\gamma(\omega)\,) =0$ for any $k\ge1$. Noting that $\gamma(0)=u$ and $\gamma^{\prime}(u) = \ci u$, we conclude that 
$$
0= \frac{\mbox{d}^{2}}{\mbox{d}^{2}\omega}E(\,\gamma(\omega)\,) \big\vert_{\omega = 0} = 
\langle E^{\prime\prime}(\,\gamma(0)\,) \, \gamma^{\prime}(0) , \gamma^{\prime}(0)  \rangle + \langle E^{\prime}(\,\gamma(0)\,) , \gamma^{\prime\prime}(0) \rangle
= \langle E^{\prime\prime}(u) \ci u , \ci u \rangle,
$$
which precisely confirms again that $E^{\prime\prime}(u)$ is singular in the direction $\ci u$. Even though a rigorous analytical justification is still an open problem of the field, numerical experiments indicate that $\ci u$ is the only singular direction and that all other eigenvalues of $E^{\prime\prime}(u)$ are indeed positive. In other words: zero is a simple eigenvalue of $E''(u)$ with kernel $\mathrm{span} \{ \i u\}$. In the numerical experiments in Section \ref{section:numerical_experiments} we computed the spectrum and confirmed this claim in our test setting. Note that this implies that the considered minimizer is well-separated from all other potential minimizers (aside from phase shifts).
We refer to \cite{DoeHe23} for a more detailed discussion of this issue. For our analysis, we therefore make the reasonable assumption that $E^{\prime\prime}(u)$ has only positive eigenvalues in the orthogonal complement of the first eigenspace $\mbox{span}\{ \ci u\}$, or, equivalently expressed, that all minimizers are locally unique up to phase shifts. The assumption is fixed as follows.
\begin{enumerate}[resume,label={(A\arabic*)}]
\item\label{A4} For any minimizer $u \in \underset{{v \in H^1(\Omega)}}{\mbox{arg\,min}} \, E(v)$, it holds
$$
\langle E^{\prime\prime}(u) v , v \rangle >0 \qquad \mbox{for all } v \in \orthiu \setminus \{ 0 \},
$$
where $\orthiu$ denotes the $L^2$-orthogonal complement of $\i u$ in $H^1$, i.e.,
\begin{align*}
\orthiu := \{ v \in H^1(\Omega) \mid (\i u, v)_{L^2(\Omega)} =0 \}. 
\end{align*}
\end{enumerate}
Recall here that our definition of $(\cdot,\cdot)_{L^2(\Omega)}$ only includes real parts, so that the orthogonal complement of $\ci u$ is not the same as the orthogonal complement of $u$. Also note that for any minimizer $u$, we naturally have $\langle E^{\prime\prime}(u) v , v \rangle \ge 0$ and  $\langle E^{\prime\prime}(u) \ci u , \cdot \,\,\rangle = 0$. Hence, assumption \ref{A4} effectively demands that the spectrum of $E^{\prime\prime}(u)$ is non-negative and that the second smallest eigenvalue of $E^{\prime\prime}(u)$ is bounded away from zero. Once $u$ is computed, this can be verified a posteriori by computing the spectrum of $E^{\prime\prime}(u)$, cf. Section \ref{section:numerical_experiments}. Details on how the eigenvalues of $E^{\prime\prime}(u)$ are practically computed are given in \cite[Appendix A]{DoeHe23}.

Next, we observe that Assumption \ref{A4} directly implies coercivity of $E^{\prime\prime}(u)$ on $\orthiu$ (cf. \cite[Proposition 2.6]{DoeHe23}).
\begin{proposition}
\label{coercivity_secE_u}
Assume \ref{A1}-\ref{A4}. Then there exists a constant $\Ccoe^{-1}>0$, which depends on $u$ and $\kappa$, such that
\begin{align*}
\langle E^{\prime\prime}(u) v , v \rangle \ge \Ccoe^{-1} \, \| v \|_{H^1_{\kappa}(\Omega)}^2 \qquad \mbox{for all } v \in \orthiu.
\end{align*}
\end{proposition}
The precise dependency of $\Ccoe$ on $\kappa$ is unknown, but can be computed numerically as soon as the minimizer is available (or a sufficiently accurate approximation). Empirically it holds $\Ccoe^{-1} \rightarrow 0$ (or equivalently $\Ccoe \rightarrow \infty$) for $\kappa \rightarrow \infty$. The numerical experiments in \cite{DoeHe23} suggest that the decay is polynomial on convex domains, i.e., $\Ccoe \lesssim \kappa^{\alpha}$ for some positive $\alpha$.
\begin{proof}[Proof of Proposition \ref{coercivity_secE_u}]
The short argument is standard and we briefly sketch it for completeness. Assumption \ref{A4} and the Courant--Fischer theorem imply that
\begin{align*}
\lambda_2 := \underset{v\not= 0}{\inf_{v \in \orthiu}} \frac{\langle E^{\prime\prime}(u) v , v \rangle }{(v,v)_{L^2(\Omega)} } >0
\qquad
\mbox{and hence}
\quad
\langle E^{\prime\prime}(u) v , v \rangle \ge \lambda_2 \| v \|_{L^2(\Omega)}^2 \quad \mbox{for all } v \in \orthiu.
\end{align*}
At the same time, we easily verify with \eqref{derivaitves-2}, the pointwise estimate $|u|\le 1$, and the identity $( \, u^2 v^\ast +\vert u \vert^2 v ,  v \,)_{L^2(\Omega)}  = 2 \int_{\Omega} \Re( u  v^\ast )^2 \d x \ge 0$
that the following G{\aa}rding inequality holds:
\begin{align*}
\langle E^{\prime\prime}(u) v , v \rangle \ge \frac{1}{2\kappa^2} \| \nabla v \|^2_{L^2(\Omega)} - \| \bfA v \|_{L^2(\Omega)}^2 - \| v \|_{L^2(\Omega)}^2 
\ge  \frac{1}{2} \| v \|^2_{H^1_{\kappa}(\Omega)}  - (\tfrac{3}{2} + \| \bfA \|_{L^{\infty}(\Omega)} ) \| v \|_{L^2(\Omega)}^2 .
\end{align*}
Combining both lower bounds for $\langle E^{\prime\prime}(u) v , v \rangle$ we obtain the desired coercivity with
\begin{align}
\label{coercivity-Eprimeprime}
\langle E^{\prime\prime}(u) v , v \rangle \,\, \ge \,\, \frac{\lambda_2}{2(\lambda_2+ \tfrac{3}{2} + \| \bfA \|_{L^{\infty}(\Omega)} )}  \| v \|^2_{H^1_{\kappa}(\Omega)}.
\end{align}
The $\kappa$ dependency enters through $\lambda_2$.
\end{proof}
By the boundedness of $\bfA$ and $u$ it is easy to see that $E^{\prime\prime}(u)$ is also continuous on $H^1(\Omega)$ w.r.t.\ $\| \cdot \|_{H^1_\kappa(\Omega)}$ with a continuity constant independent of $\kappa$, that is, it holds
\begin{align*}
\langle E^{\prime\prime}(u) v, w \rangle \lesssim \| v \|_{H^1_{\kappa}(\Omega)}  \| w \|_{H^1_{\kappa}(\Omega)}
\end{align*}
for all $v,w \in H^1(\Omega)$. See also \cite[Lemma 2.3]{DoeHe23}.

\begin{remark}[Boundary condition]
\label{boundary-condition}
As mentioned above, any minimizer $u \in H^1(\Omega)$ solves the Ginzburg--Landau equation \eqref{eq_GL} with the natural boundary condition $(\tfrac{\ci}{\kappa}\nabla u + \bfA \hspace{1pt} u) \cdot \mathbf{n}=0$ on $\partial \Omega$. Since assumption \ref{A2} guarantees $\bfA \cdot \mathbf{n}=0$, the boundary condition reduces to $\nabla u \cdot \mathbf{n}=0$. Recalling now that $\Omega$ is Lipschitz-continuous (which follows from the convexity) and $u\in H^2(\Omega)$ for any minimizer, we verify that the boundary condition $\nabla u \cdot \mathbf{n}=0$ is in fact fulfilled in the sense of traces.
\end{remark}

\subsection{A stabilized bilinear form on $H^1_{\kappa}(\Omega)$}
A crucial component of our error analysis and the later construction of multiscale spaces is a bilinear form that is based on the $u$-independent part of $\langle E^{\prime\prime}(u) \,\cdot ,\cdot \rangle$ and that is stabilized by a sufficiently large $L^2$-contribution required to ensure coercivity, cf. \cite{DoeHe23}. To be precise, we let
\begin{eqnarray}
\abeta(v,w) &:=& ( \tfrac{\i}{\kappa} \nabla v+ \bfA v , \tfrac{\i}{\kappa} \nabla w + \bfA w  )_{L^2(\Omega)} + \beta \, (v,w)_{L^2(\Omega)}
	\label{bilinear_forms}
\end{eqnarray}
for a stabilization parameter $\beta \ge 0$. We will track the parameter in our estimates and investigate different choices both analytically and numerically. The bilinear form is continuous and coercive for sufficiently large values of $\beta$.
\begin{lemma} \label{Lemma_2.1}
	Assume \ref{A1}-\ref{A3}. Then there is a constant $C_{\hspace{-1pt}\bfA} > 0$ independent of $\kappa$ and $\beta$ such that for all $v,w \in H^1(\Omega)$  it holds
	\begin{equation*} 
	\abeta (v,w) \, \leq \, C_{\hspace{-1pt}\bfA} \Vert v \Vert_{H^1_\kappa(\Omega)} \, \Vert w \Vert_{H^1_\kappa(\Omega)} +  \beta \, \Vert v \Vert_{L^2(\Omega)} \, \Vert w \Vert_{L^2(\Omega)}.
	\end{equation*} 
	Furthermore, if $\beta$ and $c_{\hspace{-1pt}\bfA}$ are selected such that $0 \le c_{\hspace{-1pt}\bfA} < 1$ and $ \beta  \ge \frac{1+  2 \| \bfA \|_{L^\infty(\Omega)}^2 }{2(1 - c_{\hspace{-1pt}\bfA})}$, then it holds
	\begin{equation*} 
	\abeta (v,v) \, \geq \, \tfrac{1}{2} \Vert v \Vert^2_{H^1_\kappa(\Omega)} + c_{\hspace{-1pt}\bfA} \,\beta \| v \|_{L^2(\Omega)}^2 \,\,\, \ge\, \tfrac{1}{2} \Vert v \Vert^2_{H^1_\kappa(\Omega)}.
	\end{equation*} 
\end{lemma}
The result can be already extracted from \cite{DoeHe23}. We briefly present the arguments here to specify the $\beta$-dependency and because we need to revisit it at a later point (cf. Lemma \ref{lemma-abeta-coercive-on-W} below).
\begin{proof}[Proof of Lemma \ref{Lemma_2.1}]
The continuity is straightforward. For the coercivity we use the inverse triangle inequality and the Young inequality to obtain
	\begin{align*}
	| \tfrac{\ci}{\kappa} \nabla v + \bfA v |^2 \ge \tfrac{1}{\kappa^{2}} |\nabla v |^2 + | \bfA v|^2 - 2 \tfrac{1}{\kappa} |\bfA v| \, \, |\nabla v | \ge \tfrac{1}{2} \tfrac{1}{\kappa^{2}} |\nabla v |^2 -  | \bfA v|^2.
	\end{align*}
	Hence
	\begin{align*}
	| \tfrac{\ci}{\kappa} \nabla v + \bfA v |^2 + \beta \, |v|^2 &\ge \tfrac{1}{2}  \tfrac{1}{\kappa^{2}} |\nabla v |^2 + \tfrac{1}{2}| v|^2 + ( \beta - \| \bfA \|_{L^\infty(\Omega)}^2 - \tfrac{1}{2})  | v|^2.
	\end{align*}
	With the condition $(1 - c_{\hspace{-1pt}\bfA}) \beta \ge \tfrac{1}{2}+ \| \bfA \|_{L^\infty(\Omega)}^2$ we conclude
	$\abeta (v,v) \geq \tfrac{1}{2} \Vert v \Vert^2_{H^1_\kappa(\Omega)} + c_{\hspace{-1pt}\bfA} \beta  \| v\|^2_{L^2(\Omega)}$.
\end{proof}
Lemma \ref{Lemma_2.1} shows that $\abeta(\cdot,\cdot)$ is always an inner product on $H^1(\Omega)$ if $ \beta \ge \tfrac{1}{2}+  \| \bfA \|_{L^\infty(\Omega)}^2 $. In this case, we choose $c_{\hspace{-1pt}\bfA}=0$.

\section{Finite element discretization}
\label{section:standardFEM}
In this section we recall the essential approximations properties of $\mathbb{P}^1$-Lagrange finite element approximations of minimizers. For that we assume that 
\begin{enumerate}[resume,label={(A\arabic*)}]
\item\label{A5} $\mathcal{T}_h$ is a shape-regular, conforming and quasi-uniform partition of $\Omega$ into triangles ($d=2$) or tetrahedra ($d=3$).
\end{enumerate}
The conformity implies that two distinct elements $T, T'\in \mathcal{T}_h$ are either disjoint or share a common vertex, edge or face. We define the corresponding (maximum) mesh size of the partition $\mathcal{T}_h$ as $h:=\max\{ \diam(T)\,|\,T\in \mathcal{T}_{h}\}$. 
By 
\begin{align}
\label{def-Vh}
V_h := \{ v_h \in H^1(\Omega) \, | \,\, v_h \vert_{T} \in \mathbb{P}^1(T) \mbox{ for all } T \in \mathcal{T}_h \}
\end{align}
we denote the corresponding space of (conforming) $\mathbb{P}^1$-Lagrange finite elements. The following theorem summarizes the approximation properties of minimizers in $V_h$ as proved in \cite[Theorem 3.3 and Corollary 3.4]{DoeHe23}.
\begin{theorem}\label{theorem-FEM-estimates}
Assume \ref{A1}-\ref{A5} and let $u_h \in V_h$ denote a discrete minimizer, i.e., with
$$
E(u_h) = \min_{v_h \in V_h} E(v_h). 
$$ 
If the mesh size $h$ is sufficiently small (possibly depending on $\kappa$), then there exists an open ball $B_{\delta}(u_h):= \{ v\in H^1(\Omega)\,|\, \| v-u_h \|_{H^1_{\kappa}(\Omega)} < \delta \}$ for some $\delta>0$, such that there is a unique exact minimizer $u \in H^1(\Omega)$ to \eqref{minimizer-energy-def} with 
\begin{align}
\label{same-phase-u-uh}
u \in B_{\delta}(u_h), \qquad 
(\ci u , u_h  )_{L^2(\Omega)} = 0 \qquad \mbox{and} \qquad (u , u_h  )_{L^2(\Omega)} \ge 0
\end{align}
and with
\begin{align*}
\| u - u_h \|_{H^1_{\kappa}(\Omega)} \lesssim (1+ \kappa \Ccoe h ) \inf_{v_h \in V_h} \| u - v_h \|_{H^1_{\kappa}(\Omega)}.
\end{align*}
Here we recall $\Ccoe^{-1}$ as the coercivity constant from Proposition \ref{coercivity_secE_u}. Exploiting the $H^2$-regularity of $u$ (cf. Theorem \ref{estimates_u}) we conclude
\begin{align}
\label{H1-error-estimate-P1FEM}
\| u - u_h \|_{H^1_{\kappa}(\Omega)} \lesssim (1+ \kappa \Ccoe h ) \, h \kappa.
\end{align}
\end{theorem}
Let us briefly discuss the theorem. First of all, property \eqref{same-phase-u-uh} in Theorem \ref{theorem-FEM-estimates} states that the phases of $u$ and $u_h$ are aligned, which is important to make a reasonable comparison between the two due to lack of uniqueness (cf. Section \ref{subsection:first-second-order-minimality-conditions}). This can be achieved by demanding that the real and imaginary parts of $u$ and $u_h$ are aligned in the sense that $\Re \int_{\Omega} u \, u_h^{\ast} \,\mbox{d}x \,\ge 0$ and $\Im \int_{\Omega} u \, u_h^{\ast} \, \mbox{d}x \,= 0$. Note that the two conditions are equivalently expressed as $(u_h, u)_{L^2(\Omega)} \ge 0$ and $(u_h,\ci u)_{L^2(\Omega)}=0$, hence we have \eqref{same-phase-u-uh}. The condition is natural, since it must be fulfilled in the limit for $u_h=u$. On the contrary, if $u$ and $u_h$ are arbitrary minimizers (whose real and imaginary parts are not necessarily aligned) then we have in general $\int_{\Omega} u \, u_h^{\ast} \d x = r \exp(\ci \omega )$ for some $r \in \R_{\ge 0}$ and $\omega \in [-\pi , \pi)$. The states $u$ and $u_h$ are therefore not directly comparable. However, if we multiply $u$ with $\exp(-\ci \omega )$ then $\tilde{u} := \exp(-\ci \omega ) u$ is still a minimizer since $E(\tilde{u})=E(u)$ and we further have $\int_{\Omega} u \, u_h^{\ast} \d x = r  \in \R_{\ge 0}$. Hence, $\tilde{u}$ fulfills property \eqref{same-phase-u-uh} and is comparable to $u_h$. Consequently, for any discrete minimizer $u_h$ we can always find an exact minimizer $u$ whose real and imaginary parts are aligned with $u_h$ in the sense of \eqref{same-phase-u-uh}.

Second, estimate \eqref{H1-error-estimate-P1FEM} shows that the error in the $H^1_{\kappa}$-norm behaves asymptotically like $h \kappa$, which implies the fundamental resolution condition $h \lesssim \kappa^{-1}$ to obtain small errors and hence meaningful approximations. However, at the same time, the estimate also requires $\kappa \Ccoe h  \lesssim 1$ so that the rate $h \kappa$ becomes visible. Since we expect $\Ccoe$ to grow polynomially with $\kappa$, this suggests a preasymptotic effect where convergence is in fact only observed when $h \lesssim \kappa^{-1} \Ccoe^{-1}$. Numerical experiments \cite{DoeHe23} confirm that such a preasymptoic effect is indeed visible, even though it is open if the condition $h \lesssim \kappa^{-1} \Ccoe^{-1}$ is optimal.\\[0.3em]
In the following sections we will investigate an alternative discretization based on localized orthogonal decomposition and demonstrate how higher order convergence rates can be achieved in these spaces.
\section{Ideal Localized Orthogonal Decomposition}
\label{section:ideal_LOD}
We consider the minimization of $E$ in generalized multiscale spaces based on the concept of {\it localized orthogonal decomposition} as proposed in \cite{DoeHe23}. In a nutshell, the ideal LOD space is constructed by identifying the bilinear form $\abeta(\cdot,\cdot)$ with a differential operator $\Abeta$ and by applying $\Ainvbeta$ to the FE space $V_h$ from Section \ref{section:standardFEM}. For this we need to ensure that $a_{\beta}(\cdot,\cdot)$ is coercive on $H^1(\Omega)$ which is guaranteed by Lemma \ref{Lemma_2.1} if we assume $\beta \ge \tfrac{1}{2} + \| A \|_{L^\infty(\Omega)}^2$. In this case, we define for $f \in L^2(\Omega)$ the image $\Ainvbeta f \in H^1(\Omega)$ as the solution to
\begin{align*}
\abeta( \Ainvbeta f , v) = (f,v)_{L^2(\Omega)}\qquad \mbox{for all } v \in H^1(\Omega).
\end{align*}
Then the ideal LOD space is defined by
$$
\Vlod := \Ainvbeta V_h
$$
which is however only well-defined under the (sufficient) constraint $\beta \ge \tfrac{1}{2} + \| A \|_{L^\infty(\Omega)}^2$. In general, the necessary condition is that $-\beta$ is not an eigenvalue of $a_0(\cdot,\cdot)$ on $V_h$. To avoid such restrictions on $\beta$, we consider an alternative (\quotes{classical}) construction of $\Vlod$ which leads to an approximation space that is in fact well-defined for any $\beta \ge 0$ as long as the mesh size fulfills the natural resolution condition $h \lesssim \kappa^{-1}$. Both constructions of $\Vlod$ are equivalent if $\Ainvbeta V_h$ exists (in particular for sufficiently large $\beta$). With the alternative construction and relaxed assumptions on $\beta$, we generalize the results from \cite[Sec.~6]{DoeHe23}. The subsequent characterization of the space $\Vlod$ lays also the foundation for localization results following later in Section \ref{section:localization-LOD-basis}. Before presenting the characterization, we require a short preparation, that is, we introduce the kernel of the $L^2$-projection on $V_h$ and show coercivity of $\abeta(\cdot,\cdot)$ on that kernel.
\subsection{Coercivity of $\abeta(\cdot,\cdot)$ on the detail space $W$}
In the following we let $\pi_h : H^1(\Omega) \rightarrow V_h$ denote the $L^2$-projection given by
\begin{align}
\label{def-L2-proj}
( \pi_h v ,v_h )_{L^2(\Omega)} = ( v ,v_h )_{L^2(\Omega)}  \qquad \mbox{for all } v_h \in V_h.
\end{align}
Note that even though the inner product $(\cdot,\cdot)_{L^2(\Omega)}$ only contains the real part of the usual complex $L^2$ inner product, the above definition of the $L^2$-projection still coincides with the traditional definition.
With this, we define the kernel of $\pi_h$ on $H^1(\Omega)$ by 
\begin{align}
\label{def-W-detail-space}
W := \mbox{kern}\,\pi_h \vert_{H^1(\Omega)} = \{ v \in H^1(\Omega) \,|\, \pi_hv=0 \}.
\end{align}
This space is a closed subspace of $H^1(\Omega)$ due to the $H^1$-stability of $\pi_h$ for quasi-uniform meshes (which is guaranteed by \ref{A5}). We call $W$ a {\it detail space}, because it contains fine scale details that cannot be represented in the FE space $V_h$.

Even though the coercivity of $\abeta(\cdot,\cdot)$ 
in Lemma \ref{Lemma_2.1} only holds under additional assumptions on $\beta$, we will see next that the bilinear form 
is in fact coercive on $W$ even for $\beta=0$, as long as $h$ is small enough. In general, this property is not true on the full space $H^1(\Omega)$.
\begin{lemma}
\label{lemma-abeta-coercive-on-W}
Assume \ref{A1}-\ref{A3} and \ref{A5}. Then there exists a generic constant $C_{\mbox{\normalfont\tiny res}} >0$ that only depends on $\bfA$, $\Omega$ and the mesh regularity constants of $\mathcal{T}_h$, such that if
$$
h \le C_{\mbox{\normalfont\tiny res}} \kappa^{-1},
$$
then for every $\beta \ge 0$ it holds
	\begin{equation*} 
	\abeta (w,w) \geq \frac{1}{4} \Vert w \Vert^2_{H^1_\kappa(\Omega)} + \beta\, \| w \|_{L^2(\Omega)}^2
	\qquad \mbox{for all } w\in W.
	\end{equation*} 	
\end{lemma}
\begin{proof}
Analogously as in the proof of Lemma \ref{Lemma_2.1}, we have for all $w\in W$
\begin{align*}
\int_{\Omega} | \tfrac{\ci}{\kappa} \nabla w + \bfA w |^2
 \d x \ge 
\tfrac{1}{2} \tfrac{1}{\kappa^{2}} \| \nabla w \|_{L^2(\Omega)}^2 + \tfrac{1}{2} \| w \|_{L^2(\Omega)}^2 - (
\| \bfA \|_{L^{\infty}(\Omega)}^2 + \tfrac{1}{2}) \|w \|_{L^2(\Omega)}^2.
\end{align*}
Using $\pi_hw=0$ and the approximation estimate $\| \pi_hw - w\|_{L^2(\Omega)} \le C_{\pi_h} \, h \, \| \nabla w \|_{L^2(\Omega)}$, we get
\begin{align*}
\int_{\Omega} | \tfrac{\ci}{\kappa} \nabla w + \bfA w |^2 
 \d x \ge 
\tfrac{1}{2} \left( \kappa^{-2} - 2  C_{\pi_h}^2 h^2 \| \bfA \|_{L^{\infty}(\Omega)}^2 - h^2  C_{\pi_h}^2 \right) \| \nabla w \|_{L^2(\Omega)}^2 + \tfrac{1}{2} \| w \|_{L^2(\Omega)}^2.
\end{align*}
Hence, if $C_{\mbox{\normalfont\tiny res}}>0$ is such that $\left(  4  C_{\pi_h}^2  \| \bfA \|_{L^{\infty}(\Omega)}^2 + 2 C_{\pi_h}^2 \right) \le C_{\mbox{\normalfont\tiny res}}^{-2}$, we obtain
\begin{align*}
\kappa^{-2} - 2  C_{\pi_h}^2 h^2 \| \bfA \|_{L^{\infty}(\Omega)}^2 - h^2  C_{\pi_h}^2 \ge \kappa^{-2} \left( 1 - 2  C_{\pi_h}^2 C_{\mbox{\normalfont\tiny res}}^{2} \| \bfA \|_{L^{\infty}(\Omega)}^2 - C_{\pi_h}^2 C_{\mbox{\normalfont\tiny res}}^{2} \right) 
\ge \tfrac{1}{2} \kappa^{-2}.
\end{align*}
This yields
\begin{align*}
\abeta( w , w) \ge \tfrac{1}{4} \tfrac{1}{\kappa^{2}} \| \nabla w \|_{L^2(\Omega)}^2 + \tfrac{1}{2} \| w \|_{L^2(\Omega)}^2 +  \beta\, \| w \|_{L^2(\Omega)}^2 \ge \frac{1}{4} \| w \|_{H^1_{\kappa}(\Omega)}^2 +  \beta\, \| w \|_{L^2(\Omega)}^2.\\[-3.0em]
\end{align*}
\end{proof}
From now on we assume $h \le C_{\mbox{\normalfont\tiny res}} \kappa^{-1}$ and regard arbitrary $\beta \ge 0$. We fix this in an assumption.
\begin{enumerate}[resume,label={(A\arabic*)}]
\item\label{A6} Let $C_{\mbox{\normalfont\tiny res}}>0$ denote the generic constant from Lemma \ref{lemma-abeta-coercive-on-W}. We assume that $h \le C_{\mbox{\normalfont\tiny res}} \kappa^{-1}$.
\end{enumerate}
\subsection{Definition of $\Vlod$ through element correctors}
In this section we present the traditional characterization of $\Vlod$ which goes back to \cite{MaP14}. In fact, as pointed out in \cite{AHP21Acta,HaP23}, the ideal space $\Vlod = \Ainvbeta V_h$ can be equivalently defined as the $\abeta(\cdot,\cdot)$-orthogonal complement of the kernel of the $L^2$-projection, provided that we can interpret $\abeta(\cdot,\cdot)$ as an inner product on $H^1(\Omega)$. For the sake of the geometric argument, let us for a moment assume that $\beta$ is indeed sufficiently large so that $\abeta(\cdot,\cdot)$ is an inner product. In this case $\Ainvbeta v_h$ exists for any $v_h \in V_h$ and $\Vlod := \Ainvbeta V_h$ is well defined. Now, let $w \in W=\mbox{kern}\,\pi_h \vert_{H^1(\Omega)}$ and $\Ainvbeta v_h \in \Vlod$ be arbitrary, then
\begin{align*}
\abeta( \Ainvbeta v_h , w )= (v_h, w)_{L^2(\Omega)} = 0.
\end{align*}
Hence, $\Vlod \subset \Wbot$, where $\Wbot$ denotes the $\abeta(\cdot,\cdot)$-orthogonal complement of $W$. To see that $\Vlod = \Wbot$ let $\mathcal{C} : H^1(\Omega) \rightarrow W$ be a corrector given by
\begin{align}
\label{def-ideal-corrector}
\abeta ( \mathcal{C}v , w ) = \abeta ( v , w ) \qquad \mbox{for all } w \in W.
\end{align}
Then $\Wbot = (\id - \mathcal{C})H^1(\Omega)$ and since $H^1(\Omega)=W \oplus V_h$ (direct sum of kernel and image of $\pi_h $) and $(\id - \mathcal{C})W=0$ we obtain $\Wbot = (\id - \mathcal{C}) V_h$. A comparison of the dimensions yields
\begin{align}
\label{characterizationVhLOD}
\Vlod=(\id - \mathcal{C})V_h.
\end{align}
Consequently, the LOD space can be characterized in an alternative way. However, note here that the corrector $\mathcal{C}$ in \eqref{def-ideal-corrector} is well-defined, even for $\beta=0$, thanks to Lemma \ref{lemma-abeta-coercive-on-W}. Hence, the space $\Vlod:=(\id - \mathcal{C})V_h$ is well-defined without imposing additional conditions on $\beta$. For that reason, we work from now on exactly with this characterization of $\Vlod$.\\[0.3em] 
To assemble $\Vlod$ according to that definition, we would practically have to compute $(\id - \mathcal{C})\varphi_z$ for a nodal basis function $\varphi_z$ of $V_h$. However, this structure is not ideal for a localized computation due to a numerical pollution effect related to a broken partition of unity after the localization, cf. \cite{HeP13}. For that reason, the computation of $ \mathcal{C}$ is split into element contributions for each $T\in \mathcal{T}_h$. To be precise, we define 
\begin{align}
\label{definition-CT}
\mathcal{C}_T : V_h \rightarrow W \quad \mbox{by} \quad\qquad \abeta ( \mathcal{C}_Tv_h , w ) = a_{\beta,T} ( v_h , w ) \qquad \mbox{for all } w \in W,
\end{align}
where $a_{\beta,T} ( \cdot , \cdot )$ is the restriction of $\abeta(\cdot,\cdot)$ to the element $T$:
\begin{align*}
a_{\beta,T}(v,w):= (\tfrac{\ci}{\kappa} \nabla v + \bfA\,v,\tfrac{\ci}{\kappa} \nabla w + \bfA\,w)_{L^2(T)} +  \beta\, (v,w)_{L^2(T)}.
\end{align*}
Note here that analogously to the arguments in Section \ref{subsection:first-second-order-minimality-conditions}, we can drop the real parts in front of the integrals in \eqref{definition-CT} and consider the full complex $L^2$-inner products.
Thanks to the coercivity in Lemma \ref{lemma-abeta-coercive-on-W} and since $W$ is closed in $H^1(\Omega)$, the operators $\mathcal{C}_T$ are well-posed. Furthermore, it is easily seen that the element operators $\mathcal{C}_T$ decompose the global corrector $\mathcal{C}$ as
\begin{align*}
\mathcal{C} = \sum_{T \in \mathcal{T}_H} \mathcal{C}_T.
\end{align*}
The operators $\mathcal{C}_T$ are also called {\it element correctors}, because $ \mathcal{C}$ is the global correction of a standard shape function $\varphi_z$, i.e., a nodal basis function $\varphi_z \in V_h$ is corrected to a basis function $\varphi_z - \mathcal{C} \varphi_z$ of $V_h^{\LOD}$. The motivation behind the element correctors $\mathcal{C}_T$ is that problem \eqref{definition-CT} admits a local right hand side (i.e. it is only supported on a single element $T$). As we will see later, this locality of the source ensures that the corresponding solution $\mathcal{C}_T v_h$ (for any $v_h \in V_h$) decays exponentially to zero outside of $T$. This will justify the truncation of $\mathcal{C}_T v_h$ to a small patch around $T$. We will return to this aspect in Section \ref{section:localization-LOD-basis}.\\[0.3em]
We conclude this subsection by formally fixing the definition of $\Vlod$ using element correctors.
\begin{definition}[$\abeta(\cdot,\cdot)$-based LOD space] \label{def_lod}
For $\beta\ge 0$, let $\abeta(\cdot,\cdot)$ denote the bilinear form given by \eqref{bilinear_forms} and let $\mathcal{C}_T : V_h \rightarrow W$ denote the (well-posed) element correctors defined in \eqref{definition-CT}. The corresponding LOD space is defined by
\begin{align*}
\Vlod  := (\id - \mathcal{C})V_h  =  \{ v_h - \sum_{T \in \mathcal{T}_H} \mathcal{C}_T  v_h \,| \,\,  v_h \in V_h \}.
\end{align*}
Recall here $\mathcal{C} = \sum\limits_{T \in \mathcal{T}_H} \mathcal{C}_T$ as the global corrector from \eqref{def-ideal-corrector} and $V_h$ given by \eqref{def-Vh}.
\end{definition}
\subsection{Basic approximation properties of $\Vlod$}
Following standard arguments (cf. \cite{AHP21Acta}), we can prove the approximation properties of $\Vlod$. However, we need to give some attention to the aspect that $\abeta(\cdot,\cdot)$ could be singular on $H^1(\Omega)$.
\begin{lemma}
\label{lemma:bestapprox-LOD}
Assume \ref{A1}-\ref{A3}, \ref{A5}-\ref{A6} and let $\beta \ge 0$ be arbitrary. We consider $\Vlod$ from Definition \ref{def_lod}. Let $z \in H^1(\Omega)$ fulfill 
\begin{align}
\label{equation-z-abeta}
\abeta(z , w ) = (f ,w)_{L^2(\Omega)} \qquad \mbox{ for all } w \in H^1(\Omega) 
\end{align}
for some $f \in L^2(\Omega)$. Then it holds
\begin{align}
 \label{eq3.1}
\inf_{v_h^\LOD \in \Vlod} \| z - v_h^\LOD \|_{H^1_{\kappa}(\Omega)} 
\,\,\,\le\,\,\, \| z - (\id - \mathcal{C}) \pi_hz \|_{H^1_{\kappa}(\Omega)} 
\,\,\,\lesssim\,\,\,  h \,\kappa \, \| f  -\pi_h f \|_{L^2(\Omega)}.
\end{align}
The hidden constant in the estimate is generic and only depends on the mesh regularity. In particular, it does not depend on $h$, $\kappa$, $\beta$ or $\bfA$.
\end{lemma}
Note that Lemma \ref{lemma:bestapprox-LOD} does not require that problem \eqref{equation-z-abeta} is well-posed for any given right hand side $f \in L^2(\Omega)$.
\begin{proof}
For $z \in H^1(\Omega)$ we define $z_h^\LOD := (\id - \mathcal{C}) \pi_hz$ and observe that $\pi_h(z -z_h^\LOD) =0$, i.e., $z -z_h^\LOD \in W$. Consequently, with $\| \pi_hw - w\|_{L^2(\Omega)} \le C_{\pi_h} \, h \, \| \nabla w \|_{L^2(\Omega)}$ for every $w \in H^1(\Omega)$, it follows
\begin{align}
\label{L2-z-zLOD}
\| z -z_h^\LOD \|_{L^2(\Omega)} \le C_{\pi_h} \, h \,\kappa \, \|  z -z_h^\LOD \|_{H^1_{\kappa}(\Omega)}.
\end{align}
Now let $\Delta \beta>0$ be large enough so that $a_{\beta + \Delta \beta}(\cdot,\cdot)$ fulfills the coercivity in Lemma \ref{Lemma_2.1}. Note that $\Delta \beta  =\tfrac{1}{2}+  \| \bfA \|_{L^\infty(\Omega)}^2$ is a sufficient condition. We obtain
\begin{eqnarray*} 
\lefteqn{ \tfrac{1}{2} \Vert z -z_h^\LOD \Vert^2_{H^1_\kappa(\Omega)} 
\,\,\,\le\,\,\, a_{\beta + \Delta \beta}( z -z_h^\LOD,z -z_h^\LOD) } \\
&\le& a_{\beta}( z -z_h^\LOD,z -z_h^\LOD) 
+  \Delta \beta \, \| z -z_h^\LOD \|_{L^2(\Omega)}^2 \\
&=& (f , z -z_h^\LOD)_{L^2(\Omega)} - a_{\beta}( (\id - \mathcal{C}) \pi_h z ,z -z_h^\LOD) +  \Delta \beta \, \| z -z_h^\LOD \|_{L^2(\Omega)}^2 \\
&\overset{\eqref{def-ideal-corrector}}{=}&  (f  -\pi_h f , z -z_h^\LOD)_{L^2(\Omega)} +  \Delta \beta \, \| z -z_h^\LOD \|_{L^2(\Omega)}^2 \\
&\overset{\eqref{L2-z-zLOD}}{\le}&  C_{\pi_h} \, h \,\kappa \, \|  z -z_h^\LOD \|_{H^1_{\kappa}(\Omega)} \| f  -\pi_h f \|_{L^2(\Omega)} +  \Delta \beta \, C_{\pi_h}^2 \, h^2 \,\kappa^2 \, \|  z -z_h^\LOD \|_{H^1_{\kappa}(\Omega)}^2.
\end{eqnarray*} 
If $h$ is such that $\Delta \beta \, C_{\pi_h}^2 \, h^2 \,\kappa^2 \le \frac{1}{4}$ we can conclude
\begin{align*}
\Vert z -z_h^\LOD \Vert_{H^1_\kappa(\Omega)} 
\,\,\,\le\,\,\, 4  C_{\pi_h} \, h \,\kappa \, \| f  -\pi_h f \|_{L^2(\Omega)}.
\end{align*}
Thus, the proof is finished if we verify that $\Delta \beta \, C_{\pi_h}^2 \, h^2 \,\kappa^2 \le \frac{1}{4}$ is fulfilled. Recalling assumption \ref{A6}, i.e., $h \le C_{\mbox{\normalfont\tiny res}} \kappa^{-1}$, with $C_{\mbox{\normalfont\tiny res}}^{-2} \ge 4  C_{\pi_h}^2  \| \bfA \|_{L^\infty(\Omega)}^2 + 2 C_{\pi_h}^2$
(see Lemma \ref{lemma-abeta-coercive-on-W} and its proof) we have with the choice $\Delta \beta=\tfrac{1}{2}+  \| \bfA \|_{L^\infty(\Omega)}^2$ 
\begin{align*}
4 \Delta \beta \, C_{\pi_h}^2 h^2 = 
( 4 \| \bfA \|_{L^\infty(\Omega)}^2 + 2 ) C_{\pi_h}^2 h^2\le C_{\mbox{\normalfont\tiny res}}^{-2} h^2 \le  \kappa^{-2}.
\end{align*}
\end{proof}
Next, we consider a particular Ritz-projection onto the LOD-space which we exploit in our error analysis. For the fixed choice $\tilde{\beta}:=\tfrac{1}{2}+  \| \bfA \|_{L^\infty(\Omega)}^2$ we let 
$$
\RitzLOD : H^1(\Omega) \rightarrow \Vlod
$$
be defined for $v\in H^1(\Omega)$ by
\begin{align}
\label{ritz-projection}
a_{\tilde{\beta}} ( \RitzLOD v , w_h^{\LOD} ) = a_{\tilde{\beta}} ( v , w_h^{\LOD} ) 
\qquad
\mbox{for all } w_h^{\LOD} \in \Vlod.
\end{align}
The projection is well-defined by Lemma \ref{Lemma_2.1}.  Note that the $\beta$ used to construct $\Vlod$ is typically not the same as $\tilde{\beta}$. We are hence simultaneously dealing with two different bilinear forms, $\abeta(\cdot,\cdot)$ and $a_{\tilde{\beta}}(\cdot,\cdot)$.

As a direct conclusion from Lemma \ref{lemma:bestapprox-LOD}, we have the following approximation properties for $\RitzLOD$. 
\begin{lemma}
\label{lemma:estimatesRitzprojectionLOD}
Assume \ref{A1}-\ref{A3} and \ref{A5}-\ref{A6}. We consider the LOD-space $\Vlod$ (constructed from $\abeta(\cdot,\cdot)$ for any  $\beta \ge 0$) and let $\RitzLOD : H^1(\Omega) \rightarrow \Vlod$ be the Ritz-projection defined in \eqref{ritz-projection}. Then, for any function $z \in H^1(\Omega)$
with the characterization from Lemma \ref{lemma:bestapprox-LOD}, it holds
\begin{align}
\label{RitzLOD-H1kappa-est}
\Vert z - \RitzLOD z \Vert_{H^1_\kappa(\Omega)} \lesssim  h \,\kappa \, \| f  -\pi_h f \|_{L^2(\Omega)}
\end{align}
and
\begin{align}
\label{RitzLOD-L2-est}
 \| z - \RitzLOD z  \|_{L^2(\Omega)} \lesssim  h\, \kappa\, \sqrt{1+ \beta} \, \|  z - \RitzLOD z\|_{H^1_{\kappa}(\Omega)}.
\end{align}
\end{lemma}
\begin{proof}
The coercivity and continuity of $a_{\tilde{\beta}} ( \cdot , \cdot )$ from Lemma \ref{Lemma_2.1} imply with C\'ea's lemma 
\begin{align*}
\Vert z - \RitzLOD z \Vert_{H^1_\kappa(\Omega)} \lesssim \inf_{v_h^\LOD \in \Vlod} \| z - v_h^\LOD \|_{H^1_{\kappa}(\Omega)} 
\,\,\,\overset{\eqref{eq3.1}}{\lesssim}\,\,\,  h \,\kappa \, \| f  -\pi_h f \|_{L^2(\Omega)}.
\end{align*}
The hidden constants here depend on $\bfA$, but not on $\kappa$, $h$ and $\beta$. For the $L^2$-estimate we use an Aubin-Nitsche argument and consider the adjoint problem by seeking $\xi \in H^1(\Omega)$ with
\begin{align*}
a_{\tilde{\beta}} ( v , \xi ) = ( v , z - \RitzLOD z  )_{L^2(\Omega)} \qquad \mbox{for all } v\in H^1(\Omega).
\end{align*}
Note that $\xi$ fulfils equation \eqref{equation-z-abeta} with $f_{\xi}:= z - \RitzLOD z  + (\beta-\tilde{\beta}) \xi $.  We obtain
\begin{eqnarray*}
\lefteqn{ \| z - \RitzLOD z  \|_{L^2(\Omega)}^2 \,\,\,=\,\,\, a_{\tilde{\beta}} (  z - \RitzLOD z  , \xi ) 
\,\,\,=\,\,\, a_{\tilde{\beta}} (  z - \RitzLOD z  , \xi - (\id - \mathcal{C}) \pi_h\xi ) }\\
&\overset{\eqref{eq3.1}}{\lesssim}&  \|  z - \RitzLOD z\|_{H^1_{\kappa}(\Omega)}\,  h\, \kappa\, 
\| f_{\xi}  -\pi_h f_{\xi} \|_{L^2(\Omega)} \\
&\lesssim&   \|  z - \RitzLOD z\|_{H^1_{\kappa}(\Omega)}^2\,  h^2\, \kappa^2\, 
+  \|  z - \RitzLOD z\|_{H^1_{\kappa}(\Omega)}\,  h\, \kappa\, |\beta -\tilde{\beta} |\,\| \xi - \pi_h \xi \|_{L^2(\Omega)} \\
&\lesssim& h^2\, \kappa^2\,  \left(  \|  z - \RitzLOD z\|_{H^1_{\kappa}(\Omega)}^2  
+  \|  z - \RitzLOD z\|_{H^1_{\kappa}(\Omega)}\, |\beta -\tilde{\beta} |\,\| \xi \|_{H^1_{\kappa}(\Omega)} \right)  \\
&\lesssim& h^2\, \kappa^2\, (1+ \beta) \|  z - \RitzLOD z\|_{H^1_{\kappa}(\Omega)}^2,
\end{eqnarray*}
where in the last step, we used the stability estimate $\| \xi \|_{H^1_{\kappa}(\Omega)} \lesssim \| z - \RitzLOD z  \|_{L^2(\Omega)}$. Taking the square root on both sides of the above estimate finishes the proof.
\end{proof}
As pointed out in \cite[Lemma 6.3]{DoeHe23}, any function $z \in H^2(\Omega)$ with $\nabla z \cdot \mathbf{n} = 0$ on $\partial \Omega$ can be written as the solution of a problem \eqref{equation-z-abeta}, that is, there exists $f_z \in L^2(\Omega)$ such that 
\begin{align*}
\abeta(z , \varphi ) = (f_z ,\varphi)_{L^2(\Omega)} \qquad \mbox{ for all } \varphi \in H^1(\Omega) 
\end{align*}
and it holds
\begin{align*}
\| f_z \|_{L^2(\Omega)} \lesssim \| z \|_{H^2_{\kappa}(\Omega)} + \beta\, \| z \|_{L^2(\Omega)} .
\end{align*}
Note that $z$ can be in particular chosen as a minimizer $u$ because it fulfills $u \in H^2(\Omega)$ by Theorem \ref{estimates_u} and $\nabla u \cdot \mathbf{n}\vert_{\partial \Omega}=0$ by Remark \ref{boundary-condition}. With Lemma \ref{lemma:estimatesRitzprojectionLOD}, we further obtain the (in general suboptimal) estimate
\begin{align}
\label{general-approx-Ritz-H2-func}
\Vert z - \RitzLOD z \Vert_{H^1_\kappa(\Omega)} \lesssim  h \,\kappa \, ( \| z \|_{H^2_{\kappa}(\Omega)} + \beta\, \| z \|_{L^2(\Omega)} ).
\end{align}
Since the set of all functions $z\in H^2(\Omega)$ with $\nabla z \cdot \mathbf{n}\vert_{\partial \Omega}=0$ is dense in $H^1(\Omega)$ (on convex polygons, cf. \cite{Droniou}), we conclude from \eqref{general-approx-Ritz-H2-func} that $\{ \Vlod \}_{h>0}$ is a dense sequence of subspaces of $H^1(\Omega)$, i.e., for every $z \in H^1(\Omega)$ it holds
\begin{align}
\label{density-of-LOD-spaces}
\lim_{h\rightarrow 0} \,\inf_{z_h \in \Vlod} \| z -  z_h \|_{H^1(\Omega)} = 0.
\end{align} 
The density result \eqref{density-of-LOD-spaces} is a necessary assumption to apply certain abstract error estimates  for discrete minimizers. Before elaborating on that aspect, we first want to quantify the error between any minimizer $u$ and its Ritz-projection $\RitzLOD u$ on the LOD space.

\begin{lemma} \label{lem_ritz_proj_hk}
Assume \ref{A1}-\ref{A3}, \ref{A5}-\ref{A6} and let $u \in H^1(\Omega)$ denote a global minimizer of $E$. If $\Vlod$ denotes the LOD-space for $\abeta(\cdot,\cdot)$ (with $\beta \ge 0$ arbitrary) and $\RitzLOD : H^1(\Omega) \rightarrow \Vlod$ the corresponding Ritz-projection from \eqref{ritz-projection}. It holds 
\begin{align*}
\Vert u - \RitzLOD u \Vert_{H^1_\kappa(\Omega)} \lesssim  (1+\beta) \, \kappa^3 \,h^3
\qquad
\mbox{and}
\qquad
 \| u - \RitzLOD u  \|_{L^2(\Omega)} \lesssim  (1+\beta)^{3/2} \, \kappa^4 \,h^4.
\end{align*}
\end{lemma}

\begin{proof}
	In order to apply Lemma \ref{lemma:estimatesRitzprojectionLOD}, we have to express $u$ as solution to a problem of the form $\abeta(u,\cdot)=(f ,v)_{L^2(\Omega)}$ and identify the corresponding source term. For that we use that $u$ solves the Ginzburg--Landau equation \eqref{eq_GL}, which allows us to conclude
	\begin{align*}
\abeta (u,w)  &= ( \tfrac{\i}{\kappa} \nabla u + \bfA u , \tfrac{\i}{\kappa} \nabla w + \bfA w)_{L^2(\Omega)} + \beta  ( u , w )_{L^2(\Omega)} =  (\, (\beta - \vert u \vert^2 +1) u , w )_{L^2(\Omega)}
\end{align*}
for all $w \in H^1(\Omega).$
Thanks to Theorem \ref{estimates_u} we have $(\beta - \vert u \vert^2 +1) u \in H^2(\Omega)$
with the following bounds (cf. \cite{DoeHe23}):
\begin{align*}
		| (\beta - \vert u \vert^2 +1) u |_{H^s(\Omega)} \lesssim (\beta \,  + 1) \kappa^s, \qquad \mbox{for } s=0,1,2.
\end{align*}
Here, $|\cdot|_{H^s(\Omega)}$ denotes the semi-norm on $H^s(\Omega)$ with the convention $|\cdot|_{H^0(\Omega)}=\| \cdot \|_{L^2(\Omega)}$. Applying Lemma \ref{lemma:estimatesRitzprojectionLOD} together with standard approximation estimates of the $L^2$-projection $\pi_h : H^1(\Omega) \rightarrow V_h$ yields
\begin{align*}
\Vert u - \RitzLOD u \Vert_{H^1_\kappa(\Omega)} &\lesssim  h \,\kappa \, \| (\id - \pi_h) \left((\beta - \vert u \vert^2 +1) u\right) \|_{L^2(\Omega)} \\
&\lesssim  h^3 \,\kappa \, \| (\beta - \vert u \vert^2 +1) u \|_{H^2(\Omega)}  \\
&\lesssim (\beta \,  + 1) \, h^3\,\kappa^3.
\end{align*}
The $L^2$-estimate follows from
\begin{align*}
 \| u - \RitzLOD u  \|_{L^2(\Omega)} \lesssim  h\, \kappa\, \sqrt{1+ \beta} \, \|  u - \RitzLOD u\|_{H^1_{\kappa}(\Omega)}.
\end{align*}
\end{proof}
We can now quantify the error between any discrete minimizer in the LOD-space, i.e., $u_{h}^{\LOD} \in \Vlod$ with
\begin{align}
\label{LOD-minimizer-def}
E(u_{h}^{\LOD}) = \min_{v_{h}^{\LOD} \in \Vlod} E( v_{h}^{\LOD})
\end{align}
and the \quotes{closest} exact minimizer $u\in H^1(\Omega)$. For that, we apply the following approximation result.
\begin{lemma}
\label{abstract-error-estimate-idealLOD}
Assume \ref{A1}-\ref{A6}, let $u_{h}^{\LOD} \in \Vlod$ denote a minimizer in the LOD space as in \eqref{LOD-minimizer-def} and let $u \in H^1(\Omega)$ be any exact minimizer in the sense of \eqref{minimizer-energy-def} with $(u_{h}^{\LOD} , \ci u)_{L^2(\Omega)}=0$.
Then there exists a generic constant $c>0$, such that if 
\begin{align}
\label{new-resolution-condition}
 h \le c \, \kappa^{-1} (1+\beta)^{-1/2} 
\end{align}
it holds 
\begin{align*} 
\| u - u_{h}^{\LOD} \|_{H^1_\kappa(\Omega)} \lesssim \; &\left(1\,\,+\,\,  h \,\kappa \,\rho(\kappa) \, (1+h \,\kappa \, \beta ) \right) 
\, \| u - \RitzLOD u \|_{H^1_\kappa(\Omega)} %
\\
			&
			+ \rho(\kappa) \left(\| u-u_{h}^{\LOD}  \|_{L^4(\Omega)}^2 + \| u- u_{h}^{\LOD} \|_{L^6(\Omega)}^3\right).
\end{align*}
Furthermore, if $h$ is sufficiently small (where a precise quantification of the smallness is open), then there exists a particular unique minimizer $u$, such that the estimate reduces to
\begin{align*} 
\| u - u_{h}^{\LOD} \|_{H^1_\kappa(\Omega)} \lesssim \; &\left(1\,\,+\,\,  h \,\kappa \,\rho(\kappa) \, (1+h \,\kappa \, \beta ) \right) 
\, \| u - \RitzLOD u \|_{H^1_\kappa(\Omega)}. %
\end{align*}
\end{lemma}
\begin{proof}
The lemma is a particular case of an abstract approximation result that can be found in \cite[Proposition 5.6]{DoeHe23}. It is worth to mention that the results requires density of the approximation spaces, i.e., property \eqref{density-of-LOD-spaces}. The approximation result in  \cite[Proposition 5.6]{DoeHe23} states that
\begin{align}
\label{proof-ideal-LOD-step-1} \nonumber \| u - u_{h}^{\LOD} \|_{H^1_\kappa(\Omega)} \lesssim \; &\| u - \RitzLODperp u \|_{H^1_\kappa(\Omega)} + \rho(\kappa) \, \| u-\RitzLODperp u \|_{L^2(\Omega)} \\
			&+ \rho(\kappa) \left(\| u-u_{h}^{\LOD}  \|_{L^4(\Omega)}^2 + \| u- u_{h}^{\LOD} \|_{L^6(\Omega)}^3\right).
\end{align}
Where $\RitzLODperp : (\ci u)^{\perp} \rightarrow \Vlod \cap (\ci u)^{\perp}$ is the $(\ci u)^{\perp}$-restricted orthogonal projection  given by
\begin{align*}
a_{\tilde{\beta}} ( \RitzLODperp v , w_h^{\LOD} ) = a_{\tilde{\beta}} ( v , w_h^{\LOD} ) 
\qquad
\mbox{for all } w_h^{\LOD} \in \Vlod  \cap (\ci u)^{\perp}
\end{align*}
with $\tilde{\beta}=\tfrac{1}{2}+  \| \bfA \|_{L^\infty(\Omega)}^2$ as before. However, it is easily checked that $\| z - \RitzLODperp z \|_{H^1_{\kappa}(\Omega)}$ can be bounded by $\| z - \RitzLOD z\|_{H^1_{\kappa}(\Omega)}$ for any $z\in (\ci u)^{\perp}$ if $h \lesssim \kappa^{-1} (1+\beta)^{-1/2}  \| u \|_{L^2(\Omega)} $. To see this, we use $z\in (\ci u)^{\perp}$ and $\RitzLOD (\ci z)= \ci \RitzLOD z$ (which follows from the definition of $\RitzLOD$ with a straightforward calculation) together with the $H^1_{\kappa}$-quasi optimality of $\RitzLODperp$ on $(\ci u)^{\perp}$ (with constants independent of $\kappa$ and $\beta$) to conclude for sufficiently small $h > 0$ (specified below)
\begin{eqnarray*}
\lefteqn{\| z - \RitzLODperp z \|_{H^1_\kappa(\Omega)}
 \,\,\, \lesssim \,\,\, 
  \Vert z - \underbrace{ \left(\RitzLOD z   - \frac{(\RitzLOD z,\i u)_{L^2(\Omega)}}{( \RitzLOD u, u)_{L^2(\Omega)}} \ci  \RitzLOD u \right) }_{\in \Vlod \cap (\ci u)^{\perp}} \|_{H^1_\kappa(\Omega)} }\\
  &=&   \Vert z - \RitzLOD z  + \left( \frac{(\RitzLOD z - z,\i u)_{L^2(\Omega)}}{ \| u \|^2_{L^2(\Omega)} - ( u - \RitzLOD u  , u)_{L^2(\Omega)}} \ci  \RitzLOD u \right)  \|_{H^1_\kappa(\Omega)}\\
&\lesssim&
\| z - \RitzLOD z \|_{H^1_\kappa(\Omega)} +  \frac{  \| \RitzLOD u \|_{H^1_\kappa(\Omega)}  }{\left|  \| u \|_{L^2(\Omega)} - \| u - \RitzLOD u \|_{L^2(\Omega)} \right|}
\| z - \RitzLOD z \|_{L^2(\Omega)} \\
&\overset{\eqref{RitzLOD-L2-est}}{\lesssim}&
\left( 1 +  \frac{ h\,\kappa\, \sqrt{1+\beta}  }{ \left|  1 -  C \, h\,\kappa\, \sqrt{1+\beta} \right|}
 \right) \| z - \RitzLOD z \|_{H^1_\kappa(\Omega)}.
\end{eqnarray*}
In the last step, we also used $\| \RitzLOD u \|_{H^1_\kappa(\Omega)} \lesssim  \| u\|_{L^2(\Omega)}$. This estimate follows from the $H^1_{\kappa}$-stability of $\RitzLOD$, implying $\| \RitzLOD u \|_{H^1_\kappa(\Omega)} \lesssim \|  u \|_{H^1_\kappa(\Omega)}$, together with the estimate $\Vert u \Vert_{H^1_\kappa(\Omega)} \lesssim \| u\|_{L^2(\Omega)}$ from Theorem \ref{estimates_u}. If $h \le c \,\kappa^{-1}\, (1+\beta)^{-1/2}$ for a sufficiently small constant $c \lesssim 1$, we conclude for any $z \in(\ci u)^{\perp}$:
\begin{align}
\label{proof-ideal-LOD-step-2}
\| z - \RitzLODperp z \|_{H^1_\kappa(\Omega)}  \,\,\, \lesssim \,\,\, \| z - \RitzLOD z \|_{H^1_\kappa(\Omega)}.
\end{align}
With this, we also get an estimate for $\| z - \RitzLODperp z \|_{L^2(\Omega)}$ using again an Aubin-Nitsche argument with $\xi \in (\ci u )^{\perp}$ solving
\begin{align*}
a_{\tilde{\beta}} ( v , \xi ) = ( v , z - \RitzLODperp z  )_{L^2(\Omega)} \qquad \mbox{for all } v\in (\ci u )^{\perp}.
\end{align*}
It can be shown, cf. \cite[proof of Lemma 2.8]{DoeHe23}, that $\xi \in H^1(\Omega)$ simultaneously solves a problem
\begin{align*}
a_{\beta} ( \xi , v )
 = (f_{\xi} , v )_{L^2(\Omega)} \qquad \mbox{for all } v\in H^1(\Omega)
\end{align*}
for $f_{\xi} \in L^2(\Omega)$ given by
\begin{align*}
f_{\xi} := z - \RitzLODperp z  +  (\beta-\tilde{\beta})\, \xi + \alpha(z,u,\xi) \, \ci u,
\end{align*}
where
\begin{align*}
\alpha(z,u,\xi) :=   \frac{a_{\tilde{\beta}} ( \ci u , \xi ) - ( \ci u ,  z - \RitzLODperp z )_{L^2(\Omega)}}{( u, u)_{L^2(\Omega)}}
\quad
\mbox{with }\enspace
|\alpha(z,u,\xi)| \lesssim \frac{ \| z - \RitzLODperp z \|_{L^2(\Omega)} }{\| u\|_{L^2(\Omega)} }.
\end{align*}
Note that the last bound exploited the stability estimate $\| \xi \|_{H^1_{\kappa}(\Omega)} \lesssim \| z - \RitzLOD z  \|_{L^2(\Omega)}$, as well as the Ginzburg--Landau equation \eqref{eq_GL} which implies $a_{\tilde{\beta}} ( \ci u , \xi ) =  (\, ( \tilde{\beta} - \vert u \vert^2 + 1) \ci u \,, \xi )_{L^2(\Omega)}$.

Hence, we obtain with \eqref{RitzLOD-H1kappa-est} and \eqref{proof-ideal-LOD-step-2} as well as $\| \ci u -\pi_h (\ci u) \|_{L^2(\Omega)} \le \| u \|_{L^2(\Omega)}$ that
\begin{eqnarray*}
\lefteqn{ \| \xi - \RitzLODperp \xi \|_{H^1_\kappa(\Omega)}  \,\,\, \lesssim \,\,\, \| \xi - \RitzLOD \xi \|_{H^1_\kappa(\Omega)}
\,\,\, \lesssim \,\,\, h \,\kappa \, \| f_{\xi}  -\pi_h f_{\xi} \|_{L^2(\Omega)} }\\
&\lesssim&  h \,\kappa \, \| z - \RitzLODperp z\|_{L^2(\Omega)}
+  h \,\kappa \, \beta  \| \xi  -\pi_h \xi \|_{L^2(\Omega)}
+  h \,\kappa \,  | \alpha(z,u,\xi)| \, \| \ci u -\pi_h (\ci u) \|_{L^2(\Omega)} \\
&\lesssim&  h \,\kappa \, (1+h \,\kappa \, \beta  ) \| z - \RitzLODperp z\|_{L^2(\Omega)} .
\end{eqnarray*}
We can now use this estimate in the adjoint problem for $\xi$ to obtain
\begin{eqnarray}
\nonumber \lefteqn{ \| z - \RitzLODperp z  \|_{L^2(\Omega)}^2 \,\,\,=\,\,\, a_{\tilde{\beta}} (  z - \RitzLODperp z  , \xi - \RitzLODperp \xi)
}\\
\label{proof-ideal-LOD-step-3}&\lesssim&  \|  z - \RitzLODperp z\|_{H^1_{\kappa}(\Omega)}\,   h \,\kappa \, (1+h \,\kappa \, \beta ) \| z - \RitzLODperp z\|_{L^2(\Omega)} 
\end{eqnarray}
The proof is finished by combining \eqref{proof-ideal-LOD-step-1}, \eqref{proof-ideal-LOD-step-2} and \eqref{proof-ideal-LOD-step-3}. The second assertion now follows by Sobolev embedding and taking $h$ sufficiently small.
\end{proof}
Combining the results from Lemmas \ref{lem_ritz_proj_hk} and \ref{abstract-error-estimate-idealLOD} we directly obtain the following theorem.
\begin{theorem}
\label{theorem-ideal-LOD}
Assume \ref{A1}-\ref{A6} and let $u_{h}^{\LOD} \in \Vlod$ denote a minimizer in the LOD space as in \eqref{LOD-minimizer-def}. Then, for all sufficiently small mesh sizes $h$, at least,
\begin{align*}
h \lesssim \, \kappa^{-1} (1+\beta)^{-1/2}
\qquad
\mbox{and}
\qquad
 h \, \kappa \, \beta \lesssim \,1
\end{align*}
there exists a unique minimizer $u \in H^1(\Omega)$ with $(u_{h}^{\LOD} , \ci u)_{L^2(\Omega)} = 0$ and such that
\begin{eqnarray} 
\label{ideal-LOD-estimate-with-L2-tail}
\| u - u_{h}^{\LOD} \|_{H^1_\kappa(\Omega)} &\lesssim&  (1+\beta) \left( \, \kappa^3 \,h^3 + \kappa^4 \, \rho(\kappa) \,h^4 \right).
\end{eqnarray}
Asymptotically, the second term is negligible the error estimate becomes
\begin{eqnarray*} 
\| u - u_{h}^{\LOD} \|_{H^1_\kappa(\Omega)} &\lesssim&  (1+\beta) \, \kappa^3 \,h^3.
\end{eqnarray*}
\end{theorem}
Apparently, the estimate indicates that the error grows with $\beta$ (for fixed $h$ and $\kappa$) and that small values of $\beta$ (in the construction of the LOD space) are preferable.

Overall, Theorem \ref{theorem-ideal-LOD} guarantees a third order convergence $\mathcal{O}(h^3)$ in the ideal LOD space provided that the mesh size is small enough. Even though not all dependencies of the mesh resolution could be traced in the final step, estimate \eqref{ideal-LOD-estimate-with-L2-tail} still suggests small errors only if 
$$
(1+\beta) \, \kappa^4 \, \rho(\kappa) \,h^4 < 1    \qquad
\Leftrightarrow 
\qquad
h < \rho(\kappa)^{-1/4} \, \kappa^{-1} \, (1+\beta)^{-1/4} .
$$ 
Assuming that $\beta=\mathcal{O}(1)$, setting, e.g., $h=\varepsilon \rho(\kappa)^{-1/4} \, \kappa^{-1}$, the estimate yields the smallness
\begin{eqnarray*} 
\| u - u_{h}^{\LOD} \|_{H^1_\kappa(\Omega)} &\lesssim&  \varepsilon^{3} \, \rho(\kappa)^{-3/4} + \varepsilon^4 \,\,\lesssim\,\, \varepsilon^3.
\end{eqnarray*}
Let us compare this with the setting of the standard $\mathbb{P}^1$-FE space $V_h$, where the corresponding estimate (for sufficiently small $h$, cf. Theorem \ref{theorem-FEM-estimates}) reads
\begin{eqnarray*} 
\| u - u_{h} \|_{H^1_\kappa(\Omega)} &\lesssim&  \kappa \,h + \, \kappa^2 \, \rho(\kappa) \,h^2.
\end{eqnarray*}
With the choice $h=\varepsilon \rho(\kappa)^{-1/4} \, \kappa^{-1}$ from the LOD setting we would only get
\begin{eqnarray*} 
\| u - u_{h} \|_{H^1_\kappa(\Omega)} &\lesssim& \varepsilon \rho(\kappa)^{-1/4} + \, \rho(\kappa)^{1/2} \,\varepsilon^2,
\end{eqnarray*}
which is not necessarily small, if $\rho(\kappa)^{1/2}$ is still larger than $\varepsilon^{-2}$. Hence, the necessary condition to ensure smallness is $h < \rho(\kappa)^{-1/2} \, \kappa^{-1}$, which is worse than $h < \rho(\kappa)^{-1/4} \, \kappa^{-1}$ in LOD spaces. Consequently, the analysis suggests that vortices are captured on coarser meshes in $\Vlod$.
\begin{remark}[Best-approximation error estimate]
\label{remark:preasymptotic}
Assume for brevity $\beta \lesssim 1$.
From Lemma \ref{lemma:estimatesRitzprojectionLOD} and Lemma \ref{abstract-error-estimate-idealLOD} and the best-approximation properties of $\RitzLOD$ we can also conclude that
\begin{align*}
\| u - u_{h}^{\LOD} \|_{H^1_\kappa(\Omega)} \,\,\, \lesssim \,\,\, (1+ \, h\, \rho(\kappa) \, \kappa\,) \, \inf_{v_h^{\LOD} \in \Vlod} \| u-v_h^{\LOD} \|_{H^1_{\kappa}(\Omega)}.
\end{align*}
Hence, in order to obtain a quasi-best-approximation (with equivalence constants independent of $\kappa$), the estimate indicates that we need to require $h \lesssim \rho(\kappa)^{-1} \, \kappa^{-1}$, which is a significantly stronger resolution condition than the previous ones. However, there is, per se, no contradiction. It would just imply that quasi-best-approximations are only attained in the regime $h \lesssim \rho(\kappa)^{-1} \, \kappa^{-1}$, whereas meaningful approximations can be already expected in the regime $h \lesssim \rho(\kappa)^{-1/4} \, \kappa^{-1}$ for $\Vlod$ and in the regime $h \lesssim \rho(\kappa)^{-1/2} \, \kappa^{-1}$ for $V_h$.
\end{remark}
In practice, the ideal LOD space $\Vlod$ is replaced by a localized approximation. 
Next, we will analyze such localization and investigate how appropriate localizations depend on $\kappa$. 

\section{Localization of basis functions and full LOD}
\label{section:localization-LOD-basis}

In this section we describe how we can construct an approximate basis of the space $\Vlod$ using localized versions of the element correctors (as used in practical implementations of the method \cite{EHMP19}). Furthermore, we will quantify the decay rate of these basis functions depending on $\kappa$ and $\beta$ in order to make predictions about the sizes of the patches on which we will ultimately compute the LOD basis functions.

In our localization strategy we follow the classical approach originally suggested in \cite{HeP13,HeM14} through element correctors which are exponentially decaying. It is also possible to use a more sophisticated localization strategy such as super localized orthogonal decomposition (SLOD)  proposed by Hauck and Peterseim \cite{HaP23} which allows to even find super-exponentially decaying basis functions of $\Vlod$. Hence, the SLOD can lead to enormous boosts in the assembly time for the LOD spaces. The (current) drawback of the SLOD are stability issues related to the way of how basis functions are computed through a singular value decomposition. Though the SLOD can be computationally superior, we do not follow this strategy in our analysis, because a proof of the super-exponential decay rates (for $d>1$) is still an open problem of the field.\\[0.3em]
In the following subsections, we define the {\it patch neighborhood} of any (not necessarily connected) subdomain $G\subset \overline{\Omega}$ as
\begin{align}
\label{neighborhood-def}
\UN(G):=\Int(\cup\{T\in \mathcal{T}_{h}\,|\,T\cap\overline{G}\neq \emptyset\}).
\end{align}
An illustration if $G$ is a mesh element is given in Figure \ref{figure:patch-neighbor}.

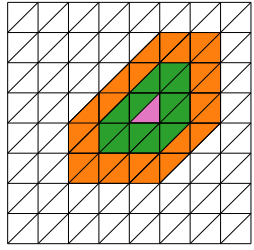
\begin{figure}[h]
	\definecolor{lightgray}{rgb}{0.9290 0.6940 0.1250}
	\definecolor{darkgray}{rgb}{0 0.4470 0.7410}
	\definecolor{dimgray}{rgb}{0 0 0}
	
	\centering
	\adjustbox{scale=0.4}{
	\begin{tikzpicture}
		\path[fill=dimgray,draw=black] (4,4) -- (5,4)-- (5,5);
		
		\path[fill=darkgray,draw=black] (4,4) -- (4,5)-- (5,5);
		\path[fill=darkgray,draw=black] (4,4) -- (4,3)-- (3,3);
		\path[fill=darkgray,draw=black] (4,4) -- (3,4)-- (3,3);
		\path[fill=darkgray,draw=black] (4,4) -- (4,5)-- (3,4);
		\path[fill=darkgray,draw=black] (4,4) -- (5,4)-- (4,3);
		\path[fill=darkgray,draw=black] (5,4) -- (5,3)-- (4,3);
		\path[fill=darkgray,draw=black] (5,4) -- (6,4)-- (5,3);
		\path[fill=darkgray,draw=black] (5,4) -- (5,5)-- (6,5);
		\path[fill=darkgray,draw=black] (5,4) -- (6,4)-- (6,5);
		\path[fill=darkgray,draw=black] (5,5) -- (6,5)-- (6,6);
		\path[fill=darkgray,draw=black] (5,5) -- (5,6)-- (6,6);
		\path[fill=darkgray,draw=black] (5,5) -- (4,5)-- (5,6);

		\path[fill=lightgray,draw=black] (3,3) -- (3,2)-- (2,2);
		\path[fill=lightgray,draw=black] (3,3) -- (2,3)-- (2,2);
		\path[fill=lightgray,draw=black] (3,3) -- (2,3)-- (3,4);
		\path[fill=lightgray,draw=black] (3,3) -- (4,3)-- (3,2);
		\path[fill=lightgray,draw=black] (3,4) -- (2,3)-- (2,4);
		\path[fill=lightgray,draw=black] (4,3) -- (4,2)-- (3,2);
		\path[fill=lightgray,draw=black] (4,3) -- (5,3)-- (4,2);
		\path[fill=lightgray,draw=black] (5,2) -- (5,3)-- (4,2);
		\path[fill=lightgray,draw=black] (5,3) -- (6,3)-- (5,2);
		\path[fill=lightgray,draw=black] (5,3) -- (6,3)-- (6,4);
		\path[fill=lightgray,draw=black] (6,4) -- (7,4)-- (6,3);
		\path[fill=lightgray,draw=black] (6,4) -- (6,5)-- (7,5);
		\path[fill=lightgray,draw=black] (6,4) -- (7,4)-- (7,5);
		\path[fill=lightgray,draw=black] (6,5) -- (7,5)-- (7,6);
		\path[fill=lightgray,draw=black] (6,5) -- (6,6)-- (7,6);
		\path[fill=lightgray,draw=black] (6,6) -- (7,6)-- (7,7);
		\path[fill=lightgray,draw=black] (6,6) -- (6,7)-- (7,7);
		\path[fill=lightgray,draw=black] (6,6) -- (6,7)-- (5,6);
		\path[fill=lightgray,draw=black] (5,6) -- (6,7)-- (5,7);
		\path[fill=lightgray,draw=black] (5,6) -- (5,7)-- (4,6);
		\path[fill=lightgray,draw=black] (5,6) -- (4,6)-- (4,5);
		\path[fill=lightgray,draw=black] (4,5) -- (4,6)-- (3,5);
		\path[fill=lightgray,draw=black] (4,5) -- (3,4)-- (3,5);
		\path[fill=lightgray,draw=black] (3,4) -- (3,5)-- (2,4);
		
		\draw (0,0) -- (8,0);
		\draw (0,0) --(0,8);
		\draw (8,0) -- (8,8);
		\draw (0,8) --(8,8);
		\draw (0,0) -- (8,8);

		\draw (1,0) -- (1,8);
		\draw (2,0) --(2,8);
		\draw (3,0) -- (3,8);
		\draw (4,0) --(4,8);
		\draw (5,0) -- (5,8);
		\draw (6,0) --(6,8);
		\draw (7,0) -- (7,8);
		\draw (0,1) -- (8,1);
		\draw (0,2) -- (8,2);
		\draw (0,3) -- (8,3);
		\draw (0,4) -- (8,4);
		\draw (0,5) -- (8,5);
		\draw (0,6) -- (8,6);
		\draw (0,7) -- (8,7);
		\draw (1,0) -- (8,7);
		\draw (2,0) -- (8,6);
		\draw (3,0) -- (8,5);
		\draw (4,0) -- (8,4);
		\draw (5,0) -- (8,3);
		\draw (6,0) -- (8,2);
		\draw (7,0) -- (8,1);
		\draw (0,1) -- (7,8);
		\draw (0,2) -- (6,8);
		\draw (0,3) -- (5,8);
		\draw (0,4) -- (4,8);
		\draw (0,5) -- (3,8);
		\draw (0,6) -- (2,8);
		\draw (0,7) -- (1,8);
	\end{tikzpicture}}
	\caption{Patch neighborhoods of a triangle $T$: $T={\color{dimgray}\bullet}$, $\UN(T)={\color{dimgray}\bullet} \cup {\color{darkgray}\bullet}$ and $\UN(\UN(T))={\color{dimgray}\bullet} \cup {\color{darkgray}\bullet} \cup {\color{lightgray}\bullet}$}
	\label{figure:patch-neighbor}
\end{figure}

Before we can dive into the localization analysis and quantify decay properties, we require a different characterization of the details space $W$. Here we recall $W$ as the kernel of the $L^2$-projection $\pi_h : H^1(\Omega) \rightarrow W$, cf. \eqref{def-W-detail-space}.

\subsection{Representation of $W$ through local quasi-interpolation}
For practical aspects and analytical considerations it is not ideal to represent $W$ as the kernel of the $L^2$-projection $\pi_h$, as $\pi_h$ is not a local operator. However, we can represent $W$ equivalently through the kernel of a local quasi-interpolation operator. For that, let $\varphi_z \in V_h$ denote the {\it real} nodal shape functions with the property $\varphi_z(z')=\delta_{zz'}$ for all $z,z' \in \mathcal{N}_h$, where $\mathcal{N}_h$ is the set of vortices (corners) of the mesh $\mathcal{T}_h$. 
Any $v_h \in V_h$ is represented by 
\begin{align*}
v_h = \sum_{z \in \mathcal{N}_h } \Re \hspace{1pt}v_h(z) \, \varphi_z +  \sum_{z \in \mathcal{N}_h } \Im\hspace{0.5pt}v_h(z)\hspace{2pt} \ci\varphi_z ,
\end{align*}
i.e., the set $\{ \varphi_z \,,\, \ci \varphi_z \,|\, z \in \mathcal{N}_h \}$ is a basis of $V_h$, as we interpret the space as a {\it real} Hilbert space.

With this, we define the Cl\'ement-type quasi-interpolation operator $P_h : H^1(\Omega) \rightarrow V_h$ by
$$
P_h v := \sum_{z\in \mathcal{N}_h} 
\left( \frac{(v,\varphi_z)_{L^2(\Omega)}}{(1,\varphi_z)_{L^2(\Omega)}} \varphi_z +
 \frac{( v, \ci \varphi_z)_{L^2(\Omega)}}{(1,\varphi_z)_{L^2(\Omega)}} \ci \varphi_z  \right)
 =  \sum_{z\in \mathcal{N}_h} c_z^{-1} \int_{\Omega} v \, \varphi_z \, \d x  \, \varphi_z
$$
with $c_z :=  \int_{\Omega} \varphi_z \, \d x$. The operator was initially suggested and analyzed by Carstensen et al. \cite{Car99,CarVer99}. 
Observe that if $P_h v =0$, then it must hold $\int_{\Omega} v \, \varphi_z \, \d x=0$ for all $z \in \mathcal{N}_h$. Hence, $P_h v$ is $L^2$-orthogonal to $V_h$ and consequently $P_h v \in W = \mbox{kern}\,\pi_h$. Conversely, if $\pi_h v= 0$ then $v \perp_{L^2} V_h$ and hence, by definition of the quasi-interpolation, $P_h v=0$. We conclude that the kernels of $P_h$ and $\pi_h$ indeed coincide, that is
\begin{align*}
W = \mbox{kern}\,\pi_h  = \mbox{kern}\, P_h.
\end{align*}
Note that even though $\pi_h$ is not local, $P_h$ is a local operator in the sense that $P_h$ can increase the support of a function $v_h \in V_h$ by at most one layer of mesh elements and the support of a function $v \in H^1$ by at most two layers of mesh elements (if the support does not align with the mesh, otherwise one layer). This is reflected in the following (local) error and stability estimates proved in \cite{Car99,CarVer99}: For all $v \in H^1(\Omega)$ and every $T\in \mathcal{T}_{h}$ it holds
\begin{align}
\label{supportgrowth-Ph}
\| v - P_h v \|_{L^2(T)} + h \| \nabla v - \nabla P_hv \|_{L^2(T)} \le C_{P_h}  h \| \nabla v \|_{L^2(\UN(T))},
\end{align}
where $C_{P_h}>0$ is a generic, $h$-independent constant that only depends on the mesh regularity and where we recall the definition of $\UN(T)$ from \eqref{neighborhood-def}.

It is easily seen that $P_h: V_h \rightarrow V_h$ is an isomorphism since its kernel is trivial (from $P_hv_h=0$ we have $\pi_h v_h =0 $ and hence $v_h=0$ since $\pi_h$ is a projection). Unfortunately, its natural inverse is not local. However, as first proved in \cite{MaP14}, on the full space, $P_h : H^1(\Omega) \rightarrow V_h$ has a bounded right inverse which is in fact local. In particular, there exists an (non-unique) operator
\begin{align*}
P_h^{-1} : V_h \rightarrow H^1(\Omega) 
\qquad
\mbox{such that}
\qquad (P_h \circ P_h^{-1})v_h = v_h \quad \mbox{for all } v_h \in V_h
\end{align*}
and there exists a constant $C_{\text{\tiny inv}}>0$ such for all $v_h\in V_h$ it holds
\begin{align}
\label{growth-Phinv}
\| \nabla P_h^{-1}v_h \|_{L^2} \le C_{\text{\tiny inv}} \| \nabla v_h \|_{L^2}
\qquad \mbox{and}
\qquad
\mbox{supp}( P_h^{-1}v_h) \subset \UN(\,\mbox{supp}(v_h)\,).
\end{align}
The latter statement means that $P_h^{-1}$ can increase the support of $v_h$ only by at most one layer, i.e., it is also local.  Here we stress that, in general, $ P_h^{-1}v_h \not\in V_h$, which is the price for the locality. With these insights, we are prepared to analyze the decay of the element correctors.

\subsection{Decay of the element correctors}
Recalling the construction of the ideal LOD space in Definition \ref{def_lod} the key to localization are the element correctors $\mathcal{C}_T : V_h \rightarrow W $ defined in \eqref{definition-CT} through problems of the form
\begin{align*}
\abeta ( \mathcal{C}_Tv_h , w ) = a_{\beta,T} ( v_h , w ) \qquad \mbox{for all } w \in W.
\end{align*}
To justify localization we need to show that the element correctors $\mathcal{C}_T$ are decaying exponentially outside of the element $T$. To prove that exponential decay, we can exploit the particular techniques established in \cite{GaP15,Pet15,Pet17}. Details are given in the appendix in Section \ref{section:appendix}.\\[0.3em]
As we will measure the decay in units of mesh layers, we recall for any subdomain $G\subset \overline{\Omega}$ the definition of its mesh neighborhood given by \eqref{neighborhood-def}. With this, we recursively define the $\ell$-layer patches for $\ell \in \mathbb{N}$ as
$$
\UN^1(G):=\UN(G)\qquad \text{and}\qquad\UN^{\ell}(G):=\UN(\UN^{\ell-1}(G)).
$$
Due to the shape regularity of the mesh $\mathcal{T}_{h}$, we have a uniform bound $C_{\ol, \ell}$ for the number of elements in the $\ell$-layer patch
\begin{align}
\label{bounded-overlap}
\max_{T\in \mathcal{T}_{h}}\operatorname{card}\{K\in \mathcal{T}_{h}\,|\,K\subset\overline{\UN^{\ell}(T)}\}\leq C_{\ol, \ell},
\end{align}
where $C_{\ol, \ell}$ depends polynomially on $\ell$ because of the quasi-uniformity of $\mathcal{T}_{h}$.

Furthermore, we define the restriction of $\abeta(\cdot,\cdot)$ to a (closed) subdomain $G\subset \overline{\Omega}$ by
\begin{align*}
a_{\beta,G} (v, w) :=  (\tfrac{\ci}{\kappa} \nabla v + \bfA \,v,\tfrac{\ci}{\kappa} \nabla w + \bfA\,w)_{L^2(G)} +  \beta\, (v,w)_{L^2(G)}
\end{align*} 
with induced norm
\begin{align*}
||| v |||_{\beta,G} := \sqrt{a_{\beta,G}(v, v) }.
\end{align*}
that satisfy $||| v |||_{\beta,G} \lesssim  \| v \|_{H^1_{\kappa}(G)} + \beta \| v \|_{L^2(G)}$. In these local energy norms, we can quantify the exponential decay of $ \mathcal{C}_{T}v_h $ in units of $h$. The corresponding result is as follows.
\begin{theorem}
\label{theorem-decay-ideal-corrector}
 Let $T \in \mathcal{T}_H$, $v_h \in V_h$ and let $\mathcal{C}_{T}v_h \in W$ denote the corresponding element corrector given by \eqref{definition-CT}. Then there exists a constant $0< \theta_{\beta} <1$ such that
\begin{eqnarray}
\label{first-decay-estimate}
|||  \mathcal{C}_{T}v_h  |||_{\beta,\Omega \setminus \UN^{\ell}(T)} &\lesssim& \theta_{\beta}^{\ell} \, |||  v_h  |||_{\beta,T}
\end{eqnarray}
where
\begin{align*}
\theta_{\beta} \,\,\le\,\,  \left( 1 - \tfrac{ 1}{1+C (1 + \kappa \beta h)} \right)^{1/12}  < 1
\end{align*}
for some generic constant $C>0$ independent of $\kappa$, $\beta$ and $h$. 
\end{theorem}
The proof is presented in the appendix, Section \ref{section:appendix}.

\begin{remark}[Discussion of the decay rate]
If $h \lesssim \kappa^{-1}$, then we have
\begin{align*}
\theta_{\beta} \,\,\le\,\,  \left( 1 - \tfrac{ 1}{1+C (1 + \kappa \beta h)} \right)^{1/12} \le \,\,\,\left( 1 - \tfrac{ 1}{1+C (1 + \beta )} \right)^{1/12} < 1
\end{align*}
(where $C>0$ might change in the estimate). This means the decay rate can be bounded independently of $\kappa$. In particular, if $\ell \ge p \log(h) \log(\theta_{\beta})^{-1}$ for some polynomial degree $p$, then 
$$
 \theta_{\beta}^{\ell}  = \exp( \ell \log(\theta_{\beta})  ) \le \exp( p \log(h)  ) = h^p.
$$ 
However, noting the minimum resolution condition $h \lesssim \kappa^{-1}$, we see that we require at least $\ell \gtrsim \log(\kappa)$ layers so that the decay of $ \mathcal{C}_{T}v_h$ becomes visible. Hence, there is an indirect coupling between $\ell$ and $\log(\kappa)$ to ensure smallness of the decay error.
\end{remark}
Theorem \ref{theorem-decay-ideal-corrector} suggests that the computation of the ideal element corrector $\mathcal{C}_T$ can be restricted to a small patch environment of the corresponding element $T \in \mathcal{T}_h$. For that reason, we define the restriction of $W$ to a (closed) subdomain $G\subset \overline{\Omega}$ by
$$
W(G) := \{ w \in W \, | \, w_{ \vert \Omega \setminus G} \equiv 0\}
$$
and introduce for  $\ell \in \mathbb{N}$ the ($\ell$-layer) truncated element corrector  $\mathcal{C}_{T,\ell}v_h \in W(\,\UN^{\ell}(T)\,)$ as the solution to
\begin{align}
\label{definition-CT-ell}
a_{\beta} ( \mathcal{C}_{T,\ell} v_h , w ) = a_{\beta,T} ( v_h , w )  \qquad \mbox{for all } w\in W(\,\UN^{\ell}(T)\,).
\end{align}
The above problems that define $\mathcal{C}_{T,\ell} v_h$ are small elliptic problems on the patches $\UN^{\ell}(T)$ and are hence cheap to solve due to their locality. The next lemma quantifies the error between the ideal element corrector $\mathcal{C}_{T}$ and the truncated approximation $\mathcal{C}_{T,\ell}$, where we refer to Section \ref{section:appendix} for a proof.
\begin{lemma}
\label{lemma:local-corrector-estimate}
 Let $T \in \mathcal{T}_H$, $v_h \in V_h$ and let $\mathcal{C}_{T}v_h \in W$ denote the element corrector given by \eqref{definition-CT} and $\mathcal{C}_{T,\ell} v_h \in W(\,\UN^{\ell}(T)\,)$ its localized version defined in \eqref{definition-CT-ell}. Then, if $h \lesssim \kappa^{-1}$, it holds
\begin{eqnarray}
\label{local-corrector-estimate}
|||  \mathcal{C}_{T}v_h - \mathcal{C}_{T,\ell} v_h  |||_{\beta} &\lesssim& \theta_{\beta}^{\ell} \, (1+ \kappa \beta h)  \, |||  v_h  |||_{\beta,T},
\end{eqnarray}
where $0< \theta_{\beta} <1$ is as in Theorem \ref{theorem-decay-ideal-corrector}.
\end{lemma}
 In the next step, we need to investigate the error between the global correction $\mathcal{C}v_h$ and its global approximation through a localized corrector $\mathcal{C}_{\ell}v_h$, that is,
 \begin{align}
 \label{definition-C-ell}
\mathcal{C}_{\ell} v_h := \sum_{T\in \mathcal{T}_h} \mathcal{C}_{T,\ell} v_h.
\end{align}
The following theorem is again derived in the appendix, Section \ref{section:appendix}.
\begin{theorem}
\label{theorem:C-Cell-error}
Let $v_h\in V_h$ and let $\mathcal{C} v_h \in W$ denote the ideal corrector given by \eqref{def-ideal-corrector} and $\mathcal{C}_{\ell} v_h$ the localized corrector given by \eqref{definition-C-ell}. Then it holds
\begin{eqnarray*}
||| (\mathcal{C}-\mathcal{C}_{\ell}) v_h |||_{\beta} &\lesssim&  \theta_{\beta}^{\ell} \, (1+ \kappa \beta h)  |||  v_h  |||_{\beta},
\end{eqnarray*}
where $0< \theta_{\beta} <1$ is again as in Theorem \ref{theorem-decay-ideal-corrector}. Expressed in the $H^1_{\kappa}$-norm, it holds
\begin{eqnarray*}
 \Vert (\mathcal{C}-\mathcal{C}_{\ell}) v_h \Vert_{H^1_\kappa(\Omega)} + \sqrt{\beta}\, \|  (\mathcal{C}-\mathcal{C}_{\ell}) v_h \|_{L^2(\Omega)}
&\lesssim&  \theta_{\beta}^{\ell} \, (1+ \kappa \beta h) \, \left( \Vert v_h \Vert_{H^1_\kappa(\Omega)} + \sqrt{\beta} \, \Vert v_h \Vert_{L^2(\Omega)} \right).
\end{eqnarray*}
\end{theorem}

\subsection{Localized approximations of Ginzburg--Landau minimizers}
\label{subsection:final-estimates-full-LOD}
With the previous considerations, we are now able to formulate the final localized orthogonal decomposition based on  the $\ell$-layer approximations $\mathcal{C}_{\ell} : V_h \rightarrow W$ of the ideal correctors that is given by \eqref{definition-C-ell}, i.e., $\mathcal{C}_{\ell} v_h := \sum_{T\in \mathcal{T}_h} \mathcal{C}_{T,\ell} v_h$. For that, recall that the element correctors $\mathcal{C}_{T,\ell} v_h \in W(\,\UN^{\ell}(T)\,)$ are given by \eqref{definition-CT-ell} and that they can be computed by solving problems on small patches $\UN^{\ell}(T)$. Compared to the original correctors which required to solve global problems, this reduces the computational complexity significantly. Even though the solving of all local corrector problems still introduces a considerable computational overhead, this is justified for our application of the Ginzburg--Landau equation since the correctors (respectively the LOD space) can be frequently reused within the iterations of, for example, a gradient method to find a minimizer of $E$. \\[0.3em] 
We are ready to introduce the LOD space. According to the characterization \eqref{characterizationVhLOD} of the ideal space $\Vlod$, we define the LOD space with localization parameter $\ell\in \mathbb{N}$ as
\begin{align}
\label{def-VLOD-ell}
\Vlodell=(\id - \mathcal{C}_{\ell})V_h.
\end{align}
The corresponding minimizers $u_{h,\ell}^{\LOD}\in \Vlodell$ of the energy on $\Vlodell$ are given by
\begin{align}
\label{LOD-minimizer-ell-def}
E(u_{h,\ell}^{\LOD}) = \min_{v\in \Vlodell} E(v).
\end{align}
In order to present an error estimate for the fully localized approximation $u_{h,\ell}^{\LOD}$, we start with quantifying the approximation properties of the Ritz-projection onto $\Vlodell$ similarily as in Lemma \ref{lemma:estimatesRitzprojectionLOD}. For that we set, as before, $\tilde{\beta}:=\tfrac{1}{2}+  \| \bfA \|_{L^\infty(\Omega)}^2$ and let $\RitzLODell : H^1(\Omega) \rightarrow \Vlodell$ be given by
\begin{align}
\label{RitzLODell-def}
a_{\tilde{\beta}} ( \RitzLODell v , w_h^{\LOD} ) = a_{\tilde{\beta}} ( v , w_h^{\LOD} )
\qquad
\mbox{for all } w_h^{\LOD} \in \Vlodell.
\end{align}
The following lemma quantifies the approximation properties.
\begin{lemma}
\label{lemma:estimatesRitzprojectionLOD-ell}
Assume \ref{A1}-\ref{A3} and \ref{A5}-\ref{A6}, let $\beta \ge 0$ be arbitrary and let $\Vlodell$ denote the corresponding LOD space given by \eqref{def-VLOD-ell} for some localization parameter $\ell \in \mathbb{N}$.\\[0.3em]
If $z \in H^1(\Omega)$ fulfills
\begin{align*}
\abeta(z , w ) = (f ,w)_{L^2(\Omega)} \qquad \mbox{ for all } w \in H^1(\Omega)
\end{align*}
for some $f \in L^2(\Omega)$, then we have the following error estimates for the Ritz-projection \eqref{RitzLODell-def}:
\begin{align*}
\Vert z - \RitzLODell z \Vert_{H^1_\kappa(\Omega)} \lesssim  h \,\kappa \, \| f  -\pi_h f \|_{L^2(\Omega)} +  \theta_{\beta}^{\ell} \, (1+ \kappa \beta h) \, \left( \Vert z \Vert_{H^1_\kappa(\Omega)} + \sqrt{\beta} \, \Vert z \Vert_{L^2(\Omega)} \right)
\end{align*}
and
\begin{align*}
 \| z - \RitzLODell z  \|_{L^2(\Omega)} \lesssim  \sqrt{1+ \beta} \left(  h \, \kappa\,  + \theta_{\beta}^{\ell} \, (1+ \kappa \beta h) \right) \|  z - \RitzLODell z\|_{H^1_{\kappa}(\Omega)}
\end{align*}
with the decay parameter $0<\theta_{\beta}<1$ from Theorem \ref{theorem-decay-ideal-corrector}.
\end{lemma}
\begin{proof}
For $z \in H^1(\Omega)$ we define $z_{h,\ell}^\LOD := (\id - \mathcal{C}_{\ell}) \pi_h z$ and obtain with Lemma \ref{lemma:bestapprox-LOD} and Theorem \ref{theorem:C-Cell-error}
\begin{eqnarray}
\nonumber \| z - \RitzLODell z \|_{H^1_{\kappa}(\Omega)}  &\lesssim& \| z - z_{h,\ell}^\LOD \|_{H^1_{\kappa}(\Omega)} 
\,\,\,\le\,\,\, \| z - (\id - \mathcal{C}) \pi_h z  \|_{H^1_{\kappa}(\Omega)}  + \| (\mathcal{C} - \mathcal{C}_{\ell}) \pi_h z \|_{H^1_{\kappa}(\Omega)} \\
\nonumber&\lesssim& h \,\kappa \, \| f  -\pi_h f \|_{L^2(\Omega)} +  \theta_{\beta}^{\ell} \, (1+ \kappa \beta h) \, \left( \Vert \pi_h z \Vert_{H^1_\kappa(\Omega)} + \sqrt{\beta} \, \Vert \pi_h z \Vert_{L^2(\Omega)} \right)  \\
&\lesssim& h \,\kappa \, \| f  -\pi_h f \|_{L^2(\Omega)} +  \theta_{\beta}^{\ell} \, (1+ \kappa \beta h) \, \left( \Vert z \Vert_{H^1_\kappa(\Omega)} + \sqrt{\beta} \, \Vert z \Vert_{L^2(\Omega)} \right),
\label{proof-step-1-LODell}
\end{eqnarray}
where we used the $L^2$- and $H^1$-stability of $\pi_h$ for quasi-uniform meshes. To estimate $z - \RitzLODell z $ in $L^2(\Omega)$, we use again an Aubin-Nitsche argument where  $\xi \in H^1(\Omega)$ solves
\begin{align*}
a_{\tilde{\beta}} ( v , \xi ) = ( v , z - \RitzLODell z  )_{L^2(\Omega)} \qquad \mbox{for all } v\in H^1(\Omega).
\end{align*}
Analogously as in the proof of Lemma \ref{lemma:estimatesRitzprojectionLOD} we obtain with $f_{\xi}:= z - \RitzLODell z  + (\beta-\tilde{\beta}) \xi $:
\begin{eqnarray*}
\lefteqn{ \| z - \RitzLODell z  \|_{L^2(\Omega)}^2 \,\,\,=\,\,\, a_{\tilde{\beta}} (  z - \RitzLODell z  , \xi ) 
\,\,\,=\,\,\, a_{\tilde{\beta}} (  z - \RitzLODell z  , \xi - (\id - \mathcal{C}_{\ell}) \pi_h \xi ) }\\
&\overset{\eqref{proof-step-1-LODell}}{\lesssim}& \| z - \RitzLODell z  \|_{H^1_{\kappa}(\Omega)}
\left( h \,\kappa \, \| f_{\xi}  -\pi_h f_{\xi} \|_{L^2(\Omega)} +  \theta_{\beta}^{\ell} \, (1+ \kappa \beta h) \, \left( \Vert \xi \Vert_{H^1_\kappa(\Omega)} + \sqrt{\beta} \, \Vert \xi \Vert_{L^2(\Omega)} \right) \right) \\
&\lesssim& \| z - \RitzLODell z  \|_{H^1_{\kappa}(\Omega)}
\left( h \,\kappa \, \| f_{\xi}  -\pi_h f_{\xi} \|_{L^2(\Omega)} +  \theta_{\beta}^{\ell} \, (1+ \kappa \beta h) (1+\sqrt{\beta}) \, \| z - \RitzLODell z  \|_{L^2(\Omega)} \right) \\
&\lesssim& \| z - \RitzLODell z  \|_{H^1_{\kappa}(\Omega)}
\left( \|  z - \RitzLOD z\|_{H^1_{\kappa}(\Omega)}\,  h^2\, \kappa^2\,
+ |\beta -\tilde{\beta} |\,\| \xi \|_{H^1_{\kappa}(\Omega)} \,h^2\, \kappa^2\,  \right.\\
&\enspace&\hspace{100pt} \left. +  \theta_{\beta}^{\ell} \, (1+ \kappa \beta h) (1+\sqrt{\beta}) \, \| z - \RitzLODell z  \|_{L^2(\Omega)} \right) \\
&\lesssim&  h^2\, \kappa^2\, (1+ \beta) \|  z - \RitzLODell z\|_{H^1_{\kappa}(\Omega)}^2
+ \frac{1}{\delta} \theta_{\beta}^{2\ell} \, (1+ \kappa \beta h)^2 (1+\sqrt{\beta})^2 \| z - \RitzLODell z  \|_{H^1_{\kappa}(\Omega)}^2 \\
&\enspace&\hspace{20pt} +  \delta \| z - \RitzLODell z  \|_{L^2(\Omega)}^2,
\end{eqnarray*}
for any $\delta>0$ by Young's inequality. With an absorption argument we hence conclude
\begin{eqnarray*}
 \| z - \RitzLODell z  \|_{L^2(\Omega)}
&\lesssim& \sqrt{1+ \beta} \left(  h \, \kappa\,  + \theta_{\beta}^{\ell} \, (1+ \kappa \beta h) \right) \|  z - \RitzLODell z\|_{H^1_{\kappa}(\Omega)}.
\end{eqnarray*}
\end{proof}
With the above lemma, we can quantify the projection errors for minimizers $u$ of the Ginzburg-Landau energy.
\begin{lemma} \label{lem_ritz_proj_hk-localized}
Assume \ref{A1}-\ref{A3}, \ref{A5}-\ref{A6}, $\beta \ge 0$ and let $\Vlodell$ be given by \eqref{def-VLOD-ell} for $\ell \in \mathbb{N}$ and  let $u \in H^1(\Omega)$ denote a global minimizer of $E$. Then it holds
\begin{align*}
\Vert u - \RitzLODell u \Vert_{H^1_\kappa(\Omega)} \,\,\lesssim\,\,  (1+\beta) \, h^3\,\kappa^3 +  (1+\beta^{1/2}) \, (1+ \kappa \beta h)\, \theta_{\beta}^{\ell}
\end{align*}
and
\begin{align*}
 \| u - \RitzLODell u  \|_{L^2(\Omega)} \,\,\lesssim \,\,
 (1+\beta)^{3/2} \left( h^4\,\kappa^4 +  (1+ \kappa \beta h)^2 \, \theta_{\beta}^{2\ell}
 + (1+\beta)^{-1/2} \, (1+ \kappa \beta h)  \,h \, \kappa\, \theta_{\beta}^{\ell} \right).
\end{align*}
\end{lemma}

\begin{proof}
We proceed as in the proof of Lemma \ref{lem_ritz_proj_hk} and express $u \in H^1(\Omega)$ as
	\begin{align*}
\abeta (u,w)  &=  (\, (\beta - \vert u \vert^2 +1) u , w )_{L^2(\Omega)}
\qquad
\mbox{for all } w \in H^1(\Omega),
\end{align*}
where $(\beta - \vert u \vert^2 +1) u \in H^2(\Omega)$ and $| (\beta - \vert u \vert^2 +1) u |_{H^s(\Omega)} \lesssim (\beta \,  + 1) \kappa^s$ for $s=0,1,2$. Hence, Lemma \ref{lemma:estimatesRitzprojectionLOD-ell}  yields
\begin{eqnarray*}
\lefteqn{ \Vert u - \RitzLODell u \Vert_{H^1_\kappa(\Omega)} } \\
&\lesssim&
h \,\kappa \, \| (\id - \pi_h) \left((\beta - \vert u \vert^2 +1) u\right) \|_{L^2(\Omega)}  +  \theta_{\beta}^{\ell} \, (1+ \kappa \beta h) \, \left( \Vert u \Vert_{H^1_\kappa(\Omega)} + \sqrt{\beta} \, \Vert u \Vert_{L^2(\Omega)} \right) \\
  &\lesssim&
 (1+\beta) \, h^3\,\kappa^3 + (1+\beta)^{1/2} \, (1+ \kappa \beta h)\, \theta_{\beta}^{\ell}.
\end{eqnarray*}
The $L^2$-estimate follows readily with the second part of Lemma \ref{lemma:estimatesRitzprojectionLOD-ell} as
\begin{eqnarray*}
\lefteqn{ \| u - \RitzLODell u  \|_{L^2(\Omega)} }\\
&\lesssim& (1+\beta)^{1/2}  \left(  h \, \kappa\,  + \theta_{\beta}^{\ell} \, (1+ \kappa \beta h) \right)
\left( (\beta \,  + 1) \, h^3\,\kappa^3 +  \theta_{\beta}^{\ell} \, (1+ \kappa \beta h) \, (1+\beta)^{1/2}  \right) \\
&\lesssim& (1+\beta)^{3/2} h^4\,\kappa^4 \,+\, (1+\beta)^{3/2}  \theta_{\beta}^{2\ell} \, (1+ \kappa \beta h)^2 \,+\, (1+\beta)  h \, \kappa\, \theta_{\beta}^{\ell} \, (1+ \kappa \beta h) .
\end{eqnarray*}
\end{proof}
We are now prepared to formulate the final approximation result in the fully localized setting.
\begin{theorem}
\label{abstract-error-estimate-idealLOD-ell}
Assume \ref{A1}-\ref{A6} and let $u_{h,\ell}^{\LOD} \in \Vlod$ denote a minimizer of the Ginzburg-Landau energy in the LOD space as in \eqref{LOD-minimizer-ell-def}. Then, for all sufficiently small mesh sizes $h$, in particular
\begin{align*}
h \lesssim \, \min\left\{  (1+\beta)^{-1/2} , \beta^{-1} \right\} \kappa^{-1},
\end{align*}
there exists a unique minimizer $u \in H^1(\Omega)$ with $(u_{h}^{\LOD} , \ci u)_{L^2(\Omega)}=0$ (and $(u_{h}^{\LOD} , u)_{L^2(\Omega)}\ge 0$) and such that
\begin{eqnarray*}
\| u - u_{h}^{\LOD} \|_{H^1_\kappa(\Omega)} &\lesssim& \left( 1 + h \, \rho(\kappa) \kappa + \theta_{\beta}^{\ell} \, \rho(\kappa) \sqrt{1+\beta} \right) \, \| u - \RitzLODell \|_{H^1_\kappa(\Omega)}.
\end{eqnarray*}
Furthermore, with Lemma \ref{lem_ritz_proj_hk-localized} we have
\begin{eqnarray*}
\| u - u_{h}^{\LOD} \|_{H^1_\kappa(\Omega)} &\lesssim& \left( 1 + h \, \rho(\kappa) \kappa + \theta_{\beta}^{\ell} \, \rho(\kappa) \sqrt{1+\beta} \right) \,
\left( (1+\beta) \, h^3\,\kappa^3 +  (1+\beta^{1/2}) \, \theta_{\beta}^{\ell} \right).
\end{eqnarray*}
In the regime $\beta \lesssim 1$, the estimate simplifies to
\begin{eqnarray*}
\| u - u_{h}^{\LOD} \|_{H^1_\kappa(\Omega)} &\lesssim&
h^3\,\kappa^3 \,+\, \theta_{\beta}^{\ell} \,+\, \rho(\kappa)  ( h\,\kappa + \theta_{\beta}^{\ell} ) \, ( h^3\,\kappa^3 + \theta_{\beta}^{\ell} ).
\end{eqnarray*}
Finally, if $\ell \gtrsim 3 \log( \theta_{\beta} )^{-1} \log( h \kappa)$ we have
\begin{eqnarray*}
\| u - u_{h}^{\LOD} \|_{H^1_\kappa(\Omega)} &\lesssim&
h^3\,\kappa^3 \,+\, h^4\,\kappa^4 \, \rho(\kappa).
\end{eqnarray*}
\end{theorem}
\begin{proof}
The proof is again based on \cite[Proposition 5.6]{DoeHe23} with analogous arguments as in the proof of Lemma \ref{abstract-error-estimate-idealLOD}.
\end{proof}
The theorem shows that if the localization parameter $\ell$ is sufficiently large (with a constraint only depending on the size of $h \kappa$), we recover the same approximation properties as for the ideal method in Theorem \ref{theorem-ideal-LOD}. Furthermore, we again observe that large values of $\beta$ are not desirable as it appears as a multiplicative pollution factor in the error estimates and imposes further constraints for $\ell$, which becomes at least $\ell \gtrsim 3 \log( \theta_{\beta} )^{-1} \log( h \kappa (1+\beta)^{1/6})$.

\section{Numerical Experiments}
\label{section:numerical_experiments}

In this section we verify the theoretical findings with numerical experiments. Therefore, we first describe the specific problem setting for our tests. The different tests are divided according to the influence of the examined parameters. The implementation of the experiments is available as a MATLAB code on \url{https://github.com/cdoeding/GinzburgLandauLOD}.

\subsection{Problem setting}
According to the definition of $\Vlodell$ in \eqref{def-VLOD-ell}, the LOD space is constructed from linear Lagrange finite elements on a quasi-uniform triangular mesh $\mathcal{T}_h$ with mesh size $h$. On $\Vlodell$, we compute a minimizer $u_{h,\ell}^{\LOD}$ fulfilling \eqref{LOD-minimizer-ell-def} with a simple gradient descent method as described in \cite{DoeHe23} with a tolerance of $\delta =10^{-10}$ w.r.t. the difference in energy of two consecutive approximations. We fix the domain $\Omega = [0,1]^2 \subset \R^2$ and consider the magnetic potential
$$
\bfA(x,y):= \sqrt{2} \begin{pmatrix} \sin (\pi x) \cos (\pi y)\\ -\cos (\pi x) \sin (\pi y) \end{pmatrix}
$$
that fulfills our assumption \ref{A2}. For a practical implementation of the LOD we need two different meshes. The aforementioned \quotes{coarse} mesh $\mathcal{T}_h$ and an additional (quasi-uniform) \quotes{fine} mesh $\mathcal{T}_{h_{\fine}}$ on which we solve the local corrector problems \eqref{definition-CT-ell} whose solutions are required to assemble $\Vlodell$ via \eqref{def-VLOD-ell}. For more details on the implementation, we refer to \cite{EHMP19}. The fine mesh size is consistently chosen as $h_{\fine}= 2^{-10}$ throughout all the experiments, which is significantly smaller than the coarse meshes sizes which fulfill $h\ge 2^{-6}$. Furthermore, all reference minimizers (for the various $\kappa$ values) are computed in the $\mathbb{P}^1$-FE space $V_{h_{\fine}}$. We shall denote these reference minimizers by
\begin{align*}
u_{\text{\tiny ref}}^{\FEM} := \underset{v \in V_{h_{\fine}}}{\mbox{arg\,min}} \, E(v).
\end{align*}
 In order to ensure that our numerically computed errors are not polluted by inconsistent phase factors (due to the missing uniqueness of minimizers), we always adjust the phase of $u_{h,\ell}^{\LOD}$ such that $\Re \int_{\Omega} u_{\text{\tiny ref}}^{\FEM} (u_{h,\ell}^{\LOD})^{\ast}\,\,\mbox{d}x  \ge 0$ and $\Im \int_{\Omega} u_{\text{\tiny ref}}^{\FEM} (u_{h,\ell}^{\LOD})^{\ast}\,\,\mbox{d}x = 0$. This can be easily achieved by multiplying $u_{h,\ell}^{\LOD}$ with $\tfrac{\alpha}{|\alpha|}$ where $\alpha:= \int_{\Omega} u_{\text{\tiny ref}}^{\FEM} (u_{h,\ell}^{\LOD})^{\ast}\,\,\mbox{d}x$. We exploit this in the error computation without further mentioning. All errors are measured in the $H^1_\kappa$-norm and we mainly study $\varepsilon^{\LOD}_h := \| u_{h,\ell}^{\LOD} - u_{\text{\tiny ref}}^{\FEM} \|_{H^1_{\kappa}(\Omega)}$ and $\varepsilon^{\FEM}_h := \| u_{h}^{\FEM} - u_{\text{\tiny ref}}^{\FEM} \|_{H^1_{\kappa}(\Omega)}$,
where the coarse mesh size in our experiments varies between $h = 2^{-2}$ and $h = 2^{-6}$. We also consider various parameter constellations for $\kappa$, $\beta$ and $\ell$.

Figure \ref{fig:vortices_reference} shows the reference solution $u_{\text{\tiny ref}}^{\FEM}$, and in particular the  Abrikosov vortex lattice, for the values $\kappa=8,12,16,20,32$ and Table \ref{tab:energy_reference_and_ev} lists the corresponding minimal energy values. We further compute the spectrum of the second derivative $E''(u_{\text{\tiny ref}}^{\FEM})$ to verify assumption \ref{A4} numerically. The five smallest eigenvalues are also shown in Table \ref{tab:energy_reference_and_ev} and we can clearly identify the zero eigenvalue $\lambda_1$ followed by four positive eigenvalues $\lambda_2,\dots,\lambda_5$. Finally, note that the vortex pattern for $\kappa = 20$ is remarkably non-symmetric. Due to the symmetry of the vector potential and the domain we see that rotations by an angle of $\pi/2$ are again minimizers of the Ginzburg--Landau energy. Thus, there are four isolated minimizers of the same energy with this significant vortex pattern, and indeed we have found them in numerical experiments.\\[-0.5em]

\begin{figure}[h]
\centering
\begin{minipage}{0.3\textwidth}
\includegraphics[scale=0.25]{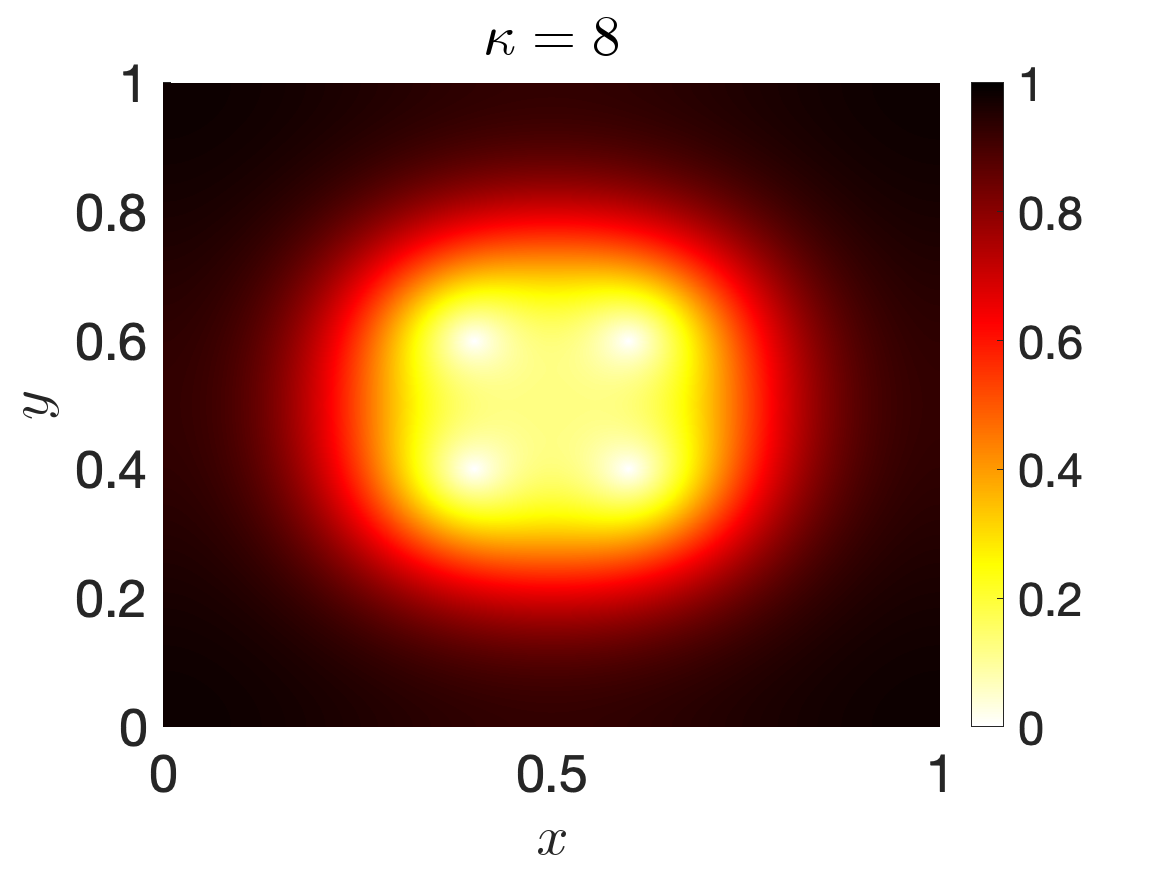}
\end{minipage}
\begin{minipage}{0.3\textwidth}
\includegraphics[scale=0.25]{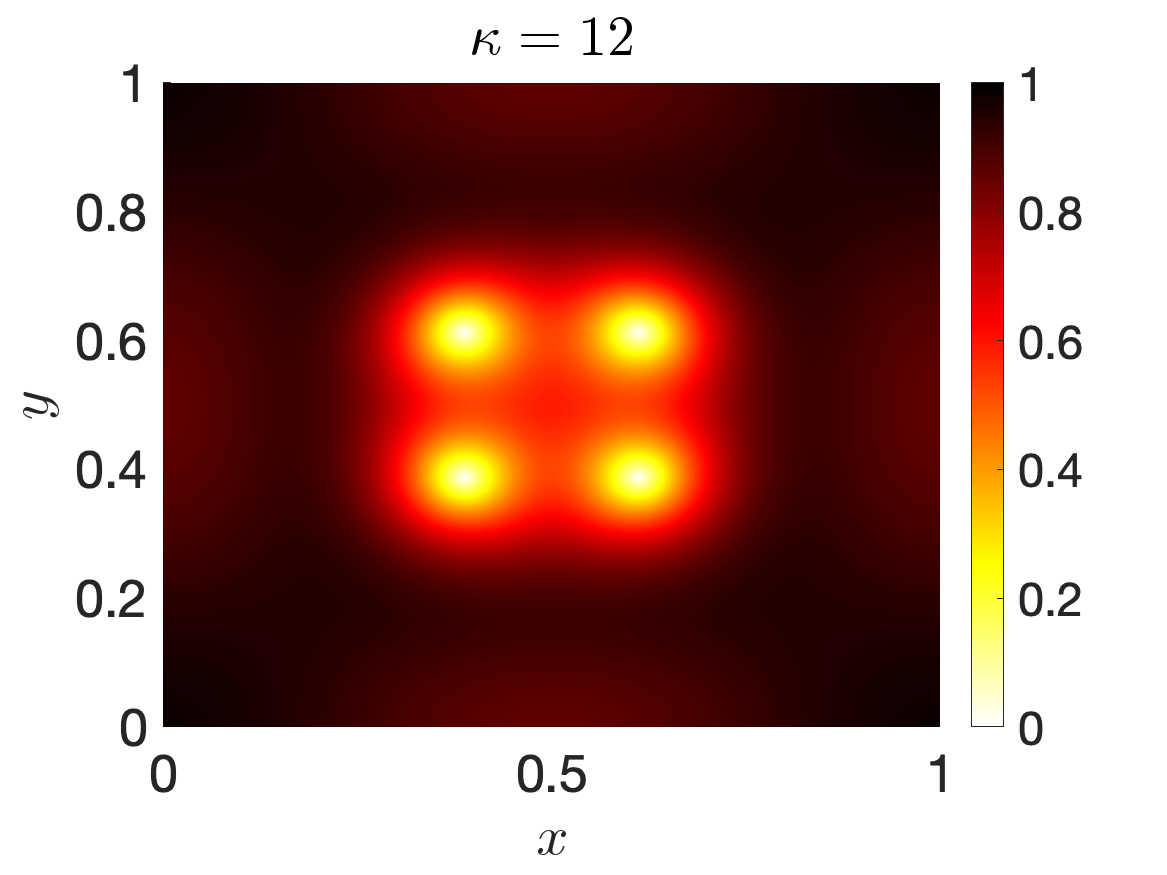}
\end{minipage}
\begin{minipage}{0.3\textwidth}
\includegraphics[scale=0.25]{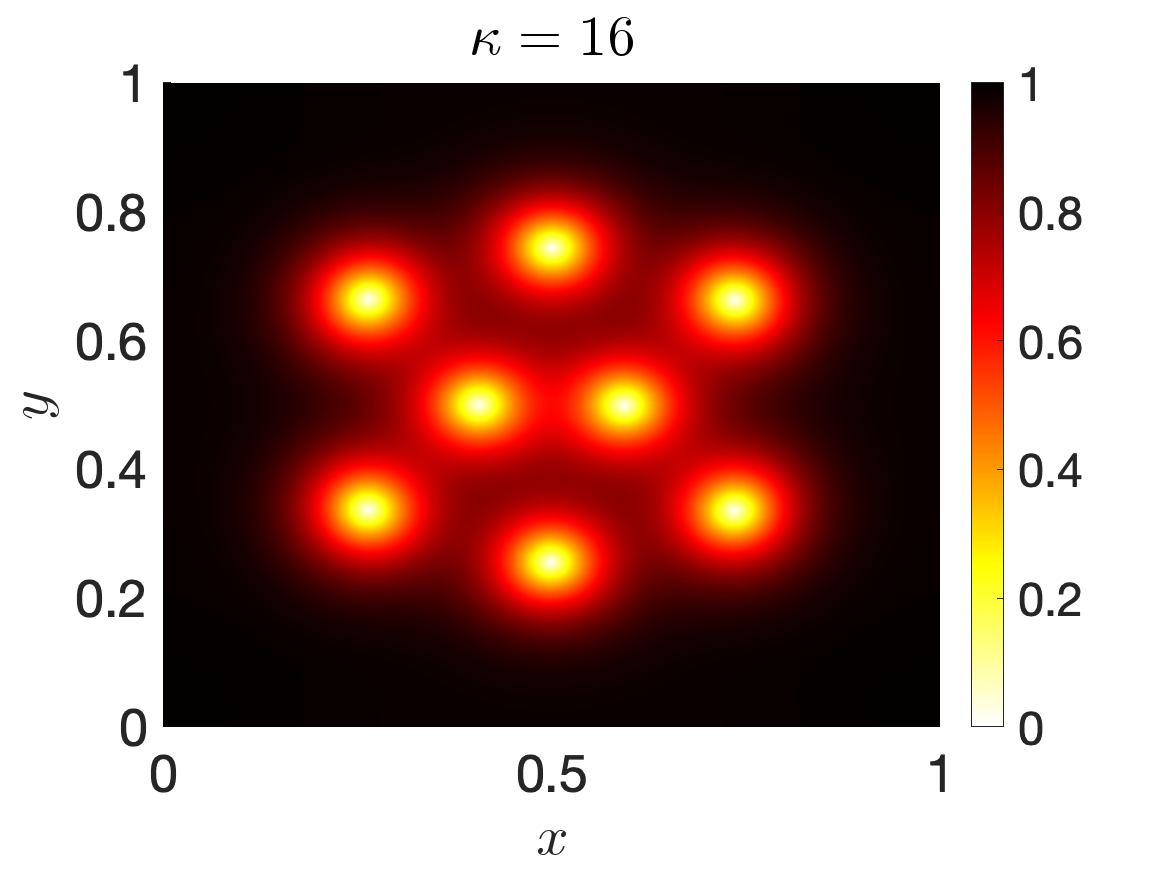}
\end{minipage} \\
\begin{minipage}{0.3\textwidth}
\includegraphics[scale=0.25]{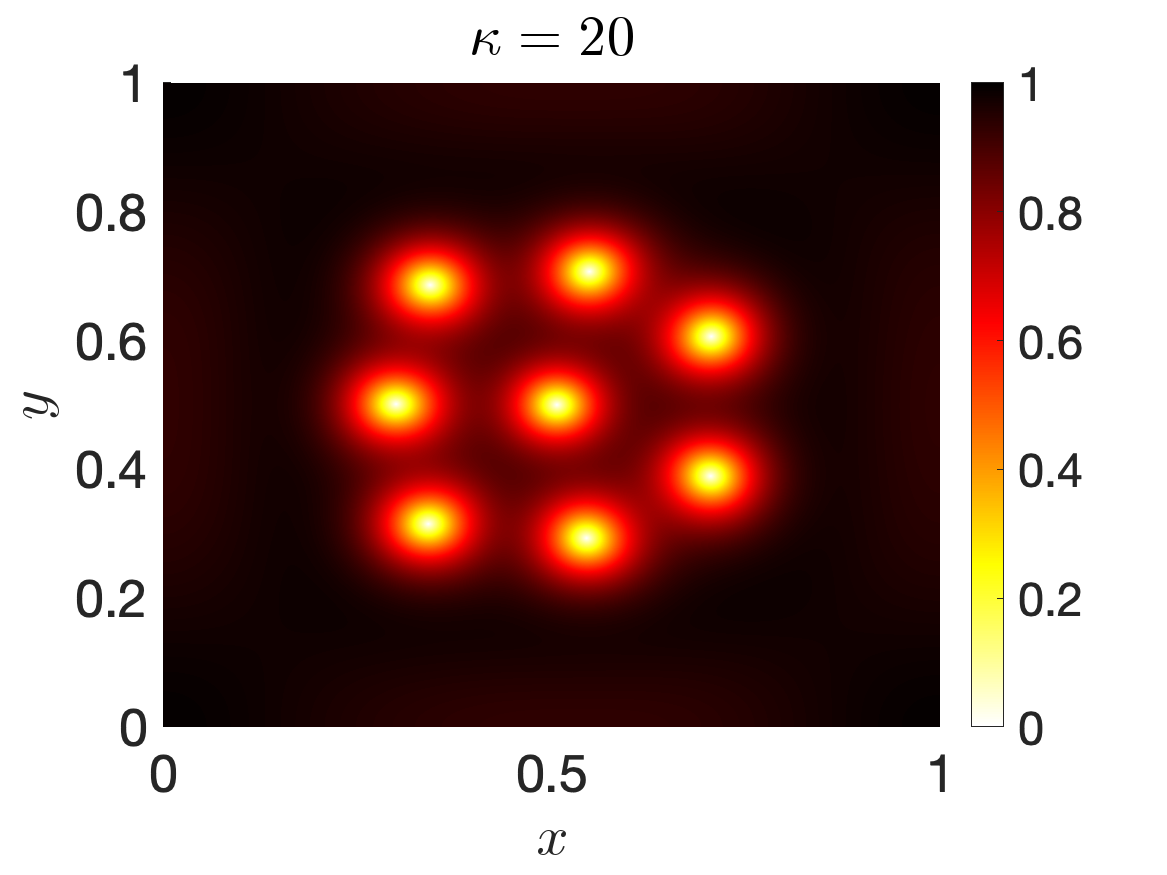}
\end{minipage}
\begin{minipage}{0.3\textwidth}
\includegraphics[scale=0.25]{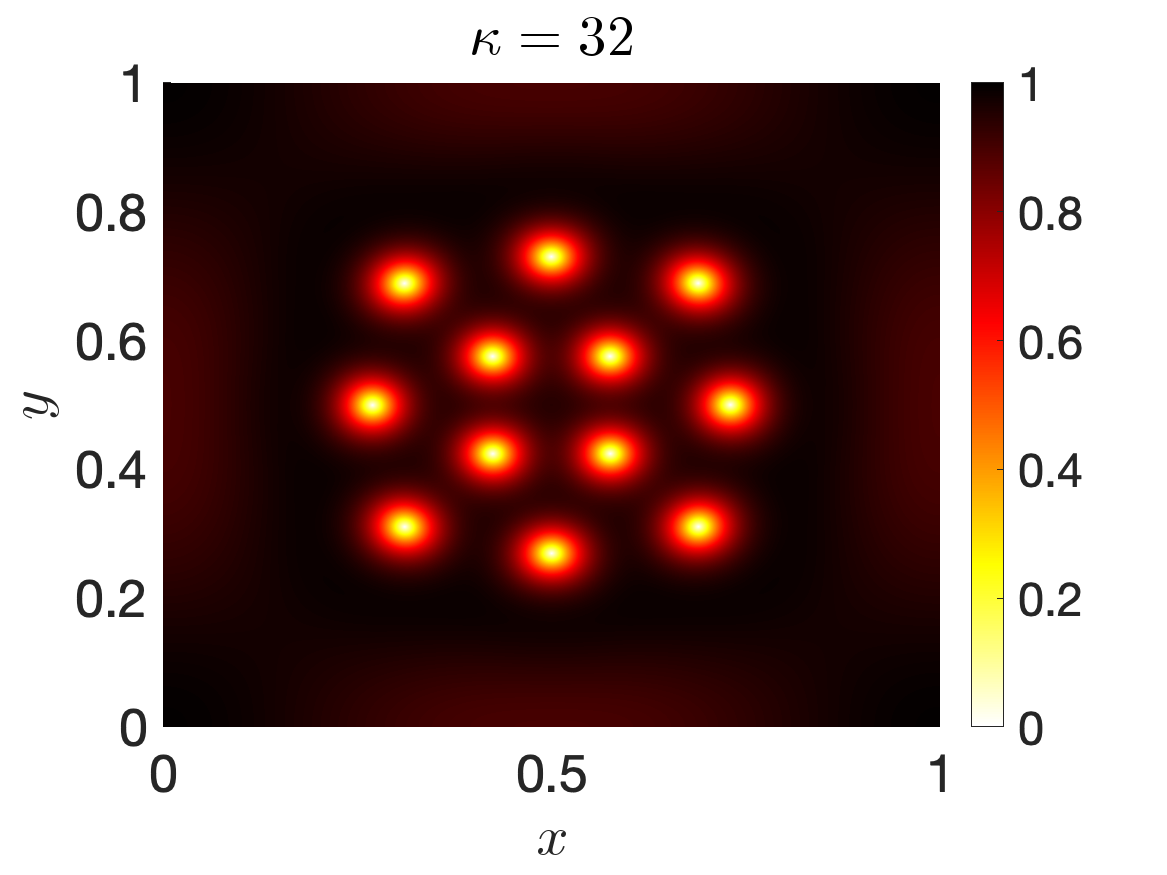}
\end{minipage}
	\caption{Vortex patterns of the reference solution $|u_{\text{\tiny ref}}^{\FEM}|$ with $\kappa=8,12,16,20,32$.}
	\label{fig:vortices_reference}
\end{figure}

\renewcommand*{\arraystretch}{1.2}
\begin{table}[h]
	\centering

	\begin{tabular}[h]{ |c|c|c|c|c|c| }
	\hline
		$\kappa$ & 8 & 12 & 16 & 20 & 32 \\
		\hlinewd{1.2pt}
		$E(u_{\text{\tiny ref}}^{\FEM})$ 
		& 1.2853e-01%
		& 1.0392e-01 %
		& 8.8528e-02 %
		& 7.5310e-02 %
		& 6.3582e-02 \\ %
		\hlinewd{1.2pt} 
		$\rho(\kappa)^{-1}$ 
		& 1.7124e-02 
		& 1.2444e-02 
		& 1.1930e-03 
		& 4.0621e-05 
		& 3.8279e-04 \\ 
		\hlinewd{1.2pt}
		$\lambda_1$ & $\sim 10^{-12}$ & $\sim 10^{-12}$ &  $\sim 10^{-10}$ & $\sim 10^{-8}$	& $\sim 10^{-12}$ \\ %
		\hline $\lambda_2$ & 4.1479e-02 & 2.2205e-02 &  1.9948e-03 & 6.0232e-05	& 5.6319e-04\\
		\hline $\lambda_3$ & 4.1480e-02 & 2.2577e-02 &  2.7939e-03 & 7.3989e-03	& 2.7402e-03\\
		\hline $\lambda_4$ & 1.1784e-01 & 2.2578e-02 &  9.2460e-03 & 7.4579e-03	& 2.7410e-03\\
		\hline $\lambda_5$ & 1.2045e-01 & 6.9642e-02 &  9.6807e-03 & 1.7718e-02	& 7.4593e-03\\
		\hline
	\end{tabular}
	\caption{Minimal energy values $E(u_{\text{\tiny ref}}^{\FEM})$, the coercivity constant $\rho(\kappa)^{-1}$ and the five smallest eigenvalues $\lambda_1,\dots,\lambda_5$ of $E''(u_{\text{\tiny ref}}^{\FEM})$ for the reference solution $u_{\text{\tiny ref}}^{\FEM}$ for $\kappa=8,12,16,20,32$.}
	\label{tab:energy_reference_and_ev}
\end{table}
\renewcommand*{\arraystretch}{1}

Last regarding the problem setting, we also checked the size of the coercivity constant $\rho(\kappa)^{-1}$ from Proposition \ref{coercivity_secE_u}. Note that the constant can be computed by solving the eigenvalue problem for $v_i \in H^1(\Omega)$ and $\mu_i \in \mathbb{R}$ such that
\begin{align*}
\langle E^{\prime\prime}(u) v_i , w \rangle = \mu_i \, ( v_i , w )_{H^1_{\kappa}(\Omega)}\qquad \mbox{for all } w \in H^1(\Omega).
\end{align*}
In this case, the smallest eigenvalue is $\mu_1 = 0$ and belongs to the eigenfunction $v_1 = \ci u$, whereas the second smallest eigenvalue $\mu_2$ is just the coercivity constant, i.e. $\mu_2 = \rho(\kappa)^{-1}$. The values for different $\kappa$ are shown in Table \ref{tab:energy_reference_and_ev} and we observe that they roughly behave as $\rho(\kappa) \sim \kappa^{\alpha}$ for $\alpha \approx 3.7$. Hence, we see the polynomial growth of $\rho(\kappa)$ conjectured earlier.

\subsection{Convergence with respect to $h$ and $\kappa$}
\label{subsection:numerics:hvskappa}
We start with investigating the error behavior with respect to $h$ and $\kappa$ as predicted in Theorem \ref{abstract-error-estimate-idealLOD-ell}. For the construction of the LOD space we use the stabilization parameter $\beta = 1$ and the fixed localization parameter $\ell=10$. The latter is selected sufficiently large so that the error $\varepsilon_h^{\LOD}$ is expected to behave as $\mathcal{O}(\kappa^3\,h^3 + \rho(\kappa)\kappa^4 \, h^4)$. 

The results are depicted in Figure \ref{fig:weighted_error_beta1} for $\kappa=8,12,16,20,32$, where we observe three different error regimes. First, there is a small preasymptotic regime (which seems even absent for $\kappa=8$), where we do not observe any error reduction. In this regime, the error component $\rho(\kappa)\kappa^4 \, h^4$ seems to be of order $1$ and is dominating (cf. Remark \ref{remark:preasymptotic}). We can identify differences in the regime for increasing $\kappa$-values but it is too small to draw any reliable conclusions. Error reduction is first observed from $h=2^{-2}$ to $h=2^{-3}$. After that, i.e., in the second error regime, the error follows the expected rate of $\kappa^3\,h^3$. To stress the uniform convergence in terms of $\kappa h$, Figure \ref{fig:weighted_error_beta1} also contains a plot where the $H^1_{\kappa}$-error is scaled with $\kappa^{-3}$ (right picture). This shows that all error curves are almost on top of each other and follow the line $h^3$. This confirms both the third order convergence in $h$, but also the asymptotically correct scaling of the error with $\kappa^3$. 
Finally, we can also identify a third error regime where the error flattens out again. On the one hand, this is partially caused by the localization parameter $\ell$, where we start to see the influence of the localization error $\theta_{\beta}^{\ell}$. However, we also tested larger values of $\ell$ and we could not observe any considerable improvement. For that reason, it is most likely that the flattening is caused by the fine scale discretization error. More precisely, we probably see the error of the reference solution $\| u - u_{\text{\tiny ref}}^{\FEM} \|_{H^1_{\kappa}(\Omega)}$ and the fine scale discretization error for the corrector problems, which are of order $\mathcal{O}(h_{\fine})$ and hence no longer negligible compared to $\mathcal{O}(h^3)$. In the spirit of Theorem \ref{theorem-ideal-LOD}, we can expect that a fully-discrete error analysis leads, for sufficiently large $\ell$ and $\beta \lesssim 1$, to an estimate of the form
\begin{align*}
\| u_{h,\ell}^{\LOD} - u_{\text{\tiny ref}}^{\FEM}  \|_{H^1_{\kappa}(\Omega)}
\,\,\,\lesssim\,\,\, h_{\fine} \kappa \, (1 + \kappa \, \rho(\kappa) \,h_{\fine})
\,\,+\,\,  \kappa^3 \,h^3 \, (1 + \kappa \, \rho(\kappa) \,h),
\end{align*}
where the error is eventually dominated by $h_{\fine} \kappa$ in the regime $h^3 \lesssim h_{\fine} \ll \kappa^{-1}$.
Note that this is different to linear elliptic problems, where $\| u_{h,\ell}^{\LOD} - u_{\text{\tiny ref}}^{\FEM} \|_{H^1_{\kappa}(\Omega)}$ can be bounded independent of the fine scale $h_{\fine}$. For nonlinear problems, this is typically no longer possible and the fine scale error can indeed become visible. In other words, the LOD approximations for $h=2^{-6}$ are essentially as accurate as the reference solution obtained for $h_{\fine}=2^{-10}$.
To explore if the (small) preasymptotic regime in Figure \ref{fig:weighted_error_beta1} is indeed related to the term $\rho(\kappa)\kappa^4 \, h^4$, we carry out a second experiment, the results of which are shown in Figure \ref{fig:compbest}. There we compare the error $\varepsilon_h^{\LOD}$ with the best-approximation error (of the reference solution) given by
\begin{align}
\label{best-approx-error-Vlodell}
\varepsilon_{h}^{\mathrm{best}} := \inf_{v_{h,\ell}^{\LOD} \in \Vlodell}  \| u_{\text{\tiny ref}}^{\FEM} - v_{h,\ell}^{\LOD}  \|_{H^1_{\kappa}(\Omega)}.
\end{align}
Exemplary, we compare for the values $\kappa=8,12,16,20,32$. In view of Lemma \ref{lem_ritz_proj_hk}, the best-approximation error behaves like $\mathcal{O}(h^3 \kappa^3)$ without any $\mathcal{O}(\rho(\kappa)\kappa^4 \, h^4)$-pollution. Hence, we expect the errors to be slightly different in a short regime and later match each other in the asymptotic behavior. This is precisely what we observe in Figure \ref{fig:compbest}, where the best-approximation error almost instantly follows $h^3$ without preasymptotic effects and quickly aligns with $\varepsilon_h^{\LOD}$. Note that we still expect a preasymptotic regime (for both errors) in the range $h \gg \kappa^{-1}$, which however only became visibly detectable for $\kappa=32$.

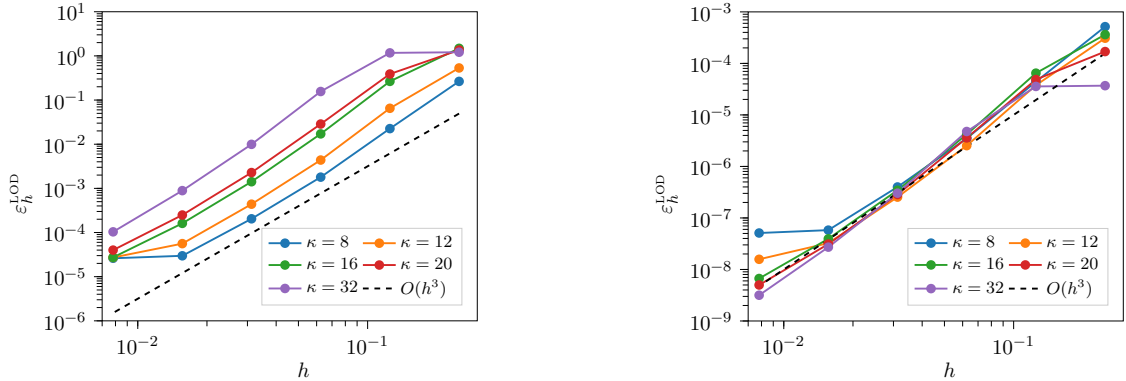
\begin{figure}[h]
	\definecolor{color0}{rgb}{0.12156862745098,0.466666666666667,0.705882352941177}
	\definecolor{color1}{rgb}{1,0.498039215686275,0.0549019607843137}
	\definecolor{color2}{rgb}{0.172549019607843,0.627450980392157,0.172549019607843}
	\definecolor{color3}{rgb}{0.83921568627451,0.152941176470588,0.156862745098039}
	\definecolor{color4}{rgb}{0.580392156862745,0.403921568627451,0.741176470588235}
	\definecolor{color5}{rgb}{0.549019607843137,0.337254901960784,0.294117647058824}
	\definecolor{color6}{rgb}{0.890196078431372,0.466666666666667,0.76078431372549}
	
	\begin{adjustbox}{width=0.4\linewidth}
		\begin{tikzpicture}
			\begin{axis}[
				legend cell align={left},
				legend style={
					fill opacity=0.8,
					draw opacity=1,
					text opacity=1,
					at={(0.96,0.05)},
					anchor=south east,
					draw=white!80!black
				},
				log basis x={2},
				log basis y={10},
				tick align=outside,
				tick pos=left,
				x grid style={white!69.0196078431373!black},
				xmin=1.4e-02, xmax=3e-01,
				xmode=log,
				xtick style={color=black},
				xlabel={$h$},
				xlabel style={xshift=0.25cm, yshift=-0.1cm},
				y grid style={white!69.0196078431373!black},
				ymin=1e-06, ymax=1e+01,
				ymode=log,
				ytick style={color=black},
				ylabel={$\varepsilon_h^{\LOD}$},
				ylabel style={xshift=-0.5cm, yshift=0.25cm},
				legend columns=2,
				]
				
				\addplot [line width = 1, color0, mark=*, mark size=2, mark options={solid}]
				table {%
					2.500000000000000e-01     2.646701063281616e-01
					1.250000000000000e-01     2.255503682609795e-02
					6.250000000000000e-02     1.795684303882613e-03
					3.125000000000000e-02     2.044910493346325e-04
					1.562500000000000e-02     2.970340831674449e-05
				};
				\addlegendentry{\footnotesize $\kappa = 8$}
				\addplot [line width = 1, color1, mark=*, mark size=2, mark options={solid}]
				table {%
					2.500000000000000e-01     3.997989078105458e-01
					1.250000000000000e-01     6.265805504326169e-02
					6.250000000000000e-02     4.460653056098238e-03
					3.125000000000000e-02     4.349410662409484e-04
					1.562500000000000e-02     5.560580600405315e-05
				};
				\addlegendentry{\footnotesize $\kappa = 12$}
				\addplot [line width = 1, color2, mark=*, mark size=2, mark options={solid}]
				table {%
					2.500000000000000e-01     1.364544026599912e+00
					1.250000000000000e-01     4.367392511317231e-01
					6.250000000000000e-02     1.713100970621161e-02
					3.125000000000000e-02     1.864924928745459e-03
					1.562500000000000e-02     9.205110724281587e-04
				};
				\addlegendentry{\footnotesize $\kappa = 16$}
				\addplot [line width = 1, color3, mark=*, mark size=2, mark options={solid}]
				table {%
					2.500000000000000e-01     1.390846386837328e+00
					1.250000000000000e-01     5.227967021619514e-01
					6.250000000000000e-02     2.938646712976997e-02
					3.125000000000000e-02     2.881889024011982e-03
					1.562500000000000e-02     2.642882865748880e-03
				};
				\addlegendentry{\footnotesize $\kappa = 20$}
				\addplot [line width = 1, color4, mark=*, mark size=2, mark options={solid}]
				table {%
					2.500000000000000e-01     1.400111402095018e+00
					1.250000000000000e-01     1.344524082761501e+00
					6.250000000000000e-02     4.284172920643262e-01
					3.125000000000000e-02     9.826353648021787e-03
					1.562500000000000e-02     8.953187995286539e-04
				};
				\addlegendentry{\footnotesize $\kappa = 32$}
				\addplot [line width = 1, black, dashed]
				table {%
					2.500000000000000e-01     4.941058844013093e-02
					1.250000000000000e-01     6.176323555016366e-03
					6.250000000000000e-02     7.720404443770458e-04
					3.125000000000000e-02     9.650505554713072e-05
					1.562500000000000e-02     1.206313194339134e-05
				};
				\addlegendentry{\footnotesize $O(h^3)$}
			\end{axis}
		\end{tikzpicture}
	\end{adjustbox}
	\hspace{2cm}
	\begin{adjustbox}{width=0.4\linewidth}
		\begin{tikzpicture}
			\begin{axis}[
				legend cell align={left},
				legend style={
					fill opacity=0.8,
					draw opacity=1,
					text opacity=1,
					at={(0.96,0.05)},
					anchor=south east,
					draw=white!80!black
				},
				log basis x={2},
				log basis y={10},
				tick align=outside,
				tick pos=left,
				x grid style={white!69.0196078431373!black},
				xmin=1.4e-02, xmax=3e-01,
				xmode=log,
				xtick style={color=black},
				xlabel={$h$},
				xlabel style={xshift=0.25cm, yshift=-0.1cm},
				y grid style={white!69.0196078431373!black},
				ymin=1e-09, ymax=1e-03,
				ymode=log,
				ytick style={color=black},
				ylabel={$\varepsilon_h^{\LOD}$},
				ylabel style={xshift=-0.5cm, yshift=0.25cm},
				legend columns=2,
				]
				
				\addplot [line width = 1, color0, mark=*, mark size=2, mark options={solid}]
				table {%
					2.500000000000000e-01     5.169338014221907e-04
					1.250000000000000e-01     4.405280630097257e-05
					6.250000000000000e-02     3.507195906020728e-06
					3.125000000000000e-02     3.993965807317040e-07
					1.562500000000000e-02     5.801446936864159e-08
				};
				\addlegendentry{\footnotesize $\kappa = 8$}
				\addplot [line width = 1, color1, mark=*, mark size=2, mark options={solid}]
				table {%
					2.500000000000000e-01     2.313651086866584e-04
					1.250000000000000e-01     3.626044852040607e-05
					6.250000000000000e-02     2.581396444501295e-06
					3.125000000000000e-02     2.517020059264748e-07
					1.562500000000000e-02     3.217928588197520e-08
				};
				\addlegendentry{\footnotesize $\kappa = 12$}
				\addplot [line width = 1, color2, mark=*, mark size=2, mark options={solid}]
				table {%
					2.500000000000000e-01     3.331406314941193e-04
					1.250000000000000e-01     1.066257937333308e-04
					6.250000000000000e-02     4.182375416555569e-06
					3.125000000000000e-02     4.553039376819968e-07
					1.562500000000000e-02     2.247341485420309e-07
				};
				\addlegendentry{\footnotesize $\kappa = 16$}
				\addplot [line width = 1, color3, mark=*, mark size=2, mark options={solid}]
				table {%
					2.500000000000000e-01     1.738557983546660e-04
					1.250000000000000e-01     6.534958777024392e-05
					6.250000000000000e-02     3.673308391221246e-06
					3.125000000000000e-02     3.602361280014978e-07
					1.562500000000000e-02     3.303603582186100e-07
				};
				\addlegendentry{\footnotesize $\kappa = 20$}
				\addplot [line width = 1, color4, mark=*, mark size=2, mark options={solid}]
				table {%
					2.500000000000000e-01     4.272800909713800e-05
					1.250000000000000e-01     4.103161873661807e-05
					6.250000000000000e-02     1.307425818067402e-05
					3.125000000000000e-02     2.998765151373836e-07
					1.562500000000000e-02     2.732296141139691e-08
				};
				\addlegendentry{\footnotesize $\kappa = 32$}
				\addplot [line width = 1, black, dashed]
				table {%
					2.500000000000000e-01     1.562500000000000e-04
					1.250000000000000e-01     1.953125000000000e-05
					6.250000000000000e-02     2.441406250000000e-06
					3.125000000000000e-02     3.051757812500000e-07
					1.562500000000000e-02     3.814697265625000e-08
				};
				\addlegendentry{\footnotesize $O(h^3)$}
			\end{axis}
		\end{tikzpicture}
	\end{adjustbox}
	\caption{Convergence of the $H^1_\kappa$-errors $\varepsilon_h^{\LOD}$ (left) and $\kappa^{-3} \varepsilon_h^{\LOD}$ (right) with $\beta=1$.}
	\label{fig:weighted_error_beta1}
\end{figure}

\begin{figure}[h]
	\definecolor{color0}{rgb}{0.12156862745098,0.466666666666667,0.705882352941177}
	\definecolor{color1}{rgb}{1,0.498039215686275,0.0549019607843137}
	\definecolor{color2}{rgb}{0.172549019607843,0.627450980392157,0.172549019607843}
	\definecolor{color3}{rgb}{0.83921568627451,0.152941176470588,0.156862745098039}
	\definecolor{color4}{rgb}{0.580392156862745,0.403921568627451,0.741176470588235}
	\definecolor{color5}{rgb}{0.549019607843137,0.337254901960784,0.294117647058824}
	\definecolor{color6}{rgb}{0.890196078431372,0.466666666666667,0.76078431372549}
	
	\begin{adjustbox}{width=0.4\linewidth}
		\begin{tikzpicture}
			\begin{axis}[
				legend cell align={left},
				legend style={
					fill opacity=0.8,
					draw opacity=1,
					text opacity=1,
					at={(0.96,0.05)},
					anchor=south east,
					draw=white!80!black
				},
				log basis x={2},
				log basis y={10},
				tick align=outside,
				tick pos=left,
				x grid style={white!69.0196078431373!black},
				xmin=1.4e-02, xmax=3e-01,
				xmode=log,
				xtick style={color=black},
				xlabel={$h$},
				xlabel style={xshift=0.25cm, yshift=-0.1cm},
				y grid style={white!69.0196078431373!black},
				ymin=1e-06, ymax=3,
				ymode=log,
				ytick style={color=black},
				ylabel={$\varepsilon_h^{\LOD}$ , $\varepsilon_h^{\mathrm{best}}$},
				ylabel style={xshift=-0.5cm, yshift=0.25cm},
				legend columns=2,
				legend pos = south east,
				]
				
				\addplot [line width = 1, color0, mark=*, mark size=2, mark options={solid}]
				table {%
					2.500000000000000e-01     2.646701063281616e-01
					1.250000000000000e-01     2.255503682609795e-02
					6.250000000000000e-02     1.795684303882613e-03
					3.125000000000000e-02     2.044910493346325e-04
					1.562500000000000e-02     2.970340831674449e-05
				};
				\addlegendentry{\footnotesize $\varepsilon_h^{\LOD}$, $\kappa = 8$}
				\addplot [line width = 1, color0, densely dotted, mark=square*, mark size=2, mark options={solid}]
				table {%
					2.500000000000000e-01     2.103258878632883e-01
					1.250000000000000e-01     2.173571814680273e-02
					6.250000000000000e-02     1.795402190898332e-03
					3.125000000000000e-02     2.044908194711811e-04
					1.562500000000000e-02     2.970339024444866e-05
				};
				\addlegendentry{\footnotesize $\varepsilon_h^{\mathrm{best}}$, $\kappa= 8$}
				\addplot [line width = 1, color1, mark=*, mark size=2, mark options={solid}]
				table {%
					2.500000000000000e-01     3.997989078105458e-01
					1.250000000000000e-01     6.265805504326169e-02
					6.250000000000000e-02     4.460653056098238e-03
					3.125000000000000e-02     4.349410662409484e-04
					1.562500000000000e-02     5.560580600405315e-05
				};
				\addlegendentry{\footnotesize $\varepsilon_h^{\LOD}$, $\kappa = 12$}
				\addplot [line width = 1, color1,densely dotted, mark=square*, mark size=2, mark options={solid}]
				table {%
					2.500000000000000e-01     3.192153255744047e-01
					1.250000000000000e-01     5.677171067980221e-02
					6.250000000000000e-02     4.430181054944957e-03
					3.125000000000000e-02     4.348965460099623e-04
					1.562500000000000e-02     5.537052778168657e-05
				};
				\addlegendentry{\footnotesize $\varepsilon_h^{\mathrm{best}}$, $\kappa = 12$}
				\addplot [line width = 1, color2, mark=*, mark size=2, mark options={solid}]
				table {%
					2.500000000000000e-01     1.364544026599912e+00
					1.250000000000000e-01     4.367392511317231e-01
					6.250000000000000e-02     1.713100970621161e-02
					3.125000000000000e-02     1.864924928745459e-03
					1.562500000000000e-02     9.205110724281587e-04
				};
				\addlegendentry{\footnotesize $\varepsilon_h^{\LOD}$, $\kappa = 16$}
				\addplot [line width = 1, color2,densely dotted, mark=square*, mark size=2, mark options={solid}]
				table {%
					2.500000000000000e-01     8.207401641858331e-01
					1.250000000000000e-01     1.796484607248328e-01
					6.250000000000000e-02     1.626367593899552e-02
					3.125000000000000e-02     1.399714073644773e-03
					1.562500000000000e-02     1.597288093350175e-04
				};
				\addlegendentry{\footnotesize $\varepsilon_h^{\mathrm{best}}$, $\kappa = 16$}
			\end{axis}
		\end{tikzpicture}
	\end{adjustbox}
\hspace{2cm}
	\begin{adjustbox}{width=0.4\linewidth}
		\begin{tikzpicture}
			\begin{axis}[
				legend cell align={left},
				legend style={
					fill opacity=0.8,
					draw opacity=1,
					text opacity=1,
					at={(0.96,0.05)},
					anchor=south east,
					draw=white!80!black
				},
				log basis x={2},
				log basis y={10},
				tick align=outside,
				tick pos=left,
				x grid style={white!69.0196078431373!black},
				xmin=1.4e-02, xmax=3e-01,
				xmode=log,
				xtick style={color=black},
				xlabel={$h$},
				xlabel style={xshift=0.25cm, yshift=-0.1cm},
				y grid style={white!69.0196078431373!black},
				ymin=1e-06, ymax=3,
				ymode=log,
				ytick style={color=black},
				ylabel={$\varepsilon_h^{\LOD}$ , $\varepsilon_h^{\mathrm{best}}$},
				ylabel style={xshift=-0.5cm, yshift=0.25cm},
				legend columns=2,
				]
				
				\addplot [line width = 1, color3, mark=*, mark size=2, mark options={solid}]
				table {%
					2.500000000000000e-01     1.390846386837328e+00
					1.250000000000000e-01     5.227967021619514e-01
					6.250000000000000e-02     2.938646712976997e-02
					3.125000000000000e-02     2.881889024011982e-03
					1.562500000000000e-02     2.642882865748880e-03
				};
				\addlegendentry{\footnotesize $\varepsilon_h^{\LOD}$, $\kappa = 20$}
				\addplot [line width = 1, color3,densely dotted, mark=square*, mark size=2, mark options={solid}]
				table {%
					2.500000000000000e-01     9.088239535476509e-01
					1.250000000000000e-01     2.630429182839991e-01
					6.250000000000000e-02     2.721439191118549e-02
					3.125000000000000e-02     2.297565934609675e-03
					1.562500000000000e-02     2.516536713673625e-04
				};
				\addlegendentry{\footnotesize $\varepsilon_h^{\mathrm{best}}$, $\kappa = 20$}
				\addplot [line width = 1, color4, mark=*, mark size=2, mark options={solid}]
				table {%
					2.500000000000000e-01     1.400111402095018e+00
					1.250000000000000e-01     1.344524082761501e+00
					6.250000000000000e-02     4.284172920643262e-01
					3.125000000000000e-02     9.826353648021787e-03
					1.562500000000000e-02     8.953187995286539e-04
				};
				\addlegendentry{\footnotesize $\varepsilon_h^{\LOD}$, $\kappa = 32$}
				\addplot [line width = 1, color4,densely dotted, mark=square*, mark size=2, mark options={solid}]
				table {%
					2.500000000000000e-01     1.125455690856485e+00
					1.250000000000000e-01     6.733552648167486e-01
					6.250000000000000e-02     9.906265839578014e-02
					3.125000000000000e-02     9.473538783011350e-03
					1.562500000000000e-02     8.841103179519604e-04
				};
				\addlegendentry{\footnotesize $\varepsilon_h^{\mathrm{best}}$, $\kappa = 32$, }
			\end{axis}
		\end{tikzpicture}
	\end{adjustbox}
	\caption{Comparison of $\varepsilon_h^{\LOD}$ (solid lines) and the $H^1_\kappa$-best-approximation error $\varepsilon_h^{\mathrm{best}}$ from \eqref{best-approx-error-Vlodell} (dashed lines) with $\beta=1$. }
	\label{fig:compbest}
\end{figure}
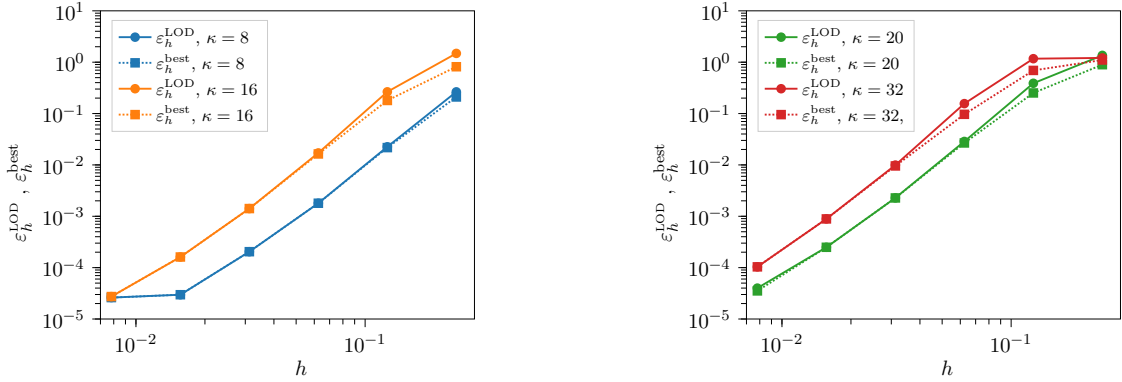

\subsection{Influence of the stabilization $\beta$}
In the previous subsection, we used $\beta=1$ for the construction of the LOD-space, where we recall that $\beta$ influences the element correctors $\mathcal{C}_{T,\ell} v_h$, defined in \eqref{definition-CT-ell}. According to the error estimates in Section \ref{subsection:final-estimates-full-LOD}, small values for $\beta$ should be preferable, since the error estimates become best for $\beta=0$. In the following, we compare the choices $\beta=0$ and $\beta=1$. As in the previous section, we keep $\ell=10$. \\
In the left picture of Figure \ref{fig:weighted_error_beta0}, the $H^1_\kappa$-error $\varepsilon_h^{\LOD}$ is shown for $\beta=0$ and different $\kappa$ values. Essentially, we make the same observations regarding the error behavior as in Section \ref{subsection:numerics:hvskappa} for $\beta=1$ (cf. Figure \ref{fig:weighted_error_beta1}). This is consistent with the analytical results in Theorem \ref{abstract-error-estimate-idealLOD-ell}.

A direct comparison of the $H^1_\kappa$-errors for both stabilization parameters is exemplary shown in the right picture of Figure \ref{fig:weighted_error_beta0} for $\kappa=8,\, 12, \,16$. The graph shows, that the errors for the same $\kappa$ and the different stabilization parameters just have small deviations for every mesh size $h$. However, the figure also supports our predictions, that small value of $\beta$ are preferable and that we get the visibly better errors for $\beta = 0$. In particular, the choice $\beta = 0$ is advantageous to capture the vortex lattice of the minimizers as we further emphasize in Section \ref{subsection:comparison-FEM-LOD} below.
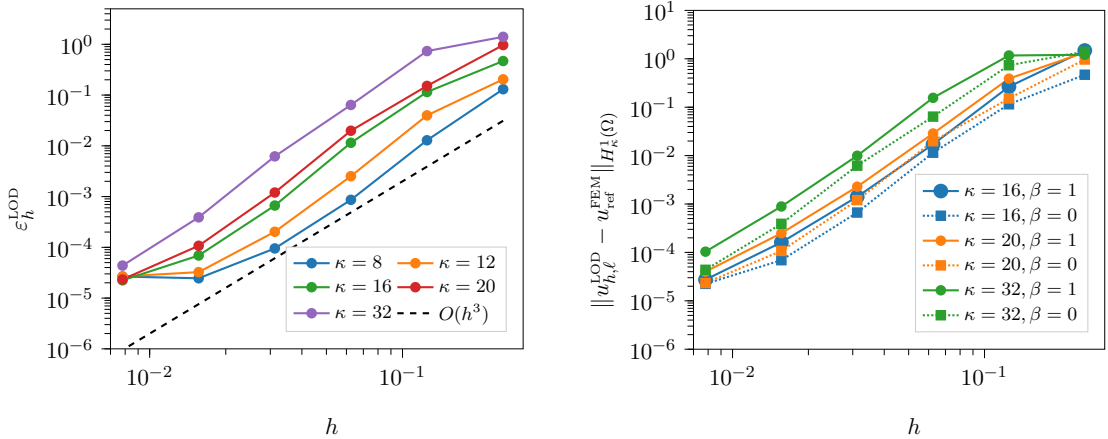
\begin{figure}[h]
	\definecolor{color0}{rgb}{0.12156862745098,0.466666666666667,0.705882352941177}
	\definecolor{color1}{rgb}{1,0.498039215686275,0.0549019607843137}
	\definecolor{color2}{rgb}{0.172549019607843,0.627450980392157,0.172549019607843}
	\definecolor{color3}{rgb}{0.83921568627451,0.152941176470588,0.156862745098039}
	\definecolor{color4}{rgb}{0.580392156862745,0.403921568627451,0.741176470588235}
	\definecolor{color5}{rgb}{0.549019607843137,0.337254901960784,0.294117647058824}
	\definecolor{color6}{rgb}{0.890196078431372,0.466666666666667,0.76078431372549}
	
	\centering
	\begin{adjustbox}{width=0.44\linewidth}
		\begin{tikzpicture}
			\begin{axis}[
				legend cell align={left},
				legend style={
					fill opacity=0.8,
					draw opacity=1,
					text opacity=1,
					at={(0.96,0.05)},
					anchor=south east,
					draw=white!80!black
				},
				log basis x={2},
				log basis y={10},
				tick align=outside,
				tick pos=left,
				x grid style={white!69.0196078431373!black},
				xmin=1.4e-02, xmax=3e-01,
				xmode=log,
				xtick style={color=black},
				xlabel={$h$},
				xlabel style={xshift=0.25cm, yshift=-0.5cm},
				y grid style={white!69.0196078431373!black},
				ymin=1e-06, ymax=0.5e+01,
				ymode=log,
				ytick style={color=black},
				ylabel={$\varepsilon_h^{\LOD}$},
				ylabel style={xshift=-0.5cm, yshift=0.25cm},
				legend columns=2,
				]
				
				\addplot [line width = 1, color0, mark=*, mark size=2, mark options={solid}]
				table {%
					2.500000000000000e-01    1.304706007964277e-01
					1.250000000000000e-01     1.285678514182162e-02
					6.250000000000000e-02     8.629814361590047e-04
					3.125000000000000e-02    9.509590395142185e-05
					1.562500000000000e-02     2.455832305304613e-05
				};
				\addlegendentry{\footnotesize $\kappa = 8$}
				\addplot [line width = 1, color1, mark=*, mark size=2, mark options={solid}]
				table {%
					2.500000000000000e-01     1.756487646700691e-01
					1.250000000000000e-01     3.883106715390972e-02
					6.250000000000000e-02     2.731329506493729e-03
					3.125000000000000e-02     2.080956954132298e-04
					1.562500000000000e-02     3.298224497909854e-05
				};
				\addlegendentry{\footnotesize $\kappa = 12$}
				\addplot [line width = 1, color2, mark=*, mark size=2, mark options={solid}]
				table {%
					2.500000000000000e-01     5.906680133984222e-01
					1.250000000000000e-01     3.878070710202812e-01
					6.250000000000000e-02     1.156764661862669e-02
					3.125000000000000e-02     1.400934613354597e-03
					1.562500000000000e-02     9.091779250625508e-04
				};
				\addlegendentry{\footnotesize $\kappa = 16$}
				\addplot [line width = 1, color3, mark=*, mark size=2, mark options={solid}]
				table {%
					2.500000000000000e-01     1.016021025248976e+00
					1.250000000000000e-01     3.388261637287355e-01
					6.250000000000000e-02     2.022473326609193e-02
					3.125000000000000e-02     2.123920041558945e-03
					1.562500000000000e-02     2.633302146863045e-03
				};
				\addlegendentry{\footnotesize $\kappa = 20$}
				\addplot [line width = 1, color4, mark=*, mark size=2, mark options={solid}]
				table {%
					2.500000000000000e-01     1.392752149813496e+00
					1.250000000000000e-01     5.519839882804088e-01
					6.250000000000000e-02     4.463368544846230e-01
					3.125000000000000e-02     6.161786715740432e-03
					1.562500000000000e-02     4.144399282156794e-04
				};
				\addlegendentry{\footnotesize $\kappa = 32$}
				\addplot [line width = 1, black, dashed]
				table {%
					2.500000000000000e-01     3.117597367138874e-02
					1.250000000000000e-01     3.896996708923593e-03
					6.250000000000000e-02     4.871245886154491e-04
					3.125000000000000e-02     6.089057357693114e-05
					1.562500000000000e-02     7.611321697116392e-06
				};
				\addlegendentry{\footnotesize $O(h^3)$}
			\end{axis}
		\end{tikzpicture}
	\end{adjustbox}
	\hspace{15pt}
	\centering
	\begin{adjustbox}{width=0.44\linewidth}
		\begin{tikzpicture}
			\begin{axis}[
				legend cell align={left},
				legend style={
					fill opacity=0.8,
					draw opacity=1,
					text opacity=1,
					at={(0.96,0.05)},
					anchor=south east,
					draw=white!80!black
				},
				log basis x={2},
				log basis y={10},
				tick align=outside,
				tick pos=left,
				x grid style={white!69.0196078431373!black},
				xmin=1.4e-02, xmax=3e-01,
				xmode=log,
				xtick style={color=black},
				xlabel={$h$},
				xlabel style={xshift=0.25cm, yshift=-0.5cm},
				y grid style={white!69.0196078431373!black},
				ymin=1.5e-05, ymax=4e+01,
				ymode=log,
				ytick style={color=black},
				ylabel={$\| u_{h,\ell}^{\LOD} - u_{\text{\tiny ref}}^{\FEM} \|_{H^1_{\kappa}(\Omega)}$},
				ylabel style={xshift=-0.5cm, yshift=0.25cm},
				legend columns=2,
				legend pos = north west,
				]
				\addplot [line width = 1, color0, mark=*, mark size=3, mark options={solid}]
				table {%
					2.500000000000000e-01     2.646701063281616e-01
					1.250000000000000e-01     2.255503682609795e-02
					6.250000000000000e-02     1.795684303882613e-03
					3.125000000000000e-02     2.044910493346325e-04
					1.562500000000000e-02     2.970340831674449e-05
				};
				\addlegendentry{\footnotesize $\kappa = 8,\beta=1$}
				\addplot [line width = 1, color0,densely dotted, mark=square*, mark size=2, mark options={solid}]
				table {%
					2.500000000000000e-01    1.304706007964277e-01
					1.250000000000000e-01     1.285678514182162e-02
					6.250000000000000e-02     8.629814361590047e-04
					3.125000000000000e-02    9.509590395142185e-05
					1.562500000000000e-02     2.455832305304613e-05
				};
				\addlegendentry{\footnotesize $\kappa = 8,\beta=0$}
				\addplot [line width = 1, color1, mark=*, mark size=3, mark options={solid}]
				table {%
					2.500000000000000e-01     3.997989078105458e-01
					1.250000000000000e-01     6.265805504326169e-02
					6.250000000000000e-02     4.460653056098238e-03
					3.125000000000000e-02     4.349410662409484e-04
					1.562500000000000e-02     5.560580600405315e-05
				};
				\addlegendentry{\footnotesize $\kappa = 12,\beta=1$}
				\addplot [line width = 1, color1,densely dotted, mark=square*, mark size=2, mark options={solid}]
				table {%
					2.500000000000000e-01     1.756487646700691e-01
					1.250000000000000e-01     3.883106715390972e-02
					6.250000000000000e-02     2.731329506493729e-03
					3.125000000000000e-02     2.080956954132298e-04
					1.562500000000000e-02     3.298224497909854e-05
				};
				\addlegendentry{\footnotesize $\kappa = 12,\beta=0$}
				\addplot [line width = 1, color2, mark=*, mark size=3, mark options={solid}]
				table {%
					2.500000000000000e-01     1.364544026599912e+00
					1.250000000000000e-01     4.367392511317231e-01
					6.250000000000000e-02     1.713100970621161e-02
					3.125000000000000e-02     1.864924928745459e-03
					1.562500000000000e-02     9.205110724281587e-04
				};
				\addlegendentry{\footnotesize $\kappa = 16,\beta=1$}
				\addplot [line width = 1, color2,densely dotted, mark=square*, mark size=2, mark options={solid}]
				table {%
					2.500000000000000e-01     5.906680133984222e-01
					1.250000000000000e-01     3.878070710202812e-01
					6.250000000000000e-02     1.156764661862669e-02
					3.125000000000000e-02     1.400934613354597e-03
					1.562500000000000e-02     9.091779250625508e-04
				};
				\addlegendentry{\footnotesize $\kappa = 16,\beta=0$}
			\end{axis}
		\end{tikzpicture}
	\end{adjustbox}
	\caption{Left: Convergence of the $H^1_\kappa$-error $\varepsilon_h^{\LOD}$ with $\beta=0$. Right: Comparison of the $H^1_\kappa$-error $\varepsilon_h^{\LOD}$ for the stabilization parameters $\beta=1$ vs. $\beta=0$ with $\kappa = 8,12,16$.}
	\label{fig:weighted_error_beta0}
\end{figure}

\subsection{Influence of the localization parameter $\ell$}

Next we investigate the influence of the localization parameter $\ell$ and compute $u_{h,\ell}^{\LOD}$ for different values of $\ell$ and fixed $\kappa=16$, $h_{\text{\tiny fine}}=2^{-10}$, $\beta=0$. For the coarse mesh size we choose $h=2^{-4}$ and $h=2^{-5}$. As a reference solution we use $u_{h,15}^{\LOD}$  with $\ell=15$ and compute the error $\varepsilon_\ell^{\LOD} := \| u_{h,\ell}^{\LOD} - u_{h,15}^{\LOD} \|_{H^1_{\kappa}(\Omega)}$. Figure \ref{fig:decay_error_and_FEMvsLOD} (left) shows the expected exponential decay for the $H^1_\kappa$-error, where we observe in average the decay rate $\exp(-1.0066 \ell)$ for $h=2^{-4}$ and  $\exp(-0.8741 \ell)$ for $h=2^{-5}$. This gives the decay parameters $\theta_\beta \approx 0.3655$ for $h=2^{-4}$ and $\theta_\beta \approx 0.4172$ for $h = 2^{-5}$ and coincides with the theory.

\begin{figure}[h]
	\definecolor{color0}{rgb}{0.12156862745098,0.466666666666667,0.705882352941177}
	\definecolor{color1}{rgb}{1,0.498039215686275,0.0549019607843137}
	\definecolor{color2}{rgb}{0.172549019607843,0.627450980392157,0.172549019607843}
	\definecolor{color3}{rgb}{0.83921568627451,0.152941176470588,0.156862745098039}
	\definecolor{color4}{rgb}{0.580392156862745,0.403921568627451,0.741176470588235}
	\definecolor{color5}{rgb}{0.549019607843137,0.337254901960784,0.294117647058824}
	\definecolor{color6}{rgb}{0.890196078431372,0.466666666666667,0.76078431372549}
	
	\centering
	\begin{adjustbox}{width=0.44\linewidth}
		\begin{tikzpicture}
			\begin{axis}[
					legend cell align={left},
					legend style={
						fill opacity=0.8,
						draw opacity=1,
						text opacity=1,
						at={(0.96,0.05)},
						anchor=south east,
						draw=white!80!black
					},
					log basis x={2},
					log basis y={10},
					tick align=outside,
					tick pos=left,
					x grid style={white!69.0196078431373!black},
					xmin=1, xmax=10,
					xtick style={color=black},
					xlabel={$\ell$},
					xlabel style={xshift=0.25cm, yshift=-0.5cm},
					y grid style={white!69.0196078431373!black},
					ymin=1e-05, ymax=1e+00,
					ymode=log,
					ytick style={color=black},
					ylabel={$\varepsilon_\ell^{\LOD}$},
					ylabel style={xshift=-0.5cm, yshift=0.25cm},
					legend columns=2,
					legend pos = north east,
					]
					\addplot [line width = 1, color0, mark=*, mark size=2, mark options={solid}]
					table {%
							1.000000000000000e+00     1.168107943932533e-01
							2.000000000000000e+00     4.468336060042428e-02
							3.000000000000000e+00     1.750244336516688e-02
							4.000000000000000e+00     6.648979976824701e-03
							5.000000000000000e+00     2.662044985356583e-03
							6.000000000000000e+00     1.054498966086386e-03
							7.000000000000000e+00     4.032293746584898e-04
							8.000000000000000e+00     1.438617013489406e-04
							9.000000000000000e+00     5.134893951764564e-05
							1.000000000000000e+01     1.746051048409995e-05
					};
					\addlegendentry{$h=2^{-4}$}
					\addplot [line width = 1, color0,densely dotted, mark=square*, mark size=2, mark options={solid}]
					table {%
						1.000000000000000e+00     1.593412810125476e-01
						2.000000000000000e+00     5.823156960044490e-02
						3.000000000000000e+00     2.128083618120551e-02
						4.000000000000000e+00     7.777121442521544e-03
						5.000000000000000e+00     2.842163598117699e-03
						6.000000000000000e+00     1.038674010450618e-03
						7.000000000000000e+00     3.795853626089869e-04
						8.000000000000000e+00     1.387201817483489e-04
						9.000000000000000e+00     5.069555024996468e-05
						1.000000000000000e+01     1.852678379422091e-05
					};
					\addlegendentry{$\exp(-1.0066 \ell)$}
					\addplot [line width = 1, color1, mark=*, mark size=2, mark options={solid}]
					table {%
						1.000000000000000e+00     4.791351816103773e-02
						2.000000000000000e+00     2.416645391654149e-02
						3.000000000000000e+00     1.087083306277031e-02
						4.000000000000000e+00     4.727778410315846e-03
						5.000000000000000e+00     2.025683811543633e-03
						6.000000000000000e+00     8.509978861655315e-04
						7.000000000000000e+00     3.540077759180766e-04
						8.000000000000000e+00     1.471340694275191e-04
						9.000000000000000e+00     6.155266082150160e-05
						1.000000000000000e+01     2.570110827744109e-05
					};
					\addlegendentry{$h=2^{-5}$}
					\addplot [line width = 1, color1,densely dotted, mark=square*, mark size=2, mark options={solid}]
					table {%
						1.000000000000000e+00     6.702238050334852e-02
						2.000000000000000e+00     2.796450759869026e-02
						3.000000000000000e+00     1.166794851755728e-02
						4.000000000000000e+00     4.868350430555889e-03
						5.000000000000000e+00     2.031277038892484e-03
						6.000000000000000e+00     8.475327459655955e-04
						7.000000000000000e+00     3.536256954273594e-04
						8.000000000000000e+00     1.475472576862061e-04
						9.000000000000000e+00     6.156281495441175e-05
						1.000000000000000e+01     2.568655117380377e-05
					};
					\addlegendentry{$\exp(-0.8741 \ell)$}
			\end{axis}
		\end{tikzpicture}
	\end{adjustbox}
	\hspace{15pt}
	\centering
	\begin{adjustbox}{width=0.44\linewidth}
		\begin{tikzpicture}
			\begin{axis}[
				legend cell align={left},
				legend style={
					fill opacity=0.8,
					draw opacity=1,
					text opacity=1,
					at={(0.96,0.05)},
					anchor=south east,
					draw=white!80!black
				},
				legend pos= south east,
				log basis x={2},
				log basis y={10},
				tick align=outside,
				tick pos=left,
				x grid style={white!69.0196078431373!black},
				xmin=1.4e-02, xmax=3e-01,
				xmode=log,
				xtick style={color=black},
				xlabel={$h$},
				xlabel style={xshift=0.25cm, yshift=-0.5cm},
				y grid style={white!69.0196078431373!black},
				ymin=1e-07, ymax=5e+00,
				ymode=log,
				ytick style={color=black},
				ylabel={$\varepsilon_h^{\FEM}$, $\varepsilon_h^{\LOD}$},
				ylabel style={xshift=-0.5cm, yshift=0.25cm},
				legend columns=2,
				]

				\addplot [line width = 1, color0, mark=*, mark size=2, mark options={solid}]
				table {%
					2.500000000000000e-01    1.304706007964277e-01
					1.250000000000000e-01     1.285678514182162e-02
					6.250000000000000e-02     8.629814361590047e-04
					3.125000000000000e-02    9.509590395142185e-05
					1.562500000000000e-02     2.455832305304613e-05
				};
				\addlegendentry{\footnotesize $\varepsilon_h^{\LOD}$}
				\addplot [line width = 1, color0,densely dotted, mark=square*, mark size=2, mark options={solid}]
				table {%
					2.500000000000000e-01     1.407379981553553e+00
					1.250000000000000e-01     5.451725423173400e-01
					6.250000000000000e-02     2.321366278966517e-01
					3.125000000000000e-02     1.088933090671152e-01
					1.562500000000000e-02     5.316390962191783e-02
				};
				\addlegendentry{\footnotesize $\varepsilon_h^{\FEM}$, $\kappa = 8$}
				\addplot [line width = 1, color1, mark=*, mark size=2, mark options={solid}]
				table {%
					2.500000000000000e-01     1.756487646700691e-01
					1.250000000000000e-01     3.883106715390972e-02
					6.250000000000000e-02     2.731329506493729e-03
					3.125000000000000e-02     2.080956954132298e-04
					1.562500000000000e-02     3.298224497909854e-05
				};
				\addlegendentry{\footnotesize $\varepsilon_h^{\LOD}$}
				\addplot [line width = 1, color1,densely dotted, mark=square*, mark size=2, mark options={solid}]
				table {%
					2.500000000000000e-01     1.295718631825076e+00
					1.250000000000000e-01     4.859453981808573e-01
					6.250000000000000e-02     2.302823486534858e-01
					3.125000000000000e-02     1.063776325139076e-01
					1.562500000000000e-02     5.098021559437676e-02
				};
				\addlegendentry{\footnotesize $\varepsilon_h^{\FEM}$, $\kappa = 12$}
				\addplot [line width = 1, color2, mark=*, mark size=2, mark options={solid}]
				table {%
					2.500000000000000e-01     5.906680133984222e-01
					1.250000000000000e-01     3.878070710202812e-01
					6.250000000000000e-02     1.156764661862669e-02
					3.125000000000000e-02     1.400934613354597e-03
					1.562500000000000e-02     9.091779250625508e-04
				};
				\addlegendentry{\footnotesize $\varepsilon_h^{\LOD}$}
				\addplot [line width = 1, color2,densely dotted, mark=square*, mark size=2, mark options={solid}]
				table {%
					2.500000000000000e-01     1.378231919983623e+00
					1.250000000000000e-01     1.417908991443185e+00
					6.250000000000000e-02     1.628438071487168e+00
					3.125000000000000e-02     2.615105681100162e-01
					1.562500000000000e-02     1.114771725752012e-01
				};
				\addlegendentry{\footnotesize $\varepsilon_h^{\FEM}$, $\kappa = 16$}
				\addplot [line width = 1, color3, mark=*, mark size=2, mark options={solid}]
				table {%
					2.500000000000000e-01     1.016021025248976e+00
					1.250000000000000e-01     3.388261637287355e-01
					6.250000000000000e-02     2.022473326609193e-02
					3.125000000000000e-02     2.123920041558945e-03
					1.562500000000000e-02     2.633302146863045e-03
				};
				\addlegendentry{\footnotesize $\varepsilon_h^{\LOD}$}
				\addplot [line width = 1, color3,densely dotted, mark=square*, mark size=2, mark options={solid}]
				table {%
					2.500000000000000e-01     1.244409623829920e+00
					1.250000000000000e-01     1.306343086969237e+00
					6.250000000000000e-02     6.343066876185528e-01
					3.125000000000000e-02     5.638098935089952e-01
					1.562500000000000e-02     1.188300871845986e-01
				};
				\addlegendentry{\footnotesize $\varepsilon_h^{\FEM}$, $\kappa = 20$}
				\addplot [line width = 1, color4, mark=*, mark size=2, mark options={solid}]
				table {%
					2.500000000000000e-01     1.392752149813496e+00
					1.250000000000000e-01     5.519839882804088e-01
					6.250000000000000e-02     4.463368544846230e-01
					3.125000000000000e-02     6.161786715740432e-03
					1.562500000000000e-02     4.144399282156794e-04
				};
				\addlegendentry{\footnotesize $\varepsilon_h^{\LOD}$}
				\addplot [line width = 1, color4,densely dotted, mark=square*, mark size=2, mark options={solid}]
				table {%
					2.500000000000000e-01     1.275826973644790e+00
					1.250000000000000e-01     1.382864564310055e+00
					6.250000000000000e-02     8.920895223668666e-01
					3.125000000000000e-02     5.223001416551410e-01
					1.562500000000000e-02     1.807131843066118e-01
				};
				\addlegendentry{\footnotesize $\varepsilon_h^{\FEM}$, $\kappa = 32$}
				\addplot [line width = 1, black, dashed]
				table {%
					2.500000000000000e-01     4.988155787422199e-01
					1.250000000000000e-01     2.494077893711099e-01
					6.250000000000000e-02     1.247038946855550e-01
					3.125000000000000e-02     6.235194734277748e-02
					1.562500000000000e-02     3.117597367138874e-02
				};
				\addplot [line width = 1, black, dashed]
				table {%
					2.500000000000000e-01     3.117597367138874e-02
					1.250000000000000e-01     3.896996708923593e-03
					6.250000000000000e-02     4.871245886154491e-04
					3.125000000000000e-02     6.089057357693114e-05
					1.562500000000000e-02     7.611321697116392e-06
				};
			\end{axis}
		\end{tikzpicture}
	\end{adjustbox}
	\caption{Left: $H^1_\kappa$-error $\varepsilon_\ell^{\LOD}$ for $h=2^{-4}$ and $2^{-5}$. Right: Comparison of the $H^1_\kappa$-errors $\varepsilon_h^{\FEM}$ and $\varepsilon_h^{\LOD}$ for the cases $\kappa=8,12,16,20,32$ with $\ell=10$ and $\beta=0$}
	\label{fig:decay_error_and_FEMvsLOD}
\end{figure}
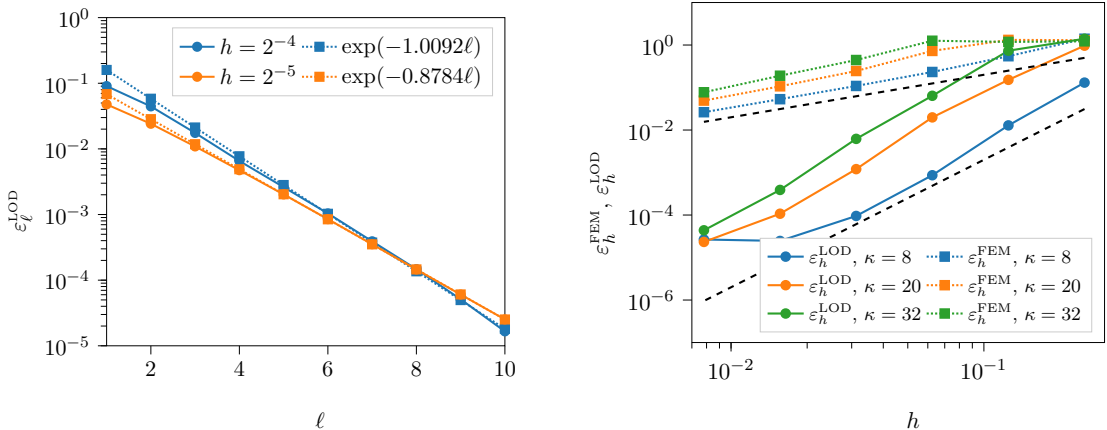

\subsection{Comparison of the LOD-method with the finite element method}
\label{subsection:comparison-FEM-LOD}

During our experiments it has become clear, that the LOD-method creates sharper images with less degrees of freedom compared to the standard finite element method. To demonstrate this, we show the vortex patterns for $\kappa=16$ computed with both methods. As parameters we use $h_{\text{\tiny fine}}=2^{-7}$, $h= 2^{-\{2,3,4,5\}}$, $\ell=4$ and $\beta=0$ or $\beta = 1$ for the LOD-method and $h=2^{-\{2,3,4,5\}}$ for the standard finite element method.

\begin{figure}[h!]
\centering
\begin{minipage}{0.24\textwidth}
\includegraphics[scale=0.21]{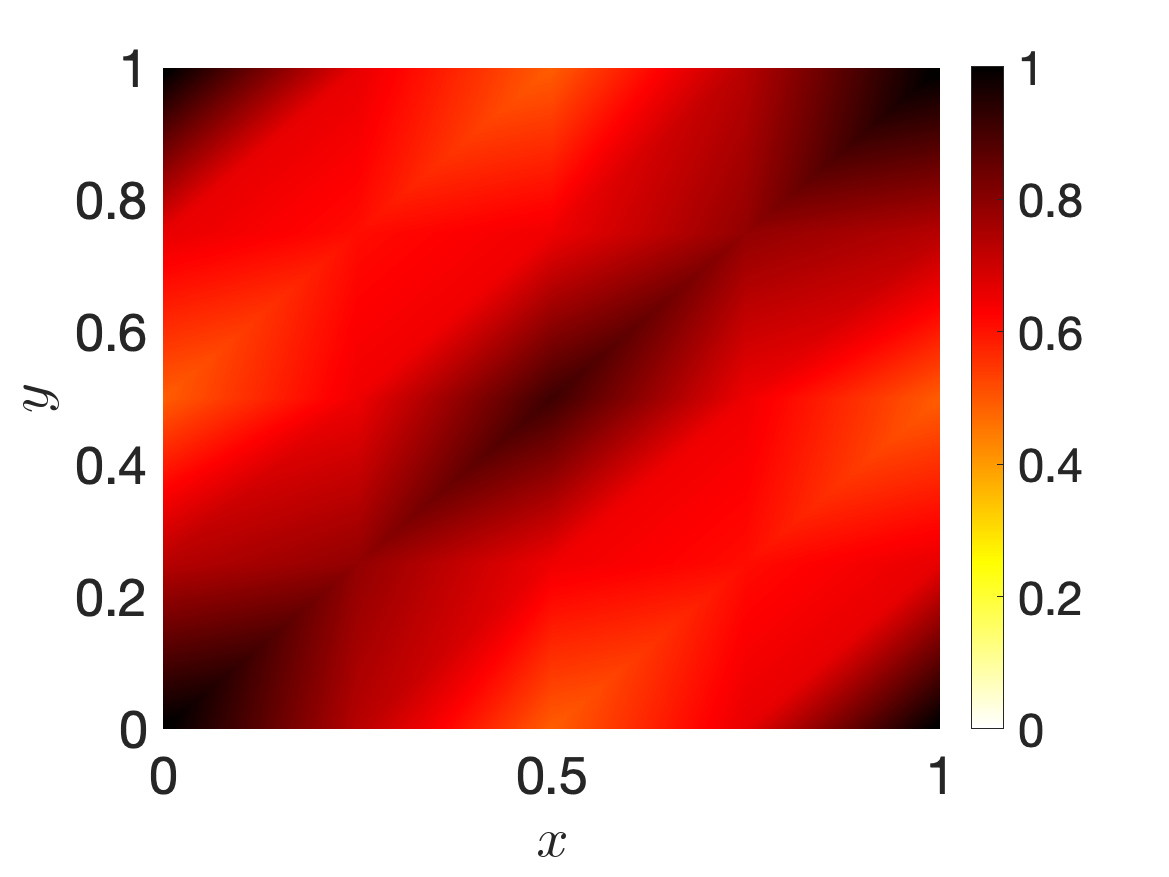}
\end{minipage}
\begin{minipage}{0.24\textwidth}
\includegraphics[scale=0.21]{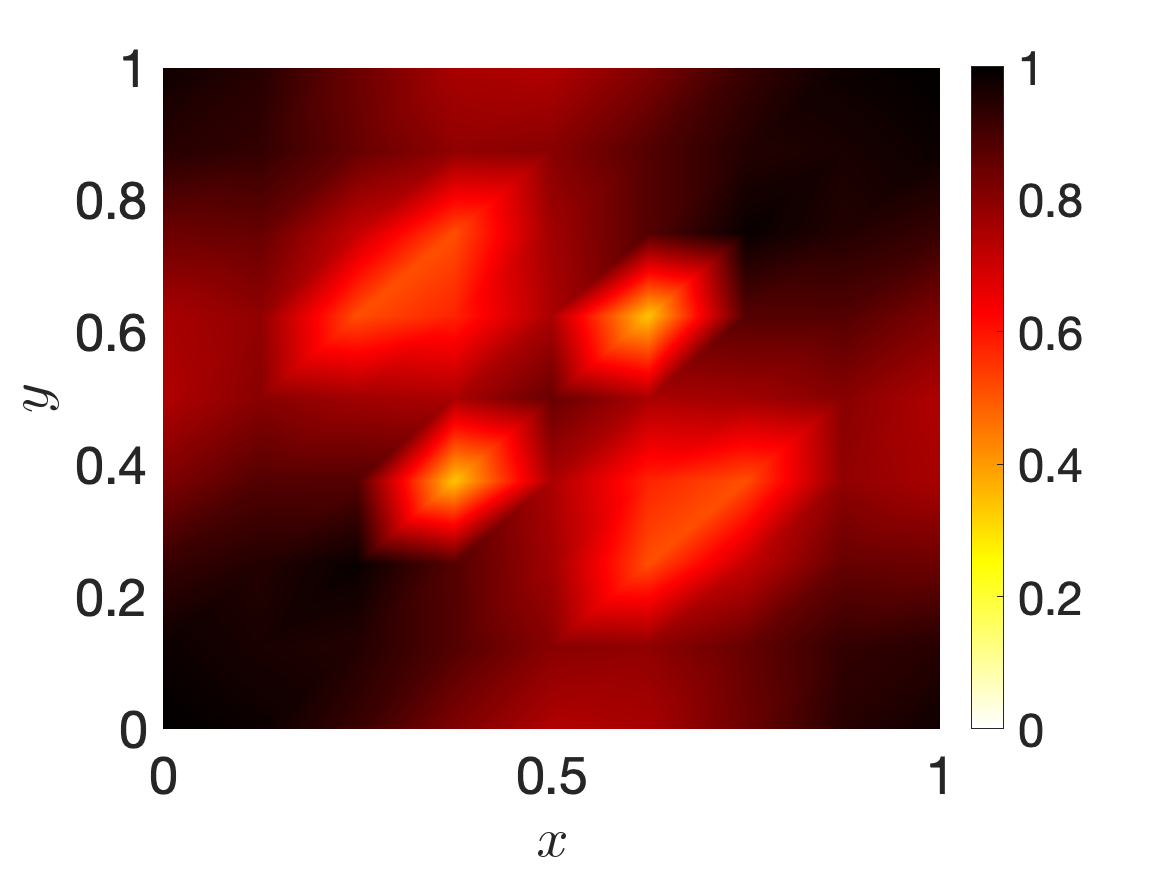}
\end{minipage}
\begin{minipage}{0.24\textwidth}
\includegraphics[scale=0.21]{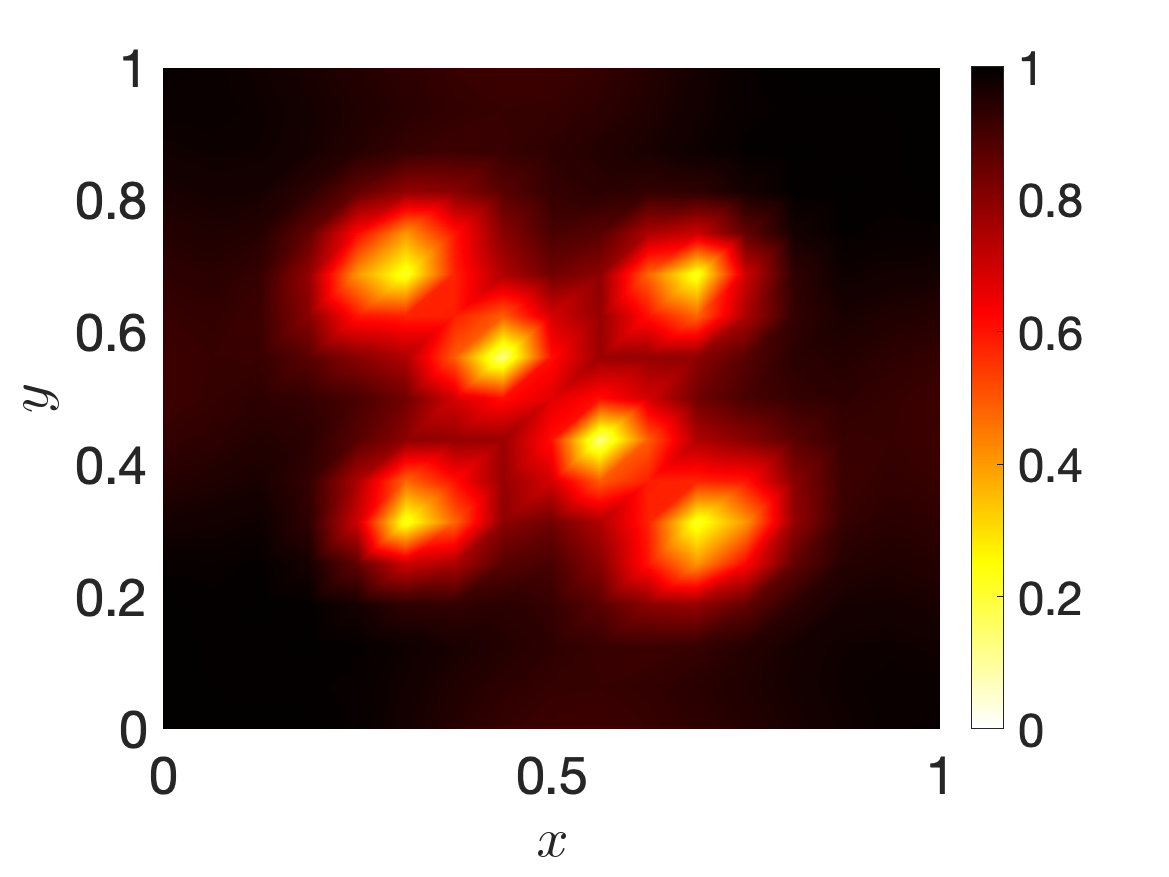}
\end{minipage}
\begin{minipage}{0.24\textwidth}
\includegraphics[scale=0.21]{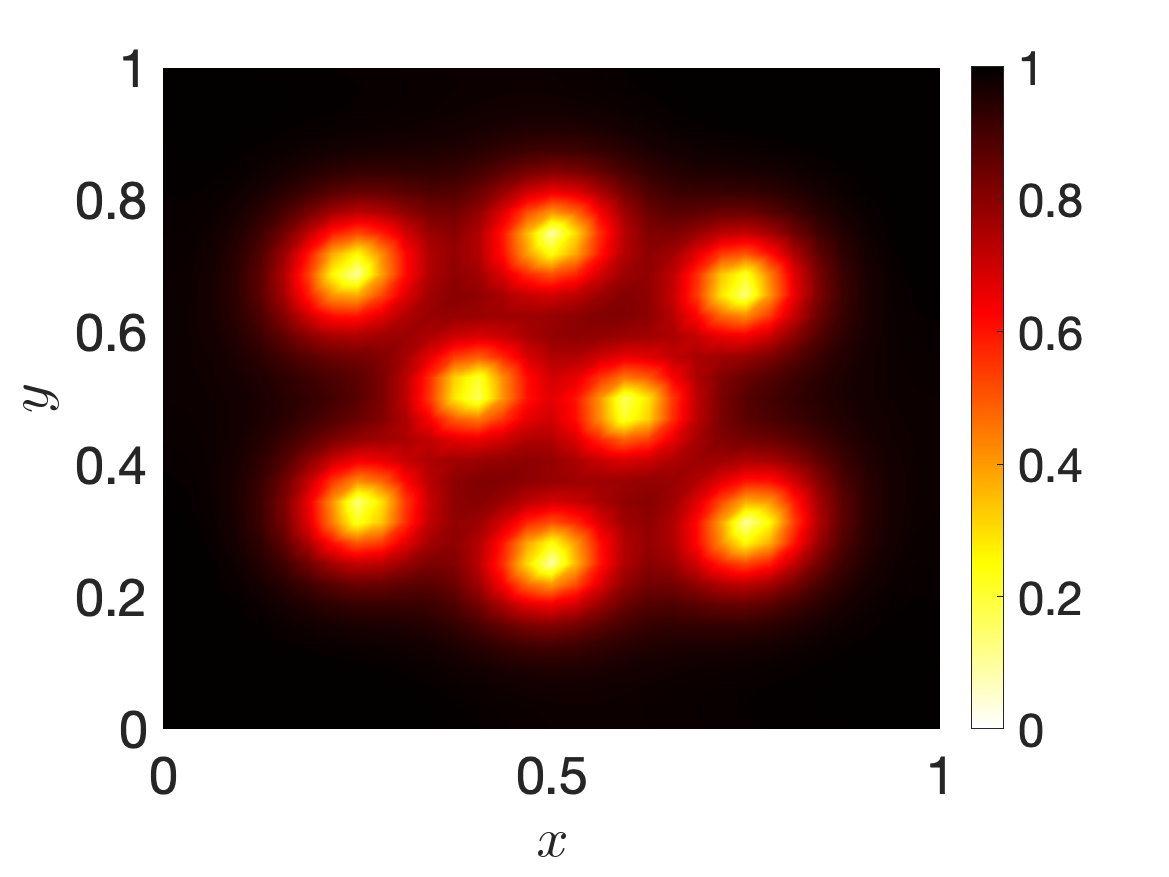}
\end{minipage}
\begin{minipage}{0.24\textwidth}
\includegraphics[scale=0.21]{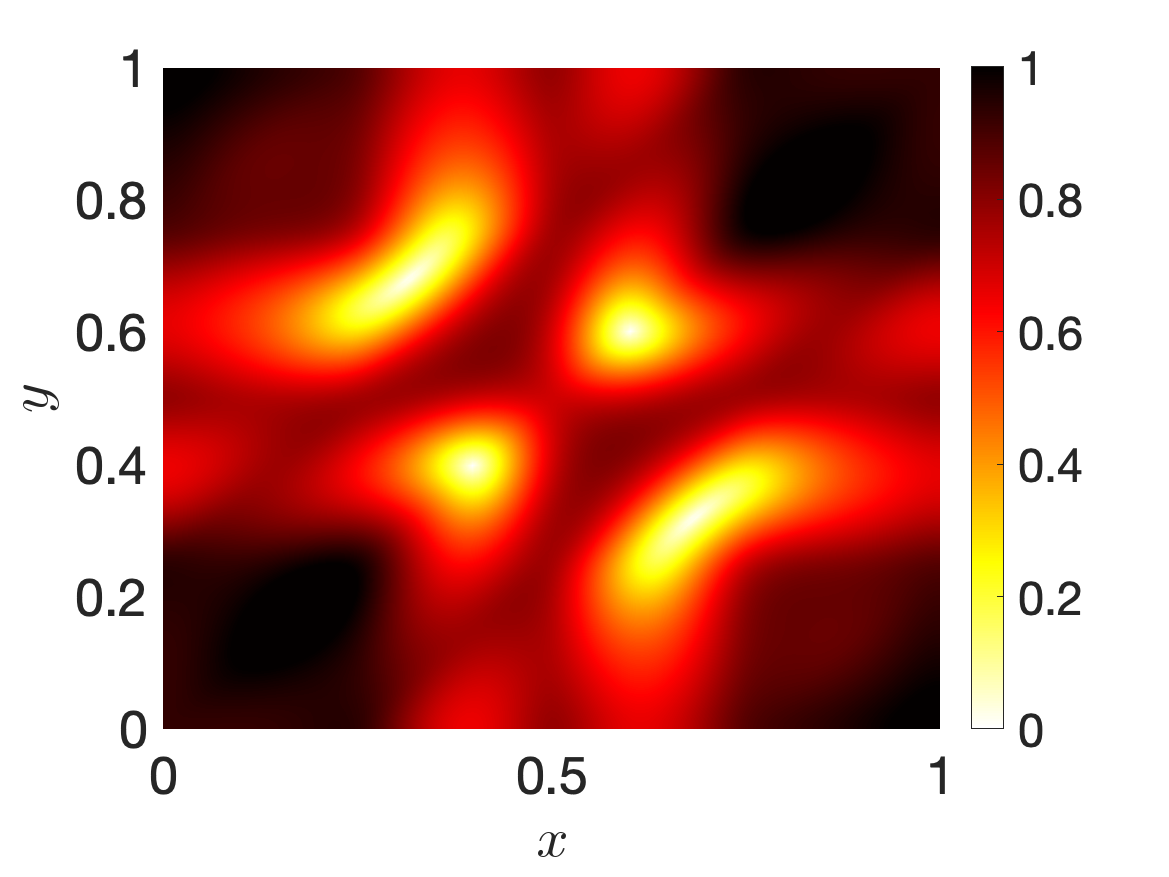}
\end{minipage}
\begin{minipage}{0.24\textwidth}
\includegraphics[scale=0.21]{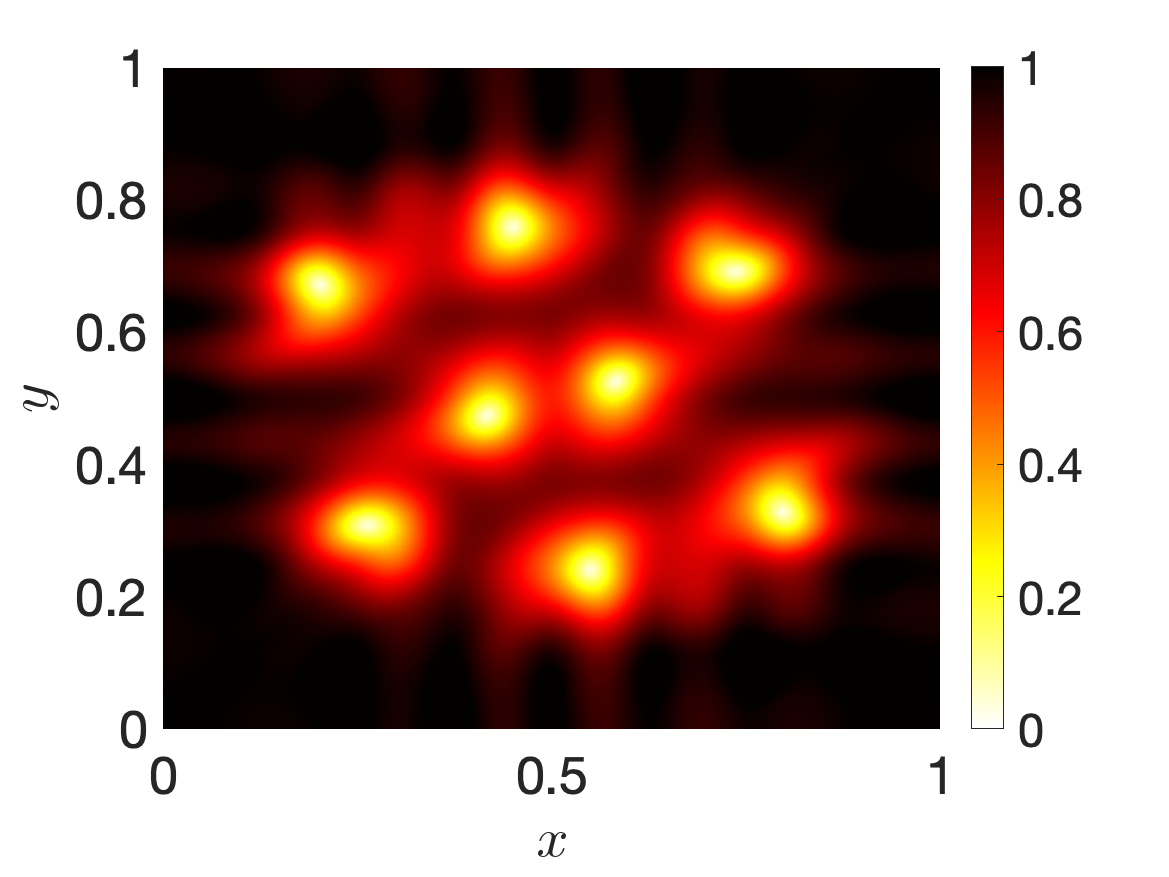}
\end{minipage}
\begin{minipage}{0.24\textwidth}
\includegraphics[scale=0.21]{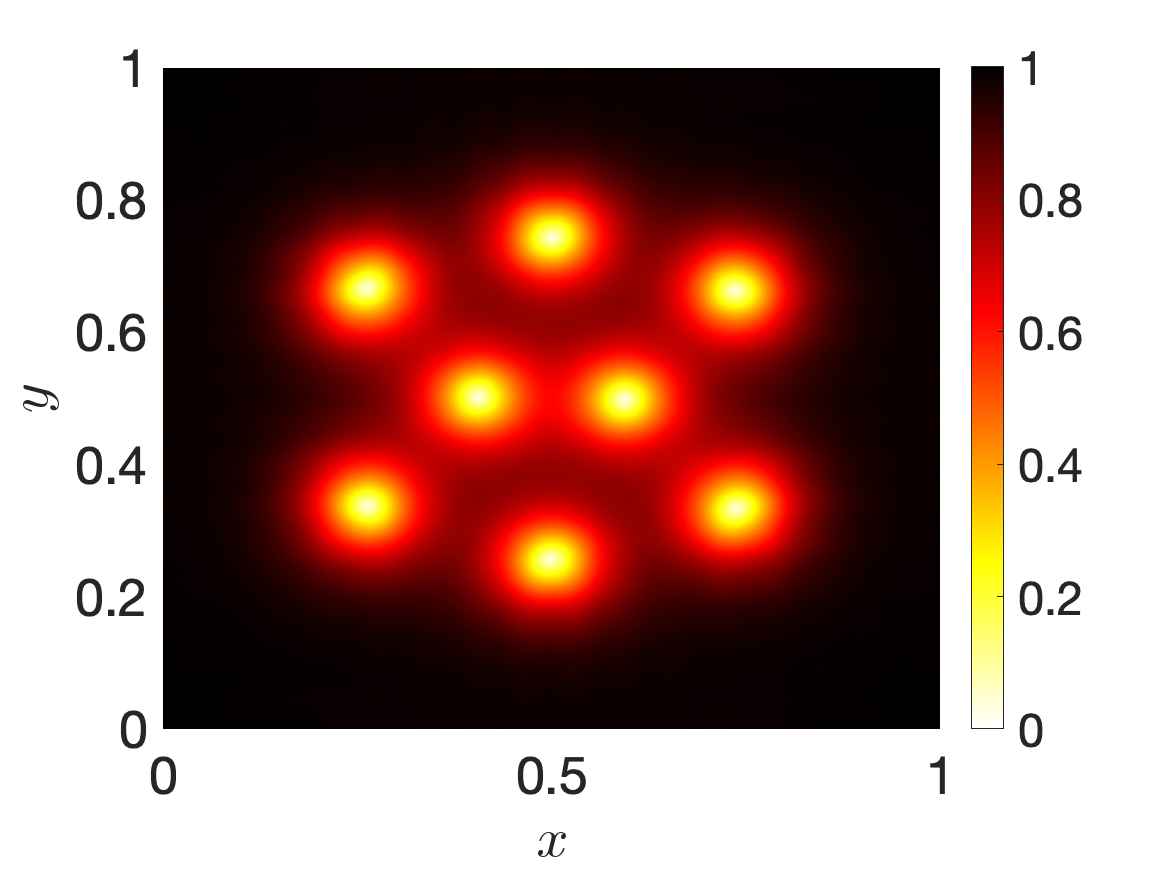}
\end{minipage}
\begin{minipage}{0.24\textwidth}
\includegraphics[scale=0.21]{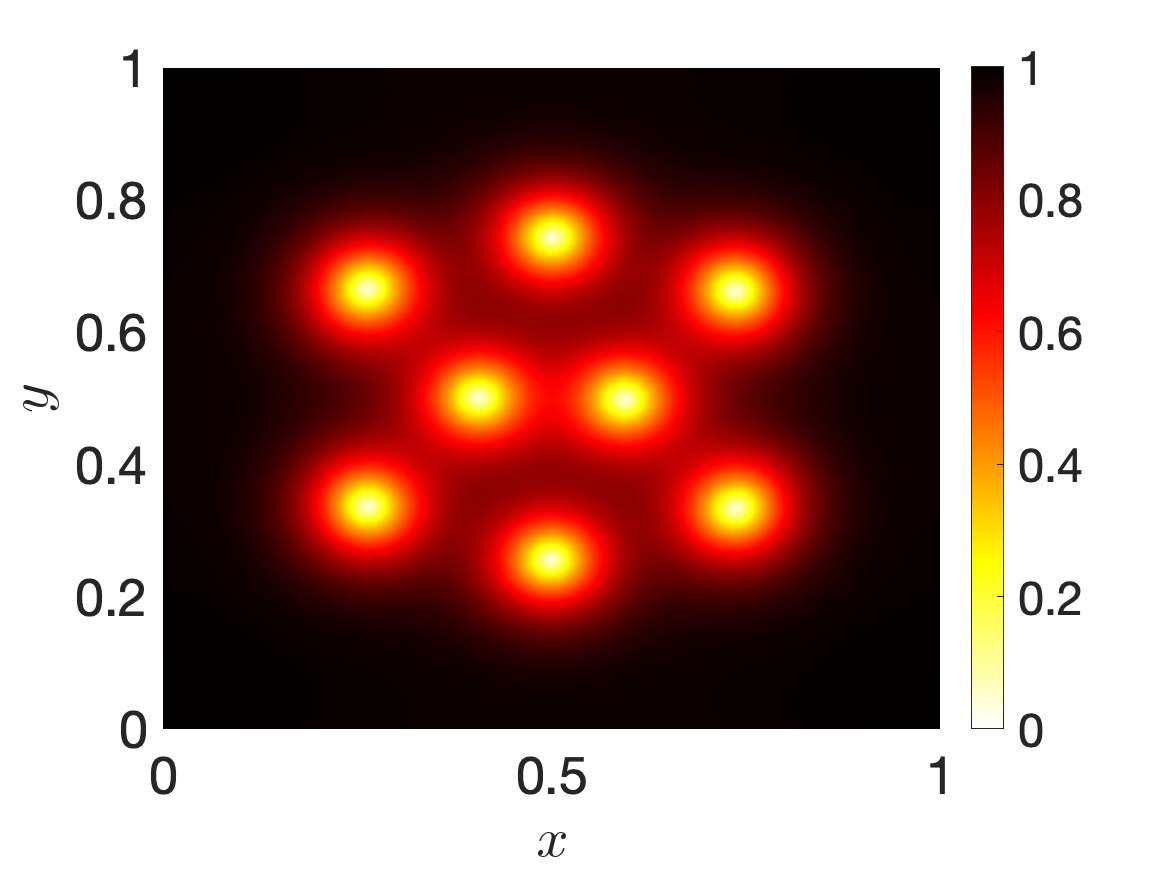}
\end{minipage}
\begin{minipage}{0.24\textwidth}
\includegraphics[scale=0.21]{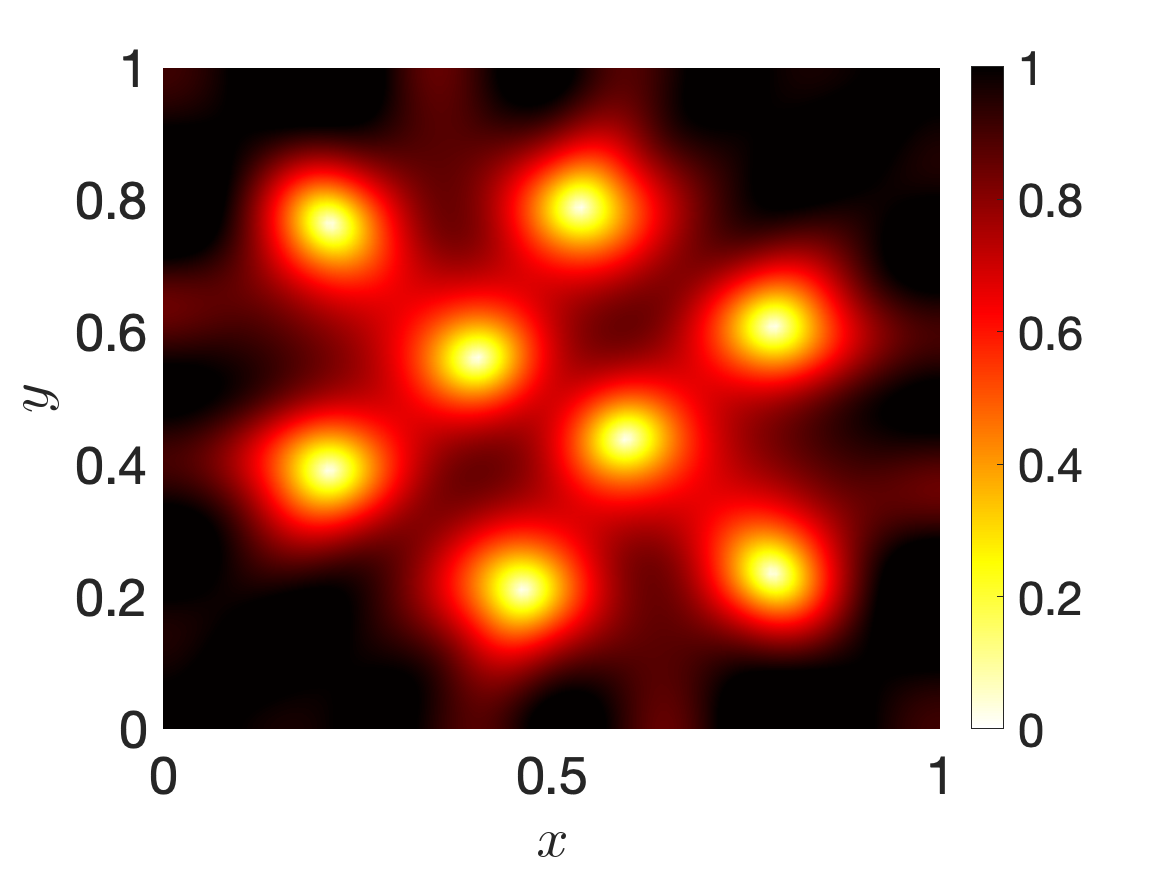}
\end{minipage}
\begin{minipage}{0.24\textwidth}
\includegraphics[scale=0.21]{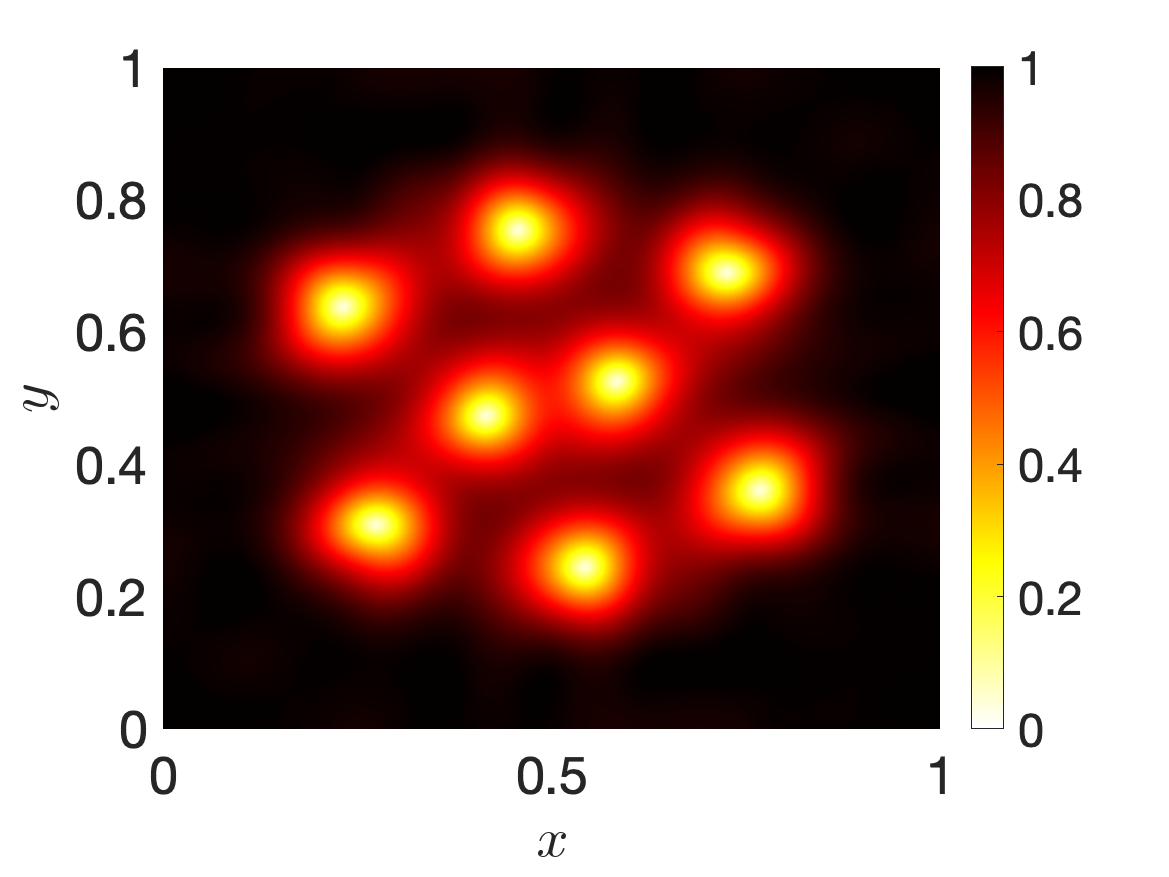}
\end{minipage}
\begin{minipage}{0.24\textwidth}
\includegraphics[scale=0.21]{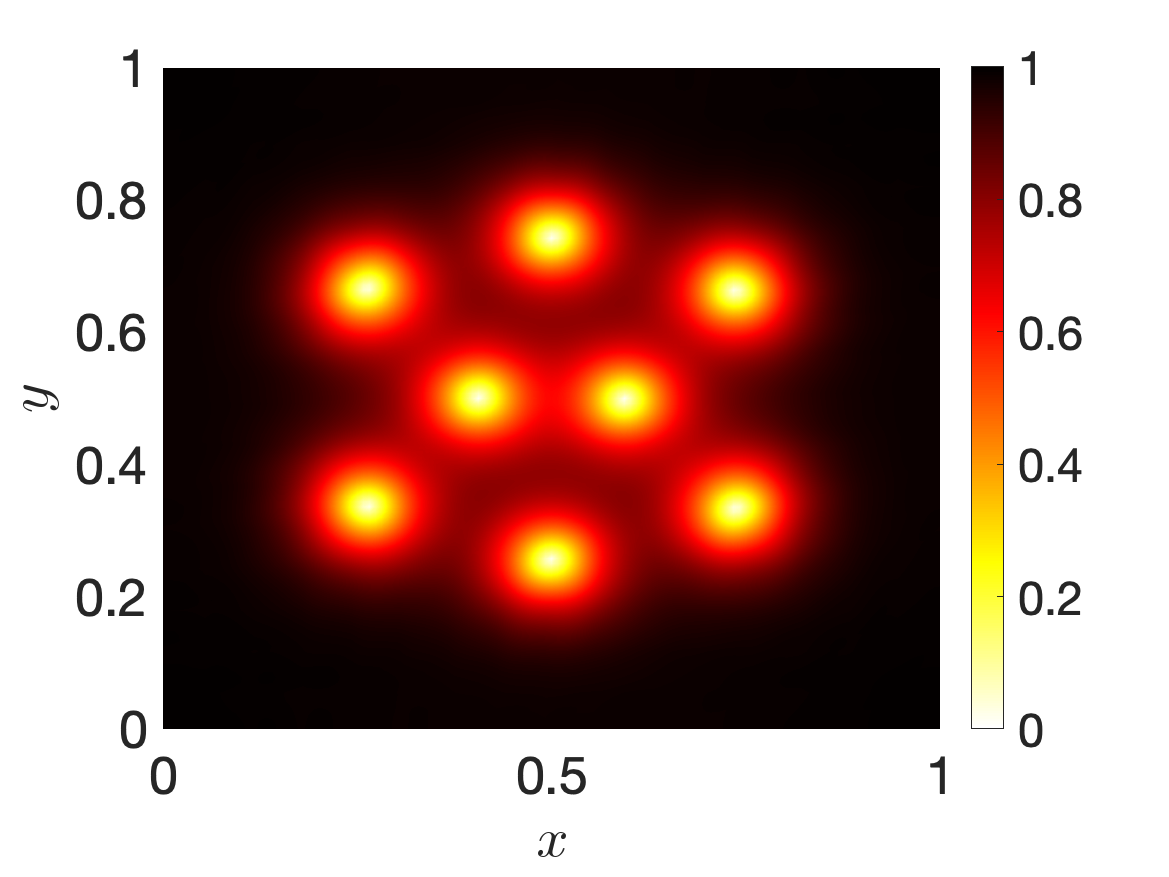}
\end{minipage}
\begin{minipage}{0.24\textwidth}
\includegraphics[scale=0.21]{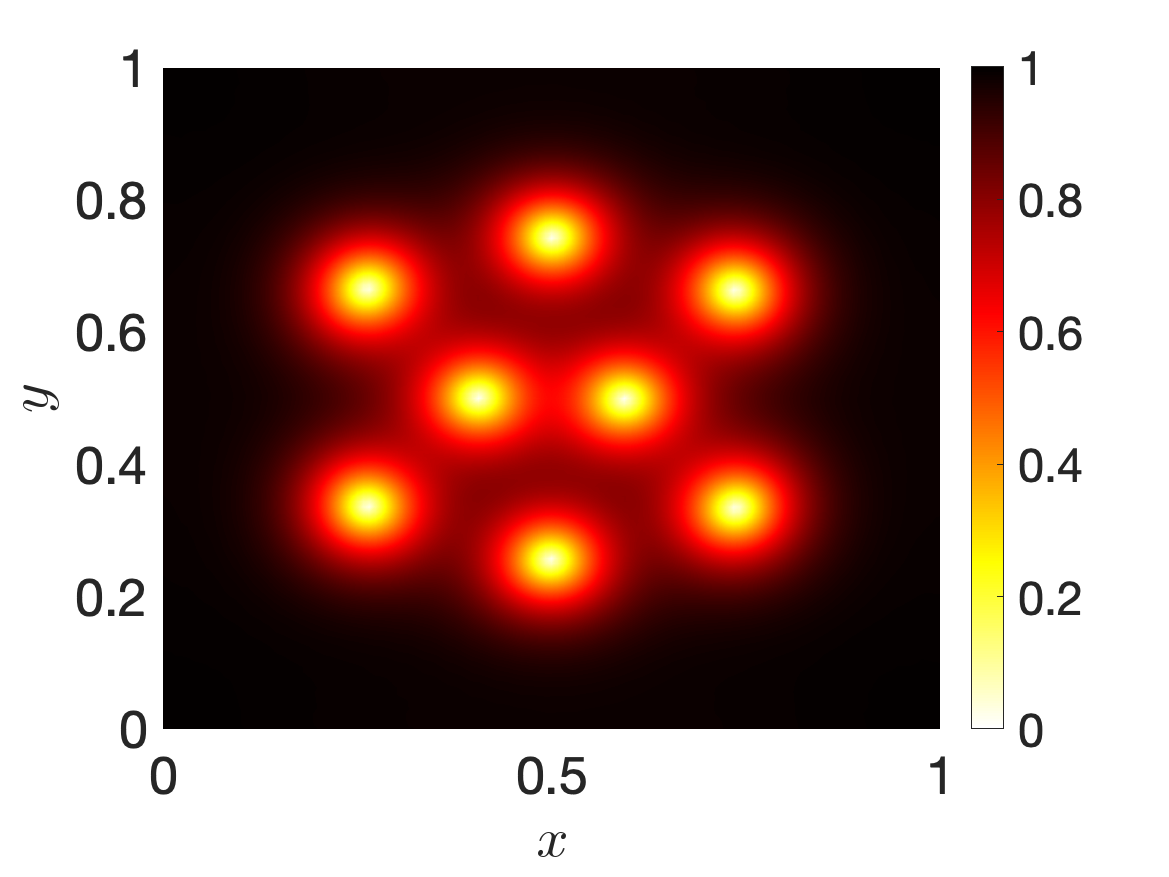}
\end{minipage}
\caption{The figure shows the absolute value $|u|$ for $\kappa = 16$ and different numerical methods. Upper row: Vortices computed with FEM, $h_{\text{\tiny ref}}=2^{-\{2,3,4,5\}}$ from left to right; Middle row: Vortices computed with LOD, $h_{\text{\tiny fine}}=2^{-7}$, $\ell=4$, $\beta=1$, and $h=2^{-\{2,3,4,5\}}$ from left to right; Lower row: Vortices computed with LOD, $h_{\text{\tiny fine}}=2^{-7}$, $\ell=4$, $\beta=0$, and $h=2^{-\{2,3,4,5\}}$ from left to right.}
	\label{fig:vortices}
\end{figure}

Figure \ref{fig:vortices} shows in the upper row the vortex pattern for the FEM approximations, the middle row for the LOD-approximations with $\beta=1$ and in the lower row for the LOD-approximations with $\beta=0$. We can already see a good resolution of the vortex pattern for a mesh size of $h=2^{-3}$ (i.e. 81 DOFs) with the LOD-method which becomes even sharper for the mesh size $h = 2^{-4}$ (i.e. 289 DOFs). In contrast to the LOD approximations, the finite element method requires a mesh size of at least $h_{\text{\tiny ref}}=2^{-5}$ (i.e. 1024 DOFs) so that the correct vortex pattern becomes visible. This emphasizes again the significantly improved resolution of the correct pattern in the LOD-spaces. Even in the case of $h = 2^{-2}$ (i.e. 25 DOFs) the almost correct vortex pattern roughly shows up that is the right number of vortices is captured but at relatively inaccurate positions. However, we see again that the stabilization parameter $\beta=0$ is clearly preferable over $\beta=1$.

Let us finally compare the errors for both methods. We choose exemplary $\kappa=8,\, 12,\, 16, \,20,\,32$ and the errors are shown in Figure \ref{fig:decay_error_and_FEMvsLOD} (right). We see that the LOD method converges faster and exhibits a smaller preasymptotic phase compared to the finite element method. In summary, our experiments show that the LOD method seems to be very suitable to compute minimizers of the Ginzburg-Landau energy and to capture vortex patterns in superconductors.

\def\cprime{$'$}

\appendix

\section{Proofs for the decay of element correctors}
\label{section:appendix}

In the following, we present the proofs to the decay results in Theorem \ref{theorem-decay-ideal-corrector}, Lemma \ref{lemma:local-corrector-estimate} and Theorem \ref{theorem:C-Cell-error}. The proofs in this appendix are essentially following the arguments by Peterseim \cite{Pet17} though various modifications are necessary to fit our setting.

Let us first recall the main statement of Theorem \ref{theorem-decay-ideal-corrector} which guarantees the existence of a constant $0< \theta_{\beta} <1$ such that for each $T \in \mathcal{T}_H$ and $v_h \in V_h$ the corrector $\mathcal{C}_{T}v_h \in W$ from \eqref{definition-CT} fulfills
\begin{eqnarray*}
|||  \mathcal{C}_{T}v_h  |||_{\beta,\Omega \setminus \UN^{\ell}(T)} &\lesssim& \theta_{\beta}^{\ell} \, |||  v_h  |||_{\beta,T}
\qquad
\mbox{where} \quad 
\theta_{\beta} \,\,\le\,\,  \left( 1 - \tfrac{ 1}{1+C (1 + \kappa \beta h)} \right)^{1/12} < 1.
\end{eqnarray*}

\begin{proof}[Proof of Theorem \ref{theorem-decay-ideal-corrector}]
Let $\ell \in \mathbb{N}$ be fixed and let $\eta \in W^{1,\infty}(\Omega)$ be a cutoff function with $0 \le \eta \le 1$ and the properties 
\begin{align}
\label{truncation-function}
\eta = 0 \quad \mbox{in } \UN^{\ell+2}(T); \qquad  \qquad \eta = 1 \quad \mbox{in } \Omega \setminus \UN^{\ell+3}(T) \qquad \mbox{and} \qquad \| \nabla \eta \|_{L^{\infty}(\Omega)} \le \tfrac{1}{h} C_{\eta} 
\end{align}
for some generic constant $C_{\eta}>0$. For brevity, let $w_{T}:=\mathcal{C}_{T}v_h$. With $w_{T}\in W$ we note that
\begin{align}
\label{proofdecay1-wT-identiy}
\mbox{on } \Omega \setminus \UN^{\ell+3}(T): \qquad w_T = \eta w _T = \eta w_T - (P_h^{-1} \circ P_h)(\eta w_T) + (P_h^{-1} \circ P_h)(\eta w_T),
\end{align}
as well as
\begin{align*}
\mbox{on } \Omega \setminus \UN^{\ell+4}(T): \hspace{5pt} 0 = P_h(w_T) = P_h(\eta w_T)
\qquad
\mbox{and}
\qquad \mbox{on } \UN^{\ell+1}(T): \hspace{5pt} 0 = \eta w_T \overset{\eqref{supportgrowth-Ph}}{= }P_h(\eta w_T).
\end{align*}
and consequently with \eqref{growth-Phinv}
\begin{align}
\label{proofdecay1-Pinvsupport}
\mbox{supp}\hspace{1pt}(P_h^{-1} \circ P_h)(\eta w_T) \subset \UN^{\ell+5}(T) \setminus \UN^{\ell}(T).
\end{align}
We obtain
\begin{align*}
& |||  w_T  |||_{\beta,\Omega \setminus \UN^{\ell+3}(T)}^2 \overset{\eqref{proofdecay1-wT-identiy}}{=} \underbrace{ a_{\beta,\Omega}( w_T ,  \eta w_T - (P_h^{-1} \circ P_h)(\eta w_T) ) }_{=:\mbox{I}} \\
& \qquad -  \underbrace{ a_{\beta, \Omega \setminus \UN^{\ell+3}(T) }( w_T ,  (P_h^{-1} \circ P_h)(\eta w_T)  ) }_{=:\mbox{II}} - \underbrace{ a_{\beta,\UN^{\ell+3}(T)} ( w_T , \eta w_T - (P_h^{-1} \circ P_h)(\eta w_T) ) }_{=:\mbox{III}}. 
\end{align*}
We now estimate the terms individually.\\[0.2em]
{\it Estimate of $\mbox{I}$:} Since $\eta w_T - (P_h^{-1} \circ P_h)(\eta w_T) \in W$ with no support on the element $T$, we obtain with the definition of $w_{T}=\mathcal{C}_{T}v_h$ (cf. \eqref{definition-CT}) that
\begin{align*}
\mbox{I} = a_{\beta,\Omega} ( w_T , \eta w_T - (P_h^{-1} \circ P_h)(\eta w_T) ) = a_{\beta,T} ( v_h , \eta w_T - (P_h^{-1} \circ P_h)(\eta w_T) ) = 0.
\end{align*}
{\it Estimate of $\mbox{II}$:} Exploiting \eqref{proofdecay1-Pinvsupport}, we obtain
\begin{align*}
\left| \mbox{II} \right| &= | a_{\beta, \UN^{\ell+5}(T) \setminus \UN^{\ell+3}(T) }( w_T ,  (P_h^{-1} \circ P_h)(\eta w_T)  ) | \\
&\le
||| w_T |||_{\beta, \UN^{\ell+5}(T) \setminus \UN^{\ell+3}(T) } \, \, |||  (P_h^{-1} \circ P_h)(\eta w_T) |||_{\beta, \UN^{\ell+5}(T) \setminus \UN^{\ell+3}(T) }.
\end{align*}
For the latter term we have with an analogous local version of Lemma \ref{Lemma_2.1} that
\begin{eqnarray*}
\lefteqn{ |||  (P_h^{-1} \circ P_h)(\eta w_T) |||_{\beta, \UN^{\ell+5}(T) \setminus \UN^{\ell+3}(T) } } \\ 
&\lesssim& \kappa^{-1} \Vert \nabla (P_h^{-1} \circ P_h)(\eta w_T) \Vert_{L^2(\UN^{\ell+5}(T) \setminus \UN^{\ell+3}(T))} + \beta \Vert (P_h^{-1} \circ P_h)(\eta w_T) \Vert_{L^2(\UN^{\ell+5}(T) \setminus \UN^{\ell+3}(T))},
\end{eqnarray*}
where we further estimate using \eqref{supportgrowth-Ph}, \eqref{growth-Phinv}, \eqref{truncation-function} and since $P_h w_T = 0$
\begin{eqnarray*}
\lefteqn{ \Vert \nabla (P_h^{-1} \circ P_h)(\eta w_T) \Vert_{L^2(\UN^{\ell+5}(T) \setminus \UN^{\ell+3}(T))} 
\,\,\,\lesssim\,\,\, \Vert \nabla P_h(\eta w_T) \Vert_{L^2(\UN^{\ell+6}(T) \setminus \UN^{\ell+2}(T))} } \\
&=& \Vert \nabla P_h(\eta w_T) \Vert_{L^2(\UN^{\ell+4}(T) \setminus \UN^{\ell+2}(T))} \,\,\, \lesssim \,\,\, \Vert \nabla (\eta w_T) \Vert_{L^2(\UN^{\ell+5}(T) \setminus \UN^{\ell+2}(T))} \\
&\lesssim& \Vert \nabla w_T \Vert_{L^2(\UN^{\ell+5}(T) \setminus \UN^{\ell+2}(T))} + \| \nabla \eta \|_{L^{\infty}(\Omega)}  \Vert w_T \Vert_{L^2(\UN^{\ell+3}(T) \setminus \UN^{\ell+2}(T))} \\
&\lesssim& \Vert \nabla w_T \Vert_{L^2(\UN^{\ell+5}(T) \setminus \UN^{\ell+2}(T))} + h^{-1}  \Vert w_T - P_h w_T\Vert_{L^2(\UN^{\ell+3}(T) \setminus \UN^{\ell+2}(T))} \\
&\lesssim& \Vert \nabla w_T \Vert_{L^2(\UN^{\ell+5}(T) \setminus \UN^{\ell+1}(T))}
\end{eqnarray*}
and similarly 
\begin{eqnarray*}
\lefteqn{ \Vert (P_h^{-1} \circ P_h)(\eta w_T) \Vert_{L^2(\UN^{\ell+5}(T) \setminus \UN^{\ell+3}(T)) } }\\ 
&\le& \Vert w_T \Vert_{L^2(\UN^{\ell+5}(T) \setminus \UN^{\ell+3}(T))} +  \Vert \eta w_T - (P_h^{-1} \circ P_h)(\eta w_T) \Vert_{L^2(\UN^{\ell+5}(T) \setminus \UN^{\ell+3}(T))} \\
&\lesssim& \Vert w_T \Vert_{L^2(\UN^{\ell+5}(T) \setminus \UN^{\ell+3}(T))} +  h \Vert \nabla (\eta w_T - (P_h^{-1} \circ P_h)(\eta w_T)) \Vert_{L^2(\UN^{\ell+6}(T) \setminus \UN^{\ell+2}(T))} \\
&\lesssim& 
\Vert w_T \Vert_{L^2(\UN^{\ell+5}(T) \setminus \UN^{\ell+3}(T))}  + h \Vert \nabla w_T \Vert_{L^2(\UN^{\ell+6}(T) \setminus \UN^{\ell+1}(T))} \\
&\lesssim&  h \Vert \nabla w_T \Vert_{L^2(\UN^{\ell+6}(T) \setminus \UN^{\ell+1}(T))}.
\end{eqnarray*}
Combining the estimates yields
\begin{align*}
\left| \mbox{II} \right| &\lesssim 
||| w_T |||_{\beta, \UN^{\ell+5}(T) \setminus \UN^{\ell+3}(T) } \, \,(1 + \kappa \beta h) \kappa^{-1} \Vert \nabla w_T \Vert_{L^2(\UN^{\ell+6}(T) \setminus \UN^{\ell}(T))} \\
&\lesssim (1 + \kappa \beta h) |||  w_T  |||_{\beta,\UN^{\ell+6}(T) \setminus \UN^{\ell+1}(T)}^2.
\end{align*}
{\it Estimate of $\mbox{III}$:} Since $\mbox{supp}\hspace{1pt}(P_h^{-1} \circ P_h)(\eta w_T) \subset \UN^{\ell+5}(T) \setminus \UN^{\ell}(T)$ we have
\begin{eqnarray*}
|\mbox{III}| &=& |a_{\beta,\UN^{\ell+3}(T)} ( w_T , \eta w_T - (P_h^{-1} \circ P_h)(\eta w_T) )| \\
&=&  |a_{\beta,\UN^{\ell+3}(T) \setminus \UN^{\ell}(T)} ( w_T , \eta w_T - (P_h^{-1} \circ P_h)(\eta w_T) )| \\
&\lesssim& ||| w_T   |||_{\beta, \UN^{\ell+3}(T) \setminus \UN^{\ell}(T) } \, \left( |||\eta w_T |||_{\beta, \UN^{\ell+3}(T) \setminus \UN^{\ell}(T) } + ||| (P_h^{-1} \circ P_h)(\eta w_T)  |||_{\beta, \UN^{\ell+3}(T) \setminus \UN^{\ell}(T) } \right) \\
&\lesssim&  ||| w_T   |||_{\beta, \UN^{\ell+3}(T) \setminus \UN^{\ell}(T) } \, \left(  ||| (P_h^{-1} \circ P_h)(\eta w_T)  |||_{\beta, \UN^{\ell+3}(T) \setminus \UN^{\ell}(T) } \right.\\
&\enspace& \qquad \left. +   \kappa^{-1} \|  \nabla  w_T \|_{L^2( \UN^{\ell+4}(T) \setminus \UN^{\ell+1}(T))} + \beta \| w_T \|_{L^2( \UN^{\ell+3}(T) \setminus \UN^{\ell+2}(T))}  \right).
\end{eqnarray*}
Here we obtain analogously to the estimates for $\mbox{II}$
\begin{eqnarray*}
\lefteqn{ |||  (P_h^{-1} \circ P_h)(\eta w_T)  |||_{\beta, \UN^{\ell+3}(T) \setminus \UN^{\ell}(T) } } \\ 
&\lesssim&  \kappa^{-1} \|  \nabla (P_h^{-1} \circ P_h)(\eta w_T) \|_{L^2( \UN^{\ell+3}(T) \setminus \UN^{\ell}(T))} + \beta \|  (P_h^{-1} \circ P_h)(\eta w_T)  \|_{L^2( \UN^{\ell+3}(T) \setminus \UN^{\ell}(T))} \\
&\lesssim& (1 + \kappa \beta h) |||  w_T  |||_{\beta,\UN^{\ell+5}(T) \setminus \UN^{\ell}(T)}
\end{eqnarray*}
and in conclusion
\begin{eqnarray*}
|\mbox{III}|  &\lesssim& (1 + \kappa \beta h) |||  w_T  |||_{\beta,\UN^{\ell+5}(T) \setminus \UN^{\ell}(T)}^2.
\end{eqnarray*}
Combing the estimates for I, II and III yields for some constant $C>0$:
\begin{eqnarray*}
|||  w_T  |||_{\beta,\Omega \setminus \UN^{\ell+6}(T)}^2 &\le& |||  w_T  |||_{\beta,\Omega \setminus \UN^{\ell+3}(T)}^2 
 \,\,\, \le \,\,\, C (1 + \kappa \beta h) |||  w_T  |||_{\beta,\UN^{\ell+6}(T) \setminus \UN^{\ell}(T)}^2 \\
 &=& C (1 + \kappa \beta h) \left( |||  w_T  |||_{\beta,\Omega \setminus \UN^{\ell}(T) }^2 - |||  w_T  |||_{\beta, \Omega \setminus \UN^{\ell+6}(T)}^2 \right).
\end{eqnarray*}
Hence
\begin{eqnarray*}
|||  w_T  |||_{\beta,\Omega \setminus \UN^{\ell+6}(T)}^2 &\le&
 \frac{ C (1 + \kappa \beta h) }{1+C (1 + \kappa \beta h)} |||  w_T  |||_{\beta,\Omega \setminus \UN^{\ell}(T)}^2
\end{eqnarray*}
and recursively
\begin{eqnarray*}
|||  w_T  |||_{\beta,\Omega \setminus \UN^{\ell+6}(T)}^2 &\le&
 \left( \frac{ C (1 + \kappa \beta h) }{1+C (1 + \kappa \beta h)} \right)^{\lfloor \tfrac{\ell}{6} \rfloor } |||  w_T  |||_{\beta,\Omega}^2 \\
&\lesssim&  \left( 1 - \frac{ 1}{1+C (1 + \kappa \beta h)} \right)^{\lfloor \tfrac{\ell}{6} \rfloor } |||  v_h  |||_{\beta,T}^2.
\end{eqnarray*}
\end{proof}
With this, we can estimate the error between an ideal element corrector $\mathcal{C}_{T}v_h \in W$ and its truncated approximation $\mathcal{C}_{T,\ell} v_h \in W(\,\UN^{\ell}(T)\,)$ which fulfills
$a_{\beta} ( \mathcal{C}_{T,\ell} v_h , w ) = a_{\beta,T} ( v_h , w )$
for all $w\in W(\,\UN^{\ell}(T)\,)$.  Here we recall the statement of Lemma \ref{lemma:local-corrector-estimate}, which guaranteed that if $h \lesssim \kappa^{-1}$, then it holds
\begin{eqnarray*}
|||  \mathcal{C}_{T}v_h - \mathcal{C}_{T,\ell} v_h  |||_{\beta} &\lesssim& \theta_{\beta}^{\ell} \, (1+ \kappa \beta h)  \, |||  v_h  |||_{\beta,T},
\end{eqnarray*}
with $0< \theta_{\beta} <1$ as before.
\begin{proof}[Proof of Lemma \ref{lemma:local-corrector-estimate}]
Due to the Galerkin orthogonality
\begin{align*}
\abeta ( \mathcal{C}_{T} v_h - \mathcal{C}_{T,\ell+5} v_h , w ) = 0  \qquad \mbox{for all } w\in W(\,\UN^{\ell+5}(T)\,)
\end{align*}
we have
\begin{align}
\label{bestapprox-error-Ctell}
||| \mathcal{C}_{T} v_h - \mathcal{C}_{T,\ell+5} v_h |||_{\beta} \lesssim \inf_{ w\in W(\,\UN^{\ell+5}(T)\,) } ||| \mathcal{C}_{T} v_h - w |||_{\beta}.
\end{align}
Let again $w_T:=\mathcal{C}_{T} v_h$ and the cutoff function $\eta \in W^{1,\infty}(\Omega)$ as in \eqref{truncation-function}. In order to estimate the best-approximation error in \eqref{bestapprox-error-Ctell}, we exploit \eqref{proofdecay1-Pinvsupport} and consider the function
\begin{align*}
(1-\eta) w_T + (P_h^{-1} \circ P_h)( \eta w_T)
\in W(\,\UN^{\ell+5}(T)\,).
\end{align*}
We obtain 
\begin{eqnarray*}
\lefteqn{ |||  w_T - (1-\eta) w_T - (P_h^{-1} \circ P_h)(\eta w_T) |||_{\beta} \,\,\, = \,\,\, ||| \eta w_T - (P_h^{-1} \circ P_h)(\eta w_T) |||_{\beta, \UN^{\ell+5}(T)} } \\
&\overset{\eqref{proofdecay1-Pinvsupport}}{\lesssim}&  ||| \eta w_T |||_{\beta, \Omega \setminus \UN^{\ell+2}(T)} + ||| (P_h^{-1} \circ P_h)(\eta w_T) |||_{\beta,\UN^{\ell+5}(T) \setminus \UN^{\ell}(T)} \\
&\lesssim&  ||| w_T |||_{\beta, \Omega \setminus \UN^{\ell+1}(T)} + ||| (P_h^{-1} \circ P_h)(\eta w_T) |||_{\beta,\UN^{\ell+5}(T) \setminus \UN^{\ell}(T)} \\
&\lesssim& ||| w_T |||_{\beta, \Omega \setminus \UN^{\ell+1}(T)} + (1+ \kappa \beta h)  |||  w_T |||_{\beta,\UN^{\ell+6}(T) \setminus \UN^{\ell}(T)} \\
&\lesssim& (1+ \kappa \beta h)  ||| w_T |||_{\beta, \Omega \setminus \UN^{\ell}(T)}
\,\,\,
\overset{\eqref{first-decay-estimate}}{\lesssim} (1+ \kappa \beta h)  \theta_{\beta}^{\ell} \, |||  v_h  |||_{\beta,T}.
\,\,\,
\end{eqnarray*}
where we used $||| (P_h^{-1} \circ P_h)(\eta w_T) |||_{\beta,\UN^{\ell+5}(T) \setminus \UN^{\ell}(T)} \lesssim (1+ \kappa \beta h)  |||  w_T |||_{\beta,\UN^{\ell+6}(T) \setminus \UN^{\ell}(T)}$ which follows as in the proof of Lemma \ref{theorem-decay-ideal-corrector}.
The previous estimate with \eqref{bestapprox-error-Ctell} finishes the proof.
\end{proof} 
We can now combine the local decay results to a global approximation result for the localized corrector $\mathcal{C}_{\ell} v_h := \sum\limits_{T\in \mathcal{T}_h} \mathcal{C}_{T,\ell} v_h$. Here we recall the main result of Theorem \ref{theorem:C-Cell-error} which stated that
\begin{eqnarray*}
||| (\mathcal{C}-\mathcal{C}_{\ell}) v_h |||_{\beta} &\lesssim&  \theta_{\beta}^{\ell} \, (1+ \kappa \beta h)  |||  v_h  |||_{\beta}.
\end{eqnarray*}

\begin{proof}[Proof of Theorem \ref{theorem:C-Cell-error}]
The error can be represented as
\begin{align}
\label{proof-error-splitting-C}
||| (\mathcal{C}-\mathcal{C}_{\ell}) v_h |||_{\beta}^2
&= \sum_{T\in \mathcal{T}_h} \abeta( (\mathcal{C}-\mathcal{C}_{\ell}) v_h , (\mathcal{C}_{T}-\mathcal{C}_{T,\ell}) v_h )
=  \sum_{T\in \mathcal{T}_h} \abeta(e_h ,e_{h,T} ) ,
\end{align}
where we introduced the short notation $e_h:=(\mathcal{C}-\mathcal{C}_{\ell}) v_h$ and $e_{h,T}:=(\mathcal{C}_{T}-\mathcal{C}_{T,\ell}) v_h$. Let $T\in \mathcal{T}_h$ be fixed and let $\eta \in W^{1,\infty}(\Omega)$ denote again the cutoff function from \eqref{truncation-function}. 
Accordingly, we consider the test function $\eta e_{h} - (P_h^{-1} \circ P_h)(\eta e_{h})  \in W$ for which we easily verify that $\mbox{supp}\left( \eta e_{h} - (P_h^{-1} \circ P_h)(\eta e_{h}) \right) \subset \Omega \setminus \UN^{\ell}(T)$. Hence, with $\mathcal{C}_{T,\ell} v_h \in W( \UN^{\ell}(T) )$, we obtain
\begin{align*}
 \abeta( e_{h,T} , \eta e_{h} - (P_h^{-1} \circ P_h)(\eta e_{h}) ) = -  \abeta( \mathcal{C}_{T,\ell} v_h , \eta e_{h} - (P_h^{-1} \circ P_h)(\eta e_{h}) ) = 0.
\end{align*}
This yields
\begin{eqnarray*}
\abeta(e_{h,T}, e_h  ) &=&  \abeta( e_{h,T} , (1-\eta) e_{h}) +  \abeta( e_{h,T} , (P_h^{-1} \circ P_h)(\eta e_{h}) ) \\
&\overset{\eqref{proofdecay1-Pinvsupport}}{\lesssim}& ||| e_{h,T} |||_{\beta} \left(  ||| e_{h} |||_{\beta,\UN^{\ell+4}(T)}  +  ||| (P_h^{-1} \circ P_h)(\eta e_{h}) |||_{\beta, \UN^{\ell+5}(T) \setminus  \UN^{\ell}(T)}  \right) \\
& \lesssim& ||| e_{h,T} |||_{\beta}  \,\, ||| e_{h} |||_{\beta, \UN^{\ell+6}(T)}
\,\,\, \overset{\eqref{local-corrector-estimate}}{\lesssim}\,\,\, \theta_{\beta}^{\ell} \, (1+ \kappa \beta h)  \, |||  v_h  |||_{\beta,T} \,  ||| e_{h} |||_{\beta , \UN^{\ell+6}(T)}.
\end{eqnarray*}
Combining the above estimate with \eqref{proof-error-splitting-C} we obtain
\begin{eqnarray*}
||| (\mathcal{C}-\mathcal{C}_{\ell}) v_h |||_{\beta}^2
&\lesssim& \sum_{T\in \mathcal{T}_h} \theta_{\beta}^{\ell} \, (1+ \kappa \beta h)  \, |||  v_h  |||_{\beta,T} \,  ||| (\mathcal{C}-\mathcal{C}_{\ell}) v_h |||_{\beta , \UN^{\ell+6}(T)} \\
&\lesssim& \theta_{\beta}^{\ell} \, (1+ \kappa \beta h)  |||  v_h  |||_{\beta}  \Big( \sum_{T\in \mathcal{T}_h}  \,  ||| (\mathcal{C}-\mathcal{C}_{\ell}) v_h |||_{\beta , \UN^{\ell+6}(T)}^2 \Big)^{1/2} \\
&\lesssim&  \theta_{\beta}^{\ell} \, (1+ \kappa \beta h)  |||  v_h  |||_{\beta}  \,\, ||| (\mathcal{C}-\mathcal{C}_{\ell}) v_h |||_{\beta},
\end{eqnarray*}
where in the last step that we used that element patches have a bounded overlap in the sense of \eqref{bounded-overlap}. Dividing by $ ||| (\mathcal{C}-\mathcal{C}_{\ell}) v_h |||_{\beta}$ finishes the proof of the first estimate. The second estimate in Theorem \ref{theorem:C-Cell-error} follows readily from the coercivity of $\abeta(\cdot,\cdot)$ on $W$ (Lemma \ref{lemma-abeta-coercive-on-W}) and its universal continuity (Lemma \ref{Lemma_2.1}).
\end{proof}

\end{document}